\documentclass[12pt]{article}
\usepackage[dvips]{graphicx}
\usepackage{color}
\usepackage{latexsym,amsmath,amsfonts,amscd}
\usepackage{amssymb,graphicx,amsthm}
\usepackage{epsfig}
\usepackage{mathrsfs}
\usepackage{indentfirst}
\usepackage{pstricks}
\usepackage{pst-plot}
\usepackage{multirow}
\usepackage{subfigure}
\usepackage{psfrag}
\usepackage{caption}
\usepackage{enumerate}
\usepackage{float}
\usepackage{bbm}
\usepackage{bm}

\allowdisplaybreaks
\newtheorem{theorem}{Theorem}[section]

\newtheorem{lemma}{Lemma}[section]

\newtheorem{remark}{Remark}[section]
\newtheorem{example}{Example}[section]
\numberwithin{equation}{section}
\numberwithin{figure}{section}
\numberwithin{table}{section}
\begin{document}
	
	\baselineskip=2pc
	\vspace*{.30in}
	
	\begin{center}
		{\bf Genuinely multidimensional physical-constraints-preserving
			finite volume schemes for the special relativistic hydrodynamics}
	\end{center}
	
	
	\centerline{Dan Ling\footnote{School of Mathematics and Statistics,
			Xi'an Jiaotong University, Xi'an, Shaanxi 710049, China.
			E-mail: danling@xjtu.edu.cn.} and
		Huazhong Tang\footnote{Nanchang Hangkong University, Jiangxi Province, Nanchang 330063, P.R. China;
			Center for Applied Physics and Technology, HEDPS and LMAM,
			School of Mathematical Sciences, Peking University, Beijing 100871, P.R. China.}}
	
	\vspace{.25in}
	
	\centerline{\bf Abstract}
	
	This paper develops the genuinely multidimensional HLL Riemann solver for the two-dimensional special relativistic hydrodynamic
	equations on Cartesian meshes and studies its physical-constraint-preserving (PCP) property.
	Based on the resulting HLL solver, the first- and high-order accurate PCP finite volume
	schemes are proposed. In the high-order scheme,
	the WENO reconstruction, the third-order accurate strong-stability-preserving time discretizations and the PCP flux limiter are used.
	Several numerical results are given to demonstrate the accuracy, performance and resolution of the shock waves etc. as well as
	the genuinely multi-dimensional wave structures of our PCP finite volume schemes.

	\bigskip
	
	\baselineskip=1.4pc

	\vspace{.05in}

	\vfill
	
	
	\newpage
	
	\baselineskip=2pc

\section{Introduction}\label{intro}

The paper is concerned with the physical-constraints-preserving (PCP) genuinely multidimensional finite volume schemes for
the special relativistic hydrodynamics (RHD),
which plays a major role in astrophysics, plasma physics and nuclear
physics etc.,
where the fluid moves at extremely high velocities
near the speed of light so that the relativistic effects become important.
In the (rest) laboratory frame, the two-dimensional (2D) special RHD equations governing an ideal fluid flow can be written in the divergence form
\begin{equation}\label{eq27}
	\frac{\partial\bm{U}}{\partial t}+\sum\limits_{\ell=1}^2\frac{\partial\bm{F}_\ell(\bm{U})}
	{\partial x_\ell}=0,
\end{equation}
where the conservative vector $\bm{U}$ and the flux $\bm{F}_\ell$ are defined respectively by
\begin{equation}\label{eq28}
	\bm{U}=(D, \bm{m}, E)^T, \quad \bm{F}_\ell=\left(Du_\ell,
	\bm{m}u_\ell+p\bm{e}_\ell,(E+p)u_\ell \right)^T,\ \ \ell=1,2,
\end{equation}
here $D=\rho\gamma$, $\bm{m}=Dh\gamma\bm{u}$, $E=Dh\gamma-p$ and $p$ are the mass, momentum and
total energy relative to the laboratory frame and the gas pressure, respectively, $\bm{u}=(u_1,u_2)$ is the fluid velocity vector,
$\bm{e}_\ell$ is the row vector denoting the $\ell$-th row of the unit matrix of size $2$, $\rho$ is the rest-mass density,
$\gamma=1/\sqrt{1-|\bm{u}|^2}$ is the Lorentz factor,
$|\bm{u}|^2=u_1^2+u_2^2$,
$h=1+e+\frac{p}{\rho}$ is the specific enthalpy, and $e$ is the specific internal energy.
Note that  natural unit (i.e., the speed of light $c = 1$) has been used.
The system \eqref{eq27} should be closed via the equation of state (EOS), which has a general form of $p=p(\rho,e)$.
For simplicity, this paper considers the EOS for the perfect gas, namely
\begin{equation}\label{eq15}
	p=(\Gamma-1)\rho e,
\end{equation}
with the adiabatic index $\Gamma\in(1, 2]$. Such restriction on $\Gamma$ is reasonable
under the compressibility assumptions, and $\Gamma$ is taken as 5/3 for
the mildly relativistic case and 4/3 for
the ultra-relativistic case. In this case, for $i=1,2$,
the Jacobian matrix $\bm{A}_i(\bm{U})=\partial\bm{F}_i/\partial\bm{U}$ of the system \eqref{eq27} has   $4$
real eigenvalues, which are ordered from the smallest to the biggest as follows
\begin{equation*}
	\begin{aligned}
		&\lambda_{i}^{(1)}(\bm{U})=\frac{u_i(1-c_s^2)-c_s\gamma^{-1}\sqrt{1-u_i^2-c_s^2(|\bm{u}|^2-u_i^2)}}
		{1-c_s^2|\bm{u}|^2},\\
		&\lambda_{i}^{(2)}(\bm{U})=\lambda_{i}^{(3)}(\bm{U})=u_i,\\
		&\lambda_{i}^{(4)}(\bm{U})=\frac{u_i(1-c_s^2)+c_s\gamma^{-1}\sqrt{1-u_i^2-c_s^2(|\bm{u}|^2-u_i^2)}}
		{1-c_s^2|\bm{u}|^2},
	\end{aligned}
\end{equation*}
where
$c_s$ is the speed of sound   expressed explicitly by
\begin{equation*}
	c_s=\sqrt{{\Gamma p}/{(\rho h)}},
\end{equation*}
and satisfies
\begin{equation*}
	c_s^2=\frac{\Gamma p}{\rho h}=\frac{\Gamma p}{\rho+\frac{p}{\Gamma-1}+p}=\frac{(\Gamma-1)\Gamma p}{(\Gamma-1)\rho+\Gamma p}<\Gamma-1\le 1=c.
\end{equation*}

Due to the relativistic effect, especially the appearance of the Lorentz factor,
the system \eqref{eq27} becomes more strongly nonlinear than the non-relativistic case,
which leads to that their analytic treatment
is extremely difficult and challenging, except in some special cases, for
instance,  1D Riemann problems or isentropic problems \cite{marti1,pant,lora}.
Because there are no explicit expressions of the primitive variable vector
$\bm{V}=(\rho,\bm{u},p)^T$ and the flux vectors $\bm{F}_i$ in terms of $\bm{U}$,  their values cannot be explicitly recovered from
$\bm{U}$ and need to solve a nonlinear equation, e.g. the following pressure equation
\begin{equation*}
	E+p=D\gamma+\frac{\Gamma}{\Gamma-1}p\gamma^2,
\end{equation*}
with
$\gamma=\big(1-|\bm{m}|^2/(E+p)^2\big)^{-1/2}$.
%
%
Besides those,
there are some physical constraints, such as
$\rho>0, p>0$ and $E\ge D$, as well as that the velocity can not exceed the speed of light,
i.e. $|\bm{u}|<c=1$.
For the RHD problems with large Lorentz factor or low density or low pressure, or strong discontinuity, it is easy to obtain the negative density or pressure, or the larger velocity than the speed of light in numerical computations,
so that the eigenvalues of the Jacobian matrix or the Lorentz factor may become imaginary, leading directly to the ill-posedness of the discrete problem. Consequently,
there is great necessity and significance to develop robust and accurate
PCP numerical schemes  for \eqref{eq27}, whose  solutions can satisfy the intrinsic physical constraints, or belong to the admissible states set  \cite{wu2015}
\begin{equation*}
	\mathcal{G}=\left\{\bm{U}=(D,\bm{m},E)^T\big|~\rho>0,~ p>0,~ |\bm{u}|<1\right\},
\end{equation*}
which is equivalent to
\begin{equation*}
	\mathcal{G}=\left\{\bm{U}=(D,\bm{m},E)^T\big|~D>0,~E-\sqrt{D^2+|\bm{m}|^2}>0\right\}.
\end{equation*}
Based on that, one can prove some useful properties of $\mathcal{G}$, also see \cite{wu2015}.

\begin{lemma}\label{lem1}
	The admissible state set $\mathcal{G}$ is convex.
\end{lemma}
\begin{lemma}\label{lem2}
	If assuming $\bm{U},\bm{U}_1,\bm{U}_2\in \mathcal{G}$, then:
	\begin{enumerate}[(\romannumeral1)]
		\item $\kappa\bm{U}\in\mathcal{G}$ for all $\kappa>0$.
		\item $a_1\bm{U}_1+a_2\bm{U}_2\in\mathcal{G}$ for all $a_1,a_2>0$.
		\item $\alpha\bm{U}-\bm{F}_i(\bm{U}),-\beta\bm{U}+\bm{F}_i(\bm{U})\in \mathcal{G}$ for
		 $\beta\le\lambda_i^{(1)}(\bm{U})$, $\lambda_i^{(4)}(\bm{U})\le \alpha$, and $i=1,2$.
	\end{enumerate}
\end{lemma}
The second and third properties in Lemma \ref{lem2} are  formally different from those in Lemma 2.3 of \cite{wu2015}. 
Their  proof  slightly different from that of the latter  can be found in \ref{appendix}.

The study of numerical methods for the RHDs may date back to
the finite difference code via artificial viscosity  for the spherically symmetric general RHD equations in the Lagrangian coordinate \cite{may1,may2} and  for multi-dimensional RHD equations in the Eulerian coordinate \cite{wilson}.
Since 1990s, the numerical study of the RHD began to attract considerable attention, see some early review articles \cite{marti2,marti2015,font},
and
various modern shock-capturing methods with an exact or approximate Riemann solver have been developed for the RHD equations.
Some examples are the two-shock Riemann solver
\cite{colella}, the Roe Riemann solver \cite{roe}, the HLL Riemann solver \cite{harten}
and the HLLC Riemann solver \cite{toro} and so on.
Some other higher-order accurate methods have also been well studied in the literature, e.g. the ENO (essentially non-oscillatory) and weighted ENO   methods \cite{dolezal,zanna,tchekhovskoy},
the discontinuous Galerkin  methods \cite{RezzollaDG2011,zhao,ZhaoTang-CiCP2017,ZhaoTang-JCP2017}, the adaptive moving mesh methods \cite{he1,he2,Duan-Tang2020RHD2},
and the direct Eulerian GRP schemes \cite{yang2011direct,yang2012direct,wu2014,wu2016,Yuan-Tang2020}.
Recently, based on the properties of $\mathcal{G}$, some PCP schemes were well developed for the special RHDs. They are the high-order accurate PCP finite difference WENO schemes,  discontinuous Galerkin (DG) methods and Lagrangian  finite volume schemes proposed
in \cite{wu2015,wu2017,qin,wu2017a,ling}.
Such works were successfully extended to the special relativistic magnetohydrodynamics (RMHD) in \cite{wu2017m3as,wu2018zamp},
where the importance of divergence-free fields in achieving PCP methods is shown.
Recently, the entropy-stable schemes were also developed for the special RHD or RMHD
equations \cite{dhoriya,Duan-Tang2020RHD,Duan-Tang2020RMHD,Duan-Tang2020RHD2,biswasa,duan2022}.
Most of the above mentioned methods are built on
the 1D Riemann solver, which is used to solve the local 1D Riemann problem
at the cell interface  by picking up flow variations that are orthogonal
to the cell interface and then give the exact or approximate Riemann solution. 
For multi-dimensional problems,
there are still confronted with
enormous risks that the 1D Riemann solvers may lose their computational
efficacy and efficiency to some content, because 
some flow features propagating transverse to the mesh boundary
might be discarded, 
see \cite{vanLeer1993} for more details.
Therefore, it is necessary to capture much more flow features and then incorporate genuinely
multidimensional (physical) information into numerical methods.

In the early 1990s, owing to a shift from the finite-volume approach to the flctuation approach, the state of the art in genuinely multi-dimensional upwind differencing has made dramatic advances.
A early review of multidimensional upwinding may be found in \cite{vanLeer1993}.
For the linearized Euler equations, a genuinely multidimensional
first-order finite volume scheme
was constructed in \cite{abgrall} by computing the exact solution of the Riemann problem
for a linear hyperbolic equation obtained by  linearizing the Euler equation.
Up to now, there have been some further developments on
multidimensional Riemann solvers and corresponding numerical schemes, including
the multidimensional HLL schemes for solving the Euler equations on unstructured triangular meshes
\cite{capdeville1,capdeville2},
the genuinely multidimensional HLL-type scheme
with convective pressure flux split Riemann solver \cite{mandal}, the multidimensional HLLE schemes
for gas dynamics \cite{wendroff,barsara}, the multidimensional MuSCL solver for magnetohydrodynamics \cite{barsara2015},
the multidimensional HLLC schemes \cite{barsara2,barsara3} for hydrodynamics
and magnetohydrodynamics, the well-balanced two-dimensional HLL scheme for shallow water equations \cite{schneider2021},
and the genuinely two-dimensional
scheme for compressible flows in curvilinear coordinates \cite{qu} etc.
For the 2D special RHDs, the existing genuinely multidimensional scheme is the finite volume local evolution Galerkin method, developed in \cite{wu2014b}.
%

This paper  will develop the
genuinely multidimensional PCP finite volume schemes for the RHD equations \eqref{eq27}.
It is organized as follows. Section \ref{multi-hlle} derives
the 2D HLL Riemann solver for \eqref{eq27} and
studies the PCP property of its intermediate state.
Section \ref{scheme} presents the first-order PCP genuinely multidimensional HLL scheme and then extends it to the high-order PCP scheme by using WENO reconstruction and the third-order accurate SSP
time discretizations as well as the   PCP flux limiter.
Section \ref{num} conducts several numerical experiments to demonstrate the accuracy and  performance
of the present schemes. Section \ref{con} concludes the paper with some remarks.

\section{2D HLL Riemann solver}\label{multi-hlle}

This section introduces the genuinely multidimensional HLL Riemann solver \cite{barsara} to the 2D special RHD equations \eqref{eq27} with the EOS \eqref{eq15} on Cartesian meshes and studies its PCP property.
For convenience, here and hereafter, the symbols  $(\bm{F}_1, \bm{F}_2)$, $(u_1,u_2)$, and $(x_1,x_2)$  will be replaced with $(\bm{F}, \bm{G})$, $(u,v)$, and $(x,y)$, respectively,

Consider the 2D Riemann problem of \eqref{eq27}
with the initial data as displayed in Figure \ref{fig1},
denoted by RP$(\bm{U}_{RU},\bm{U}_{LU},\bm{U}_{LD},\bm{U}_{RD})$,
where 
four constant states, $\bm{U}_{RU}$ (right-up), $\bm{U}_{LU}$ (left-up),
$\bm{U}_{LD}$ (left-down) and $\bm{U}_{RD}$ (right-down) are specified
in the first, second, third and fourth quadrants,  respectively, and $O$ is the coordinate origin.
Denote the largest left-, right-, up- and down-moving  speeds of
the elementary waves emerging from the initial discontinuities by $S_L$, $S_R$, $S_U$, $S_D$, respectively.
Specially, $S_L$ and $S_R$ are obtained as
the largest left- and right-moving wave speeds in the 1D HLL solvers \cite{toro} for the two  1D
Riemann problems in $x$-direction, denoted respectively by RP${\{\bm{U}_{LU},\bm{U}_{RU}\}}$
and RP${\{\bm{U}_{LD},\bm{U}_{RD}\}}$,
and $S_D$ and $S_U$ are obtained by considering
two  1D Riemann problems in $y$-direction, denoted respectively by RP${\{\bm{U}_{LU},\bm{U}_{LD}\}}$
and RP${\{\bm{U}_{RU},\bm{U}_{RD}\}}$.
\begin{figure}[H]
	\centering
	\vspace{-2ex}
	\includegraphics[width=3.6in]{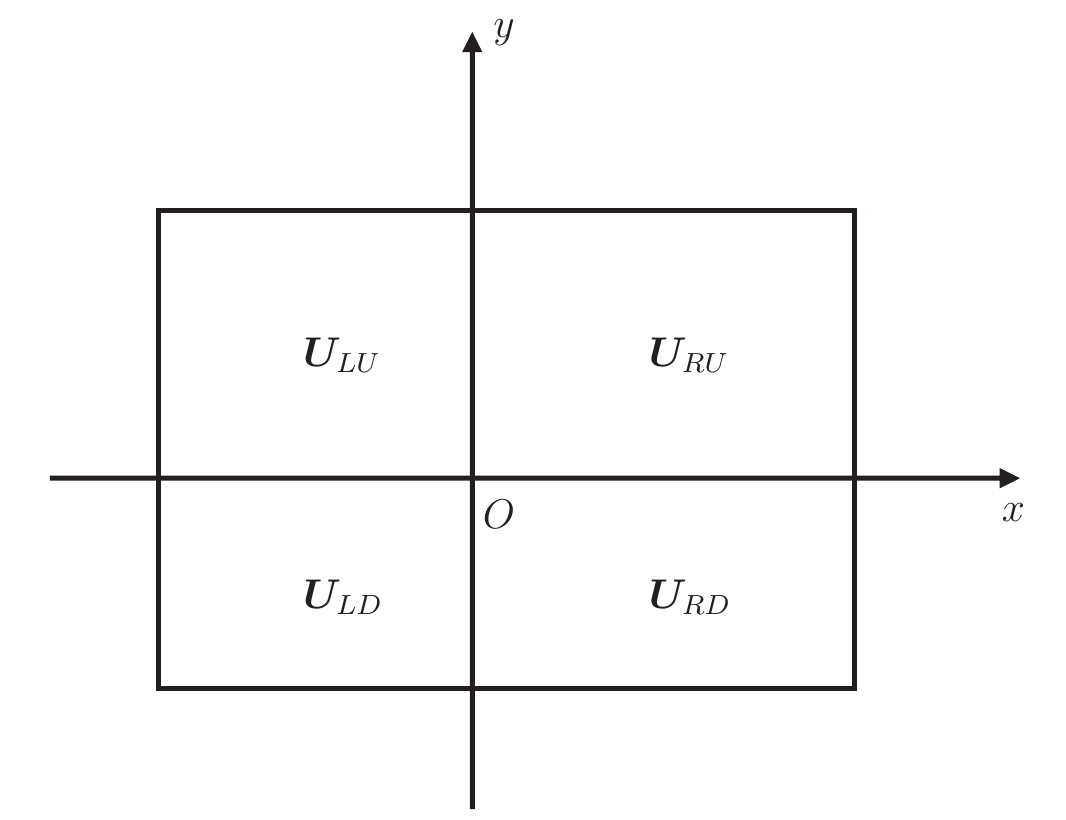}
	\caption{The 
		initial data of
		a 2D Riemann problem at the point $O$.}
	\label{fig1}
\end{figure}

Let us  estimate
the wave speeds $S_L,S_R,S_D,S_U$ used in the 2D HLL solver.
If denoting $\lambda_A^{(1)}(\bm{U}_{LD})$ and
$\lambda_A^{(4)}(\bm{U}_{LD})$ as the smallest and largest eigenvalues of Jacobian matrix
$\partial\bm{F}/\partial\bm{U}(\bm{U}_{LD})$ respectively,
 $\lambda_B^{(1)}(\bm{U}_{LD})$ and
$\lambda_B^{(4)}(\bm{U}_{LD})$ as the smallest and largest eigenvalues of Jacobian matrix
$\partial\bm{G}/\partial\bm{U}(\bm{U}_{LD})$ respectively,
 and making similar definitions
at the states $\bm{U}_{RU}$, $\bm{U}_{LU}$, and $\bm{U}_{RD}$
in the $x$- and  $y$-directions,
then the wave speeds $S_L, S_R, S_D$ and $S_U$ are respectively given by
\begin{equation}\label{eq2}
	\begin{aligned}
		&S_L=\alpha\min\big(\lambda_A^{(1)}(\bm{U}_{LD}), \lambda_A^{(1)}(\bm{U}_{RD}), \lambda_A^{(1)}(\bm{U}_{LU}),
		\lambda_A^{(1)}(\bm{U}_{RU})\big),\\
		&S_R=\alpha\max\big(\lambda_A^{(4)}(\bm{U}_{LD}), \lambda_A^{(4)}(\bm{U}_{RD}), \lambda_A^{(4)}(\bm{U}_{LU}),
		\lambda_A^{(4)}(\bm{U}_{RU})\big),\\
		&S_D=\alpha\min\big(\lambda_{B}^{(1)}(\bm{U}_{LD}), \lambda_{B}^{(1)}(\bm{U}_{RD}), \lambda_{B}^{(1)}(\bm{U}_{LU}),
		\lambda_{B}^{(1)}(\bm{U}_{RU})\big),\\
		&S_U=\alpha\max\big(\lambda_{B}^{(4)}(\bm{U}_{LD}), \lambda_{B}^{(4)}(\bm{U}_{RD}), \lambda_{B}^{(4)}(\bm{U}_{LU}),
		\lambda_{B}^{(4)}(\bm{U}_{RU})\big),
	\end{aligned}
\end{equation}
where $\alpha\ge1$ will be determined later.
Noting that the condition $\alpha>1$ is  used to preserve
the PCP property in Theorem \ref{thm1}. In practice,
$\alpha$ is sufficiently taken as one if the PCP property is not
necessary.
There exist several different ways to define the wave speeds in the 1D HLL-type Riemann solvers, see e.g. \cite{batten,einfeldt,davis}.
In the following,   we only discuss the case with $S_L<0<S_R$ and $S_D<0<S_U$, because in the case of that $S_L$ and $S_R$ (or $S_D$ and $S_U$) have the same sign, our genuinely multidimensional Riemann solver will degenerate to the 1D Riemann solver.

Similar to the 1D HLL Riemann solvers,
one has to derive the intermediate state
$\bm{U}^{\ast}$ in the approximate solution of
the above 2D Riemann problem and corresponding fluxes $\bm{F}^\ast$ and $\bm{G}^\ast$.
For any given time $T>0$, choose
a three-dimensional cuboid $\mathbb V_{LRDU}(0,T)$ in the $(x,y,t)$ space as follows:
its top and bottom   are rectangles with
four vertices
$$(TS_L,TS_D,0),~(TS_R,TS_D,0),~(TS_L,TS_U,0),~(TS_R,TS_U,0),$$
and
$$(TS_L,TS_D,T),~(TS_R,TS_D,T),~(TS_L,TS_U,T),~(TS_R,TS_U,T),$$
respectively. Integrating \eqref{eq27} over $\mathbb V_{LRDU}(0,T)$ gives
\begin{align}\nonumber
	\bm{U}^{\ast}\mathscr{A}&-\int_{TS_L}^{TS_R}\int_{TS_D}^{TS_U}
	\bm{U}(x,y,0)dydx
	\\ \nonumber
	&+\int_{0}^{T}\int_{TS_D}^{TS_U}\bm{F}(\bm{U}(TS_R,y,t))dydt
	-\int_{0}^{T}\int_{TS_D}^{TS_U}\bm{F}(\bm{U}(TS_L,y,t))dydt
	\\
	&+\int_{0}^{T}\int_{TS_L}^{TS_R}\bm{G}(\bm{U}(x,TS_U,t))dxdt
	-\int_{0}^{T}\int_{TS_L}^{TS_R}\bm{G}(\bm{U}(x,TS_D,t))dxdt=0,
	\label{eq3}
\end{align}
where
$$\bm{U}^{\ast}=\frac{1}{\mathscr{A}}\int_{TS_D}^{TS_U}\int_{TS_L}^{TS_R}
\bm{U}(x,y,T)dxdy,\ \
\mathscr{A}=T^2(S_R-S_L)(S_U-S_D).$$
From \eqref{eq3}, one has
\begin{align}\nonumber
	\bm{U}^{\ast}&=\frac{1}{\mathscr{A}}\int_{TS_D}^{TS_U}\int_{TS_L}^{TS_R}
	\bm{U}(x,y,T)dxdy
	=\frac{S_RS_U\bm{U}_{RU}+S_LS_D\bm{U}_{LD}-S_RS_D\bm{U}_{RD}
		-S_LS_U\bm{U}_{LU}}{(S_R-S_L)(S_U-S_D)}
	\\ 
	& -\frac{S_U(\bm{F}_{RU}-\bm{F}_{LU})-S_D(\bm{F}_{RD}-\bm{F}_{LD})}{(S_R-S_L)(S_U-S_D)}
	-\frac{S_R(\bm{G}_{RU}-\bm{G}_{RD})-S_L(\bm{G}_{LU}-\bm{G}_{LD})}{(S_R-S_L)(S_U-S_D)},
	\label{eq43}
\end{align}
where $\bm{F}_{LU}=\bm{F}(\bm{U}_{LU})$ and $\bm{G}_{LU}=\bm{G}(\bm{U}_{LU})$, and
$\bm{F}_{RU},\bm{F}_{LD},\bm{F}_{RD},\bm{G}_{RU},\bm{G}_{LD},\bm{G}_{RD}$
are similarly defined.
It is clear that the calculation of the intermediate  state $\bm{U}^{\ast}$ depends on four states
$\bm{U}_{LD}$, $\bm{U}_{LU}$, $\bm{U}_{RD}$ and $\bm{U}_{RU}$,
in other words, $\bm{U}^{\ast}$ contains genuinely multidimensional information.
In particular, if
\begin{equation*}
	\bm{U}_{LD}=\bm{U}_{LU},~~~\bm{U}_{RD}=\bm{U}_{RU},
\end{equation*}
then
\begin{equation*}
	\bm{F}_{LD}=\bm{F}_{LU},~~~\bm{F}_{RD}=\bm{F}_{RU},~~~\bm{G}_{LD}=\bm{G}_{LU},
	~~~\bm{G}_{RD}=\bm{G}_{RU},
\end{equation*}
and  
\begin{equation}\label{eq38}
	\bm{U}^{\ast}=\frac{S_R\bm{U}_{RD}-S_L\bm{U}_{LU}+\bm{F}_{LD}-\bm{F}_{RD}}{S_R-S_L},
\end{equation}
which is indeed the intermediate state in the 1D HLL Riemann solver in \cite{toro}.

Let us turn to obtain the interface fluxes $\bm{F}^{\ast}$ and $\bm{G}^{\ast}$ for the multidimensional Riemann solver (in the case of $S_L<0<S_R$ and $S_D<0<S_U$).
Integrating respectively the system
\eqref{eq27} over the left portion and the top portion
(or the right and bottom portions) of the control
volume $\mathbb V_{LRDU}(0,T)$ yields
\begin{align} \nonumber
	\int_{TS_L}^{0}\int_{TS_D}^{TS_U}\bm{U}(x,y,T)dydx&=\int_{TS_L}^{0}
	\int_{TS_D}^{TS_U}\bm{U}(x,y,0)dydx
	\\  \nonumber
	~~~~~-\int_{0}^{T}&\int_{TS_D}^{TS_U}\big(\bm{F}(\bm{U}(0,y,t))-\bm{F}(\bm{U}(TS_L,y,t))\big)dydt\\
	-\int_{0}^{T}&\int_{TS_L}^{0}\big(\bm{G}(\bm{U}(x,TS_U,t))-\bm{G}(\bm{U}(x,TS_D,t))\big)dxdt,
	\label{eq11}
	\\ \nonumber
	\int_{TS_L}^{TS_R}\int_{0}^{TS_U}\bm{U}(x,y,T)dydx&=
	\int_{TS_L}^{TS_R}\int_{0}^{TS_U}\bm{U}(x,y,0)dydx
	\\ \nonumber
	-\int_{0}^{T}\int_{0}^{TS_U}&\big(\bm{F}(\bm{U}(TS_R,y,t))
	-\bm{F}(\bm{U}(TS_L,y,t))\big)dydt\\
	-\int_{0}^{T}\int_{TS_L}^{TS_R}&\big(\bm{G}(\bm{U}(x,TS_U,t))-\bm{G}(\bm{U}(x,0,t))\big)dxdt.
	\label{eq12}
\end{align}
\begin{figure}[H]
	\centering
	\includegraphics[width=5.6in,height=2.0in]{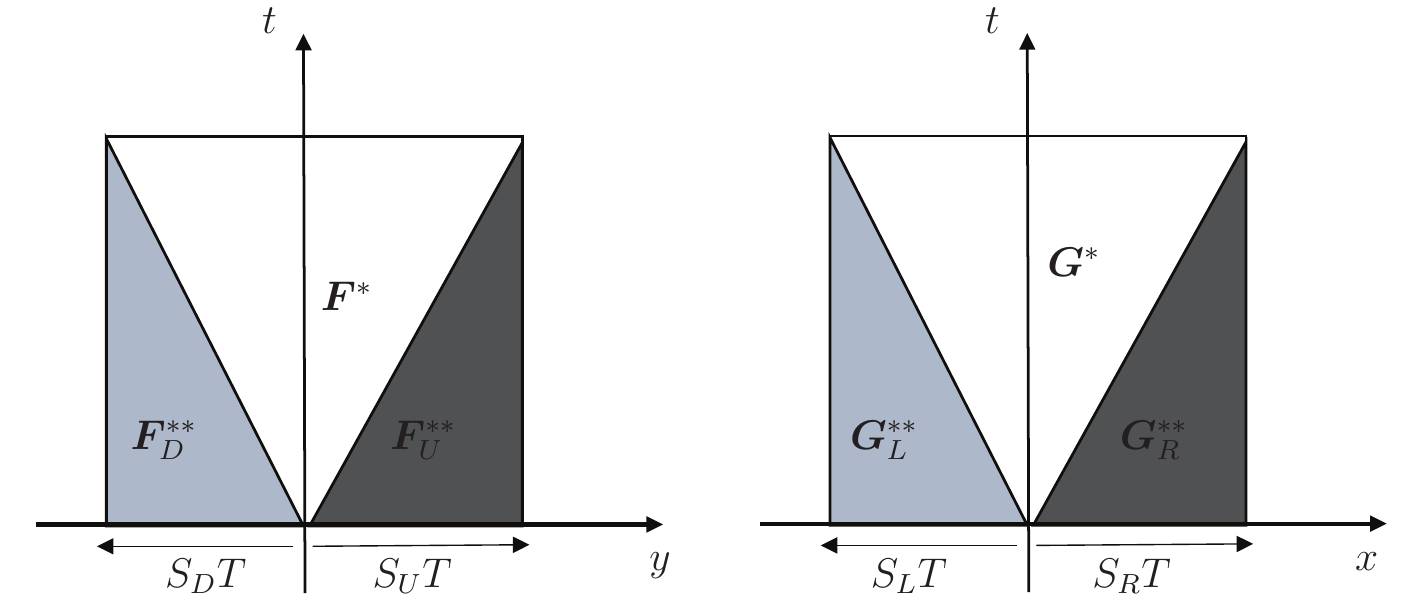}
	\caption{Fluxes along the faces $x=0$ (left) and $y=0$ (right) respectively
		consisting of several different portions.}
	\label{facefluxes}
\end{figure}
As shown in the Figure \ref{facefluxes},
the fluxes along the faces $x=0$ (left) and $y=0$ (right) respectively
are consisting of several different portions,
so that the integrals in \eqref{eq11} and \eqref{eq12}
on the $x=0$ and $y=0$ faces respectively read as
\begin{align}
	\int_{0}^{T}\int_{TS_D}^{TS_U}\bm{F}(\bm{U}(0,y,t))dydt&=\frac{T^2}{2}\bigg(S_U\bm{F}_U^{\ast\ast}
	-S_D\bm{F}_D^{\ast\ast}+(S_U-S_D)\bm{F}^\ast\bigg),\label{add14}\\
	\int_{0}^{T}\int_{TS_L}^{TS_R}\bm{G}(\bm{U}(x,0,t))dxdt&=\frac{T^2}{2}\bigg(S_R\bm{G}_R^{\ast\ast}
	-S_L\bm{G}_L^{\ast\ast}+(S_R-S_L)\bm{G}^\ast\bigg),\label{add15}
\end{align}
where $\bm{F}_U^{\ast\ast},\bm{F}_D^{\ast\ast},\bm{G}_R^{\ast\ast}$,
and $\bm{G}_L^{\ast\ast}$ are corresponding 1D HLL fluxes in the
1D HLL Riemann solver and have the specific forms of
\vspace{-1ex}\begin{align}
\widehat{\bm{F}}^{\text{\tiny 1d-HLL}}(\bm{U}_{LU},\bm{U}_{RU})
:=	\bm{F}_U^{\ast\ast}&=\frac{1}{S_R-S_L}\bigg(S_R\bm{F}_{LU}-S_L\bm{F}_{RU}
	+S_LS_R(\bm{U}_{RU}-\bm{U}_{LU})\bigg),\label{add16}\\
\widehat{\bm{F}}^{\text{\tiny 1d-HLL}}(\bm{U}_{LD},\bm{U}_{RD})
:=	\bm{F}_D^{\ast\ast}&=\frac{1}{S_R-S_L}\bigg(S_R\bm{F}_{LD}-S_L\bm{F}_{RD}
	+S_LS_R(\bm{U}_{RD}-\bm{U}_{LD})\bigg),\label{add17}
\\
\widehat{\bm{G}}^{\text{\tiny 1d-HLL}}(\bm{U}_{RD},\bm{U}_{RU})
:=
\bm{G}_R^{\ast\ast}&=\frac{1}{S_U-S_D}\bigg(S_U\bm{G}_{RD}-S_D\bm{G}_{RU}
	+S_DS_U(\bm{U}_{RU}-\bm{U}_{RD})\bigg),\label{add18}\\
\widehat{\bm{G}}^{\text{\tiny 1d-HLL}}(\bm{U}_{LD},\bm{U}_{LU})
:= \bm{G}_L^{\ast\ast}&=\frac{1}{S_U-S_D}\bigg(S_U\bm{G}_{LD}-S_D\bm{G}_{LU}
	+S_DS_U(\bm{U}_{LU}-\bm{U}_{LD})\bigg).\label{add19}
\end{align}
Combining $\eqref{eq43}$ with the relations
in \eqref{eq11}-\eqref{add19} yields
\vspace{-1ex}
\begin{align}\label{eq34}
	\bm{F}^\ast&=\frac{1}{S_U-S_D}\bigg(S_U\bm{F}_U^{\ast\ast}-S_D\bm{F}_D^{\ast\ast}-
	\frac{2S_LS_R}{S_R-S_L}(\bm{G}_{RU}-\bm{G}_{RD}-\bm{G}_{LU}+\bm{G}_{LD})\bigg),
	\\
	\label{eq35}
\bm{G}^\ast&=\frac{1}{S_R-S_L}\bigg(S_R\bm{G}_R^{\ast\ast}-S_L\bm{G}_L^{\ast\ast}-
	\frac{2S_DS_U}{S_U-S_D}(\bm{F}_{RU}-\bm{F}_{RD}-\bm{F}_{LU}
+\bm{F}_{LD})\bigg),
\end{align}
which are the 2D HLL fluxes in the case of $S_L<0<S_R$ and $S_D<0<S_U$.
For all other cases of $S_L,S_R,S_D,S_U$ with certain signs (either positive or negative), 
we can still similarly evaluate the above integrals on the $x=0$ and $y=0$ faces, and  get the fluxes $\bm{F}^\ast$ and $\bm{G}^\ast$ in the multidimensional Riemann solver
by \eqref{eq11} and \eqref{eq12}.
Hence 
if setting \cite{toro}
\begin{equation}\label{eq29}
	S_L^{-}=\min(S_L,0),~~S_R^{+}=\max(S_R,0),~~S_D^{-}=\min(S_D,0),~~S_U^{+}=\max(S_U,0),
\end{equation}
then one gets the 2D HLL fluxes $\bm{F}^\ast$ and $\bm{G}^\ast$ in the multidimensional Riemann solver for all situations as follows
\begin{align}\label{eq41}
&\begin{aligned}
		\widehat{\bm{F}}^{\text{\tiny 2d-HLL}}&(\bm{U}_{LD},\bm{U}_{LU},\bm{U}_{RD},\bm{U}_{RU}):=\bm{F}^{\ast}
\\		&=\frac{1}{S_U^{+}-S_D^{-}}\bigg(S_U^{+}\bm{F}_U^{\ast\ast}-S_D^{-}\bm{F}_D^{\ast\ast}-
		\frac{2S_L^{-}S_R^{+}}{S_R^{+}-S_L^{-}}(\bm{G}_{RU}-\bm{G}_{RD}-\bm{G}_{LU}+\bm{G}_{LD})\bigg),
	\end{aligned}
	\\
	\label{eq42}
&\begin{aligned}
		\widehat{\bm{G}}^{\text{\tiny 2d-HLL}}&(\bm{U}_{LD},\bm{U}_{LU},\bm{U}_{RD},\bm{U}_{RU}):=\bm{G}^{\ast}
\\ &=\frac{1}{S_R^{+}-S_L^{-}}\bigg(S_R^{+}\bm{G}_R^{\ast\ast}-S_L^{-}\bm{G}_L^{\ast\ast}-
		\frac{2S_D^{-}S_U^{+}}{S_U^{+}-S_D^{-}}(\bm{F}_{RU}-\bm{F}_{RD}-\bm{F}_{LU}+\bm{F}_{LD})\bigg).
	\end{aligned}
\end{align}

Next, let us study the PCP property of the multidimensional
HLL Riemann solver, which means that the intermediate state $\bm{U}^{\ast}$
in the multidimensional Riemann solver is admissible.
Only the case  of $S_L<0<S_R,S_D<0<S_U$ needs to be discussed here,
since all other situations (except for the non-trivial case of $S_L<0<S_R,S_D<0<S_U$)
produce the 1D intermediate state which can be easily proved to be PCP according
to \cite{ling}.
The PCP property of the multidimensional HLL Riemann solver with \eqref{eq2} can be obtained as follows.
\begin{theorem}\label{thm1}
If
$\bm{U}_{RU},\bm{U}_{LU},\bm{U}_{LD},\bm{U}_{RD}\in \mathcal G$,
and the wave speeds $S_L,S_R,S_D,S_U$ are taken as \eqref{eq2} with $\alpha=2$,
	then the intermediate state $\bm{U}^{\ast}$ in \eqref{eq43} obtained for the multidimensional
	HLL Riemann solver is admissible, i.e.
	\begin{equation*}
		D^{\ast}>0,~~E^{\ast}>0,~~(E^{\ast})^2-(D^{\ast})^2-|\bm{m}^{\ast}|^2>0.
	\end{equation*}
\end{theorem}
\begin{proof}
	Assume that $S_L<0<S_R$ and $S_D<0<S_U$.
	Rewrite $\bm{U}^{\ast}$ in \eqref{eq43}
	as
	\begin{equation*}
		\bm{U}^{\ast}=\frac{1}{\mathscr{B}}\bigg(S_LS_D\bm{H}_{LD}
		-S_RS_D\bm{H}_{RD}-S_LS_U\bm{H}_{LU}+S_RS_U\bm{H}_{RU}\bigg),
	\end{equation*}
	with $\mathscr{B}=(S_R-S_L)(S_U-S_D)$ and
	\begin{equation*}
		\begin{aligned}
			&\bm{H}_{LD}=\bm{U}_{LD}-\frac{1}{S_L}\bm{F}_{LD}-\frac{1}{S_D}\bm{G}_{LD},~~
			\bm{H}_{RD}=\bm{U}_{RD}-\frac{1}{S_R}\bm{F}_{RD}-\frac{1}{S_D}\bm{G}_{RD},\\
			&\bm{H}_{LU}=\bm{U}_{LU}-\frac{1}{S_L}\bm{F}_{LU}-\frac{1}{S_U}\bm{G}_{LU},~~
			\bm{H}_{RU}=\bm{U}_{RU}-\frac{1}{S_R}\bm{F}_{RU}-\frac{1}{S_U}\bm{G}_{RU}.
		\end{aligned}
	\end{equation*}
	It means that $\bm{U}^\ast$ is a convex combination
	of $\bm{H}_{LD},\bm{H}_{RD}$, $\bm{H}_{LU},\bm{H}_{RU}$.
	Due to Lemma \ref{lem1}, 
	it is sufficient to check whether those
	$\bm{H}$-terms are in the admissible set $\mathcal{G}$.
	As an example, consider the term $\bm{H}_{LD}$,
	which can be decomposed  into two parts as follows
	\begin{equation*}
		\bm{H}_{LD}=\bm{U}_{LD}-\frac{1}{S_L}\bm{F}_{LD}-\frac{1}{S_D}\bm{G}_{LD}
		=\frac{1}{2}\bigg(\bm{U}_{LD}-\frac{2}{S_L}\bm{F}_{LD}\bigg)
		+\frac{1}{2}\bigg(\bm{U}_{LD}-\frac{2}{S_D}\bm{G}_{LD}\bigg).
	\end{equation*}
	The properties (\romannumeral2) and (\romannumeral3)
	in Lemma \ref{lem2} show
	the admissibility of
	$\bm{U}_{LD}-\frac{2}{S_L}\bm{F}_{LD}$ and $\bm{U}_{LD}-\frac{2}{S_D}\bm{G}_{LD}$ so is $\bm{H}_{LD}$.
	Similarly, one can show that other $\bm{H}$-terms are also admissible.
	The proof is completed.
\end{proof}

\section{Numerical schemes}\label{scheme}
This section presents the first- and high-order accurate PCP finite volume schemes with the above multidimensional HLL Riemann solver
for the special RHD equations \eqref{eq27}.

\subsection{First-order PCP scheme}
Consider 2D Cartesian mesh in $(x,y)$ space
$\{(x_{i+\frac{1}{2}},y_{j+\frac{1}{2}})|
\Delta x_i=x_{i+\frac{1}{2}}-x_{i-\frac{1}{2}},
\Delta y_j=y_{j+\frac{1}{2}}-y_{j-\frac{1}{2}},
i,j\in\mathbb Z\}$
and define the rectangular cell
$I_{ij} =[x_{i-\frac{1}{2}},x_{i+\frac{1}{2}}]\times[y_{j-\frac{1}{2}},
y_{j+\frac{1}{2}}]$.
Assume that $t_{n+1}=t_n+\Delta t^n$
with $t_0=0$, where $\Delta t^n$ is the time step-size at $t=t_n$ and to be determined
later, $n=0,1,2,\cdots$.
The numerical solutions
 at $t_n$ are reconstructed as a piecewise constant function
 by using the (approximate)
cell average values $\overline{\bm{U}}_{ij}^{n}$  of $\bm{U}$ at $t_n$  over the cell $I_{ij}$.

Integrating the RHD system \eqref{eq27} over the cell $I_{ij}\times [t_n,t_{n+1})$ 
obtains the following finite volume scheme
\begin{equation}\label{eq20}
	\overline{\bm{U}}_{ij}^{n+1}=\overline{\bm{U}}_{ij}^{n}-\frac{\Delta t^n}{\Delta x_i}
	\big(\widehat{\bm{F}}_{i+\frac{1}{2},j}-\widehat{\bm{F}}_{i-\frac{1}{2},j}\big)-\frac{\Delta t^n}{\Delta y_j}
	\big(\widehat{\bm{G}}_{i,j+\frac{1}{2}}-\widehat{\bm{G}}_{i,j-\frac{1}{2}}\big),
\end{equation}
where
$\widehat{\bm{F}}_{i+\frac{1}{2},j}$ and $\widehat{\bm{G}}_{i,j+\frac{1}{2}}$ are the numerical fluxes approximating the
flux integrals
\begin{equation}\label{eq20bbb}
\begin{aligned}
&I^y_{i+\frac12,j}:=\frac{1}{\Delta t^n\Delta y_j}\int_{t_n}^{t_{n+1}}
\int_{y_{j-\frac{1}{2}}}^{y_{j+\frac{1}{2}}}\bm{F}(\bm{U}(x_{i+\frac{1}{2}},y,t_n))dydt,\\ 
&I^x_{i,j+\frac12}:=\frac{1}{\Delta t^n\Delta x_i}\int_{t_n}^{t_{n+1}}\int_{x_{i-\frac{1}{2}}}^{x_{i+\frac{1}{2}}}
\bm{G}(\bm{U}(x,y_{j+\frac{1}{2}},t_n))dxdt,
\end{aligned}\end{equation}
respectively, and $\Delta t^n$ is determined by the CFL type condition
\begin{equation}\label{eq39}
	\Delta t^n\le\sigma\min\limits_{i,j}\left\{\frac{\Delta x_i}
	{\max\big(|\lambda_A^{(1)}(\overline{\bm{U}}_{ij}^n)|,
		|\lambda_A^{(4)}(\overline{\bm{U}}_{ij}^n)|\big)},\frac{\Delta y_j}
	{\max\big(|\lambda_B^{(1)}(\overline{\bm{U}}_{ij}^n)|,|\lambda_B^{(4)}(\overline{\bm{U}}_{ij}^n)|\big)}\right\},
\end{equation}
here the CFL number $\sigma\le\frac{1}{2}$.
Noting that \eqref{eq20} can also be derived by
integrating the RHD system \eqref{eq27} over the cell $I_{ij}\times [t_n, t_{n+1})$.
%
Following \cite{barsara}, see Figure \ref{facefluxes}, the numerical fluxes $\widehat{\bm{F}}$ and $\widehat{\bm{G}}$ are contributed by the 1D Riemann solver at the center of the cell edge
and the 2D Riemann solver at two endpoints of the cell edge.
For example,  the numerical flux $\widehat{\bm{F}}_{i+\frac{1}{2},j}$ consists
of three parts: $\bm{F}_{i+\frac{1}{2},j}^{\ast\ast}$ computed from the 1D HLL Riemann
solver at the point $(x_{i+\frac12}, y_j)$
and $\bm{F}_{i+\frac{1}{2},j-\frac{1}{2}}^{\ast},\bm{F}_{i+\frac{1}{2},j+\frac{1}{2}}^{\ast}$ computed from the 2D
HLL Riemann solvers at $(x_{i+\frac{1}{2}},y_{j-\frac{1}{2}})$ and
$(x_{i+\frac{1}{2}},y_{j+\frac{1}{2}})$, respectively, where
\begin{equation}\label{eq44}
\begin{aligned}
&\bm{F}^{\ast\ast}_{i+\frac{1}{2},j}
=\widehat{\bm{F}}^{\text{\tiny 1d-HLL}}
(\bm{U}_{i+\frac{1}{2},j}^{L},\bm{U}_{i+\frac{1}{2},j}^{R}),\\
&\bm{F}^{\ast}_{i+\frac{1}{2},j\pm\frac{1}{2}}
=\widehat{\bm{F}}^{\text{\tiny 2d-HLL}}
\big(\bm{U}^{LD}_{i+\frac{1}{2},j\pm\frac{1}{2}},
\bm{U}^{LU}_{i+\frac{1}{2},j\pm\frac{1}{2}},\bm{U}^{RD}_{i+\frac{1}{2},j\pm\frac{1}{2}},
\bm{U}^{RU}_{i+\frac{1}{2},j\pm\frac{1}{2}}\big),
\end{aligned}
\end{equation}
with $\bm{U}_{i+\frac{1}{2},j}^{L},\bm{U}_{i+\frac{1}{2},j}^{R}$ being the left and right
limited approximations of $\bm{U}$ at the center of the edge $x=x_{i+\frac{1}{2}}$, and
$\bm{U}^{LD}_{i+\frac{1}{2},j\pm\frac{1}{2}},
\bm{U}^{RD}_{i+\frac{1}{2},j\pm\frac{1}{2}},\bm{U}^{LU}_{i+\frac{1}{2},j\pm\frac{1}{2}},
\bm{U}^{RU}_{i+\frac{1}{2},j\pm\frac{1}{2}}$ being the left-down, right-down, left-up and right-up
limited approximations of $\bm{U}$ at the node $(x_{i+\frac{1}{2}},y_{j\pm\frac{1}{2}})$, defined respectively by
\begin{equation}\label{add21}
	\begin{aligned}
		&\bm{U}_{i+\frac{1}{2},j}^{L}=\overline{\bm{U}}_{ij},
		~~~~~~~~\bm{U}_{i+\frac{1}{2},j}^{R}=\overline{\bm{U}}_{i+1,j},
		~~~~~\bm{U}^{LD}_{i+\frac{1}{2},j+\frac{1}{2}}=\overline{\bm{U}}_{ij},\\
		&\bm{U}^{RD}_{i+\frac{1}{2},j+\frac{1}{2}}=\overline{\bm{U}}_{i+1,j},
		~\bm{U}^{LU}_{i+\frac{1}{2},j+\frac{1}{2}}=\overline{\bm{U}}_{i,j+1},
		~~\bm{U}^{RU}_{i+\frac{1}{2},j+\frac{1}{2}}=\overline{\bm{U}}_{i+1,j+1}.
	\end{aligned}
\end{equation}
In practice, for the case of  $S_D<0<S_U$, see the left schematics of -eps-converted-to.pdf \ref{facefluxes}
and \ref{inter}, the numerical flux $\widehat{\bm{F}}_{i+\frac{1}{2},j}$ may be derived by
\begin{align*}
&
I^y_{i+\frac12,j}
=\frac{1}{\Delta t^n\Delta y_j}\int_{t_n}^{t_{n+1}}
\left\{
\int_{y_{j-\frac12}}^{y^U_{j-\frac12}}
+\int_{y^U_{j-\frac12}}^{y^D_{j+\frac12}}
+\int_{y^D_{j+\frac12}}^{y_{j+\frac12}}
\right\}
\bm{F}(\bm{U}(x_{i+\frac12},y,t_n)~dydt
\\
&\approx
\frac{S_{U,i+\frac{1}{2},j-\frac{1}{2}} \Delta t^n}{2\Delta y_j}
{\bm{F}}^*_{i+\frac{1}{2},j-\frac{1}{2}}
+\Big(1-\frac{(S_{U,i+\frac{1}{2},j-\frac{1}{2}}-S_{D,i+\frac{1}{2},j+\frac{1}{2}}) \Delta t^n}{2\Delta y_j}\Big)
{\bm{F}}^{**}_{i+\frac{1}{2},j}
-\frac{S_{D,i+\frac{1}{2},j+\frac{1}{2}} \Delta t^n}{2\Delta y_j}
{\bm{F}}^*_{i+\frac{1}{2},j+\frac{1}{2}},
\end{align*}
under the assumption of
$\frac{(S_{U,i+\frac{1}{2},j-\frac{1}{2}}-S_{D,i+\frac{1}{2},j+\frac{1}{2}}) \Delta t^n}{2\Delta y_j}\leq 1$, where
$y^U_{j-\frac12}:=y_{j-\frac12}+S_{U,i+\frac12,j-\frac12}\Delta t^n$,
$y^D_{j+\frac12}:=y_{j+\frac12}+S_{D,i+\frac12,j+\frac12}\Delta t^n$.
Similarly, the numerical flux $\widehat{\bm{G}}_{i,j+\frac{1}{2}}$ consists of $\bm{G}_{i,j+\frac{1}{2}}^{\ast\ast}$ and
$\bm{G}_{i\pm\frac{1}{2},j+\frac{1}{2}}^{\ast}$, which are computed from the 1D HLL Riemann
solver at $(x_i, y_{j+\frac12})$ and the 2D HLL Riemann solvers
at $(x_{i-\frac{1}{2}},y_{j+\frac{1}{2}})$ and
$(x_{i+\frac{1}{2}},y_{j+\frac{1}{2}})$ respectively, and may also be derived by approximating the second flux integral 
in \eqref{eq20bbb} under the assumption of\\
$\frac{(S_{R,i-\frac{1}{2},j+\frac{1}{2}}-S_{L,i+\frac{1}{2},j+\frac{1}{2}}) \Delta t^n}{2\Delta x_i}\leq 1$.
With the help of the definitions of $S_L^{-},S_R^{+},S_D^{-},S_U^{+}$  in \eqref{eq29}, 
the numerical fluxes $\widehat{\bm{F}}$ and $\widehat{\bm{G}}$ in \eqref{eq20} are finally given by
\begin{align}
	\widehat{\bm{F}}_{i+\frac{1}{2},j}&=\frac{\Delta t^n}{2\Delta y_j}
S_{U,i+\frac{1}{2},j-\frac{1}{2}}^{+}\bm{F}^{\ast}_{i+\frac{1}{2},j-\frac{1}{2}}
	-\frac{\Delta t^n}{2\Delta y_j}S_{D,i+\frac{1}{2},j+\frac{1}{2}}^{-}\bm{F}^{\ast}_{i+\frac{1}{2},j+\frac{1}{2}}\notag\\
	&~~~+\left(1-\frac{\Delta t^n}{2\Delta y_j}
	(S_{U,i+\frac{1}{2},j-\frac{1}{2}}^{+}-S_{D,i+\frac{1}{2},j+\frac{1}{2}}^{-})\right)
	\bm{F}_{i+\frac{1}{2},j}^{\ast\ast},\label{add23}\\
	\widehat{\bm{G}}_{i,j+\frac{1}{2}}&=\frac{\Delta t^n}{2\Delta x_i}
S_{R,i-\frac{1}{2},j+\frac{1}{2}}^{+}\bm{G}^{\ast}_{i-\frac{1}{2},j+\frac{1}{2}}
	-\frac{\Delta t^n}{2\Delta x_i}S_{L,i+\frac{1}{2},j+\frac{1}{2}}^{-}\bm{G}^{\ast}_{i+\frac{1}{2},j+\frac{1}{2}}\notag\\
	&~~~+\left(1-\frac{\Delta t^n}{2\Delta x_i}
	(S_{R,i-\frac{1}{2},j+\frac{1}{2}}^{+}-S_{L,i+\frac{1}{2},j+\frac{1}{2}}^{-})\right)
	\bm{G}_{i,j+\frac{1}{2}}^{\ast\ast},\label{add24}
\end{align}  
under the time stepsize constraint 
\begin{equation}\label{dt}
\Delta t^n\le \min\limits_{i,j}\left(\frac{2\Delta x_i}{S^+_{R,i-\frac12,j\pm \frac12}-S^-_{L,i+\frac12,j\pm \frac12}},
\frac{2\Delta y_j}{S^+_{U,i\pm\frac12,j-\frac12}-S^-_{D,i\pm\frac12,j+\frac12}}
\right),
\end{equation}
which is weaker than  \eqref{eq39}.
\begin{figure}[H]
	\begin{minipage}{0.55\textwidth}
		\centering
		\includegraphics[width=3.5in]{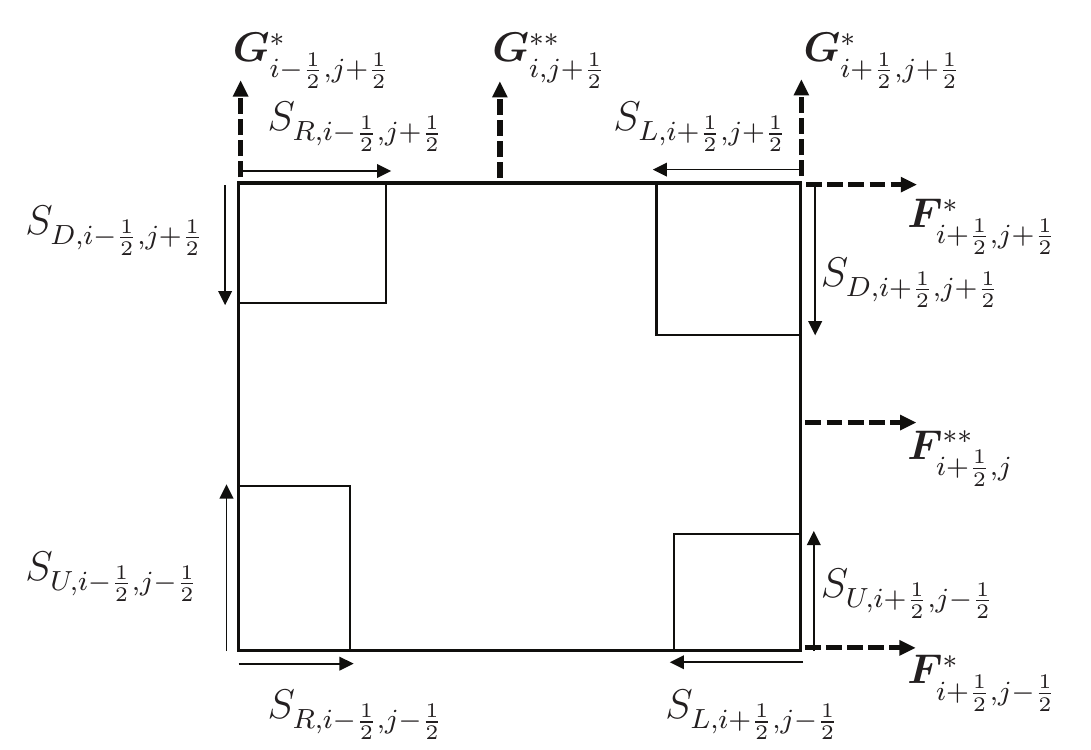}
	\end{minipage}
	\begin{minipage}{0.45\textwidth}
		\centering
		\includegraphics[width=2.6in]{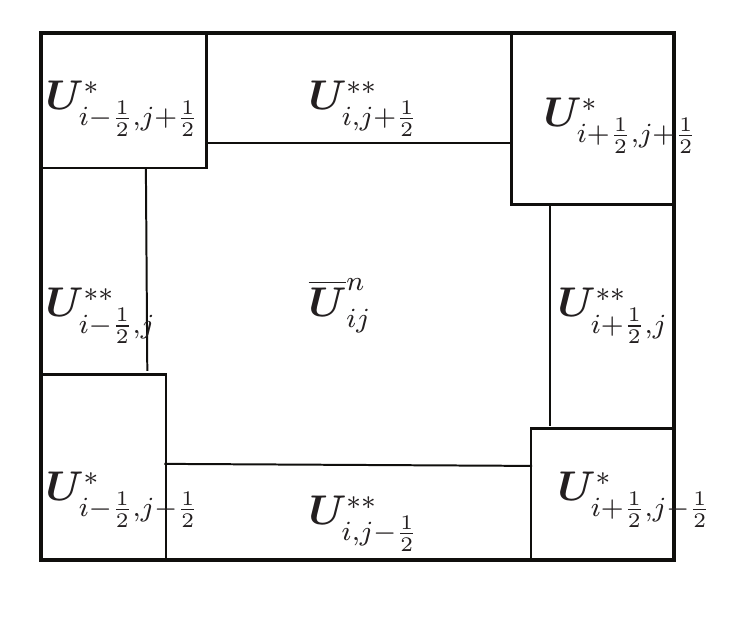}
	\end{minipage}
	\caption{Illustration for the numerical fluxes (left) and approximate solution (right) for the numerical scheme
		\eqref{eq20}.}
	\label{inter}
\end{figure}

Let us now discuss the PCP property of the scheme \eqref{eq20} with the numerical
fluxes  \eqref{add23} and \eqref{add24}.
According to the 1D and 2D HLL Riemann solvers,
under the CFL condition \eqref{eq39} with $\sigma\le\frac{1}{2}$,
$\overline{\bm{U}}_{ij}^{n+1}$ in the scheme
\eqref{eq20} can be written as an exact integration of those approximate Riemann solutions
over the cell $I_{ij}$, namely
\begin{equation*}
\begin{aligned}
\overline{\bm{U}}_{ij}^{n+1}&=\frac{1}{\Delta x_i\Delta y_j}
\bigg(\int_{\Omega_1}R_h(x/t,y/t,\overline{\bm{U}}_{i-1,j-1}^n,\overline{\bm{U}}_{i-1,j}^n,
\overline{\bm{U}}_{i,j-1}^n,\overline{\bm{U}}_{ij}^n)dxdy
+\int_{\Omega_5}\widetilde{R}_h(x/t,\overline{\bm{U}}_{i-1,j}^n,\overline{\bm{U}}_{ij}^n)dxdy\\
&~~~+\int_{\Omega_2}R_h(x/t,y/t,\overline{\bm{U}}_{i,j-1}^n,\overline{\bm{U}}_{i,j}^n,
\overline{\bm{U}}_{i+1,j-1}^n,\overline{\bm{U}}_{i+1,j}^n)dxdy
+\int_{\Omega_6}\widetilde{R}_h(x/t,\overline{\bm{U}}_{ij}^n,\overline{\bm{U}}_{i+1,j}^n)dxdy\\
&~~~+\int_{\Omega_3}R_h(x/t,y/t,\overline{\bm{U}}_{i-1,j}^n,\overline{\bm{U}}_{i-1,j+1}^n,
\overline{\bm{U}}_{ij}^n,\overline{\bm{U}}_{i,j+1}^n)dxdy
+\int_{\Omega_7}\widetilde{R}_h(y/t,\overline{\bm{U}}_{i,j-1}^n,\overline{\bm{U}}_{ij}^n)dxdy\\
&~~~+\int_{\Omega_4}R_h(x/t,y/t,\overline{\bm{U}}_{ij}^n,\overline{\bm{U}}_{i,j+1}^n,
\overline{\bm{U}}_{i+1,j}^n,\overline{\bm{U}}_{i+1,j+1}^n)dxdy
+\int_{\Omega_8}\widetilde{R}_h(y/t,\overline{\bm{U}}_{ij}^n,\overline{\bm{U}}_{i,j+1}^n)dxdy\\
&~~~+\int_{I_{ij}\setminus\bigcup\limits_{m=1}^8\Omega_m}\overline{\bm{U}}_{ij}^ndxdy\bigg),
\end{aligned}
\end{equation*}
where $R_h(x/t,y/t,\overline{\bm{U}}_1,\overline{\bm{U}}_2,
\overline{\bm{U}}_3,\overline{\bm{U}}_4)$ is the approximate solution of the 2D Riemann problem
with  four initial states $\overline{\bm{U}}_1,\overline{\bm{U}}_2,
\overline{\bm{U}}_3$ and $\overline{\bm{U}}_4$,
$\widetilde{R}_h(z/t,\overline{\bm{U}}_1,\overline{\bm{U}}_{2})$
is the approximate solution of 1D Riemann problems in $z$-direction with two initial states
$\overline{\bm{U}}_1,\overline{\bm{U}}_2$, $z=x,y$,
and
$$\begin{aligned}
\Omega_1&=\big[x_{i-\frac{1}{2}},x_{i-\frac{1}{2}}+\Delta t^nS_{R,i-\frac{1}{2},j-\frac{1}{2}}^+\big]
\times\big[y_{j-\frac{1}{2}},y_{j-\frac{1}{2}}+\Delta t^nS_{U,i-\frac{1}{2},j-\frac{1}{2}}^+\big],\\
\Omega_2&=\big[x_{i+\frac{1}{2}}+\Delta t^nS_{L,i+\frac{1}{2},j-\frac{1}{2}}^-,x_{i+\frac{1}{2}}\big]
\times\big[y_{j-\frac{1}{2}},y_{j-\frac{1}{2}}+\Delta t^n{S_{U,i+\frac{1}{2},j-\frac{1}{2}}^+}\big],\\
\Omega_3&=\big[x_{i-\frac{1}{2}},x_{i-\frac{1}{2}}+\Delta t^nS_{R,i-\frac{1}{2},j-\frac{1}{2}}^+\big]
\times\big[y_{j+\frac{1}{2}}+\Delta t^nS_{D,i-\frac{1}{2},j+\frac{1}{2}}^-,y_{j+\frac{1}{2}}\big],\\
\Omega_4&=\big[x_{i+\frac{1}{2}}+\Delta t^nS_{L,i+\frac{1}{2},j-\frac{1}{2}}^-,x_{i+\frac{1}{2}}\big]
\times\big[y_{j+\frac{1}{2}}+\Delta t^nS_{D,i+\frac{1}{2},j+\frac{1}{2}}^-,y_{j+\frac{1}{2}}\big],\\
\Omega_5&=\big[x_{i-\frac{1}{2}},x_{i-\frac{1}{2}}+\Delta t^nS_1^+\big]
\times\big[y_{j-\frac{1}{2}}+\Delta t^nS_{U,i-\frac{1}{2},j+\frac{1}{2}}^+,y_{j+\frac{1}{2}}
+\Delta t^nS_{D,i-\frac{1}{2},j+\frac{1}{2}}^-\big],\\
\Omega_6&=\big[x_{i+\frac{1}{2}}+\Delta t^nS_2^-,x_{i+\frac{1}{2}}\big]
\times\big[y_{j-\frac{1}{2}}+\Delta t^nS_{U,i+\frac{1}{2},j+\frac{1}{2}}^+,y_{j+\frac{1}{2}}
+\Delta t^nS_{D,i+\frac{1}{2},j+\frac{1}{2}}^-\big],\\
\Omega_7&=\big[x_{i-\frac{1}{2}}+\Delta t^nS_{R,i-\frac{1}{2},j-\frac{1}{2}}^+,
x_{i+\frac{1}{2}}+\Delta t^nS_{L,i+\frac{1}{2},j-\frac{1}{2}}^+\big]
\times\big[y_{j-\frac{1}{2}},y_{j-\frac{1}{2}}+\Delta t^nS_3^+\big],\\
\Omega_8&=\big[x_{i-\frac{1}{2}}+\Delta t^nS_{R,i-\frac{1}{2},j+\frac{1}{2}}^+,
x_{i+\frac{1}{2}}+\Delta t^nS_{L,i+\frac{1}{2},j+\frac{1}{2}}^+\big]
\times\big[y_{j+\frac{1}{2}}+\Delta t^nS_4^-,y_{j+\frac{1}{2}}\big],\\
\end{aligned}$$
here $S^-=\min(0,S),S^+=\max(0,S)$, $S_1$ and $S_2$ are the largest and smallest wave speeds
in the 1D Riemann HLL solver for two $x-$directional Riemann problems denoted by RP$\{\overline{\bm{U}}_{i-1,j}^n,\overline{\bm{U}}_{ij}^n\}$
and RP$\{\overline{\bm{U}}_{ij}^n,\overline{\bm{U}}_{i+1,j}^n\}$, $S_3$ and $S_4$ are the largest and smallest wave speeds
in the 1D Riemann HLL solver for two $y-$directional Riemann problems denoted by RP$\{\overline{\bm{U}}_{i,j-1}^n,\overline{\bm{U}}_{ij}^n\}$
and RP$\{\overline{\bm{U}}_{ij}^n,\overline{\bm{U}}_{i,j+1}^n\}$. Based on the aforementioned 1D and 2D HLL Riemann solvers,
one can get
$$\begin{aligned}
&\int_{\Omega_1}R_h(x/t,y/t,\overline{\bm{U}}_{i-1,j-1}^n,\overline{\bm{U}}_{i-1,j}^n,
\overline{\bm{U}}_{i,j-1}^n,\overline{\bm{U}}_{ij}^n)dxdy=|\Omega_1|\bm{U}_{i-\frac{1}{2},j-\frac{1}{2}}^{\ast},\\
&\int_{\Omega_2}R_h(x/t,y/t,\overline{\bm{U}}_{i,j-1}^n,\overline{\bm{U}}_{i,j}^n,
\overline{\bm{U}}_{i+1,j-1}^n,\overline{\bm{U}}_{i+1,j}^n)dxdy=|\Omega_2|\bm{U}_{i+\frac{1}{2},j-\frac{1}{2}}^{\ast},\\
&\int_{\Omega_3}R_h(x/t,y/t,\overline{\bm{U}}_{i-1,j}^n,\overline{\bm{U}}_{i-1,j+1}^n,
\overline{\bm{U}}_{ij}^n,\overline{\bm{U}}_{i,j+1}^n)dxdy=|\Omega_3|\bm{U}_{i-\frac{1}{2},j+\frac{1}{2}}^{\ast},\\
&\int_{\Omega_4}R_h(x/t,y/t,\overline{\bm{U}}_{ij}^n,\overline{\bm{U}}_{i,j+1}^n,
\overline{\bm{U}}_{i+1,j}^n,\overline{\bm{U}}_{i+1,j+1}^n)dxdy=|\Omega_4|\bm{U}_{i+\frac{1}{2},j+\frac{1}{2}}^{\ast},\\
&\int_{\Omega_5}\widetilde{R}_h(x/t,\overline{\bm{U}}_{i-1,j}^n,\overline{\bm{U}}_{ij}^n)dxdy=|\Omega_5|\bm{U}_{i-\frac{1}{2},j}^{\ast\ast},
~~~\int_{\Omega_6}\widetilde{R}_h(x/t,\overline{\bm{U}}_{ij}^n,\overline{\bm{U}}_{i+1,j}^n)dxdy=|\Omega_6|\bm{U}_{i+\frac{1}{2},j}^{\ast\ast},\\
&\int_{\Omega_7}\widetilde{R}_h(y/t,\overline{\bm{U}}_{i,j-1}^n,\overline{\bm{U}}_{ij}^n)dxdy=|\Omega_7|\bm{U}_{i,j-\frac{1}{2}}^{\ast\ast},
~~~\int_{\Omega_8}\widetilde{R}_h(y/t,\overline{\bm{U}}_{ij}^n,\overline{\bm{U}}_{i,j+1}^n)dxdy=|\Omega_8|\bm{U}_{i,j+\frac{1}{2}}^{\ast\ast},
\end{aligned}$$
where $|\Omega_m|$ stands for the area of the domain $\Omega_m$, $m=1,2,\cdots,8$,
the terms with the superscripts ``$\ast$'' and ``$\ast\ast$''
are obtained in the 2D and 1D HLL
Riemann solvers, respectively.
Clearly,
for the the non-trivial case of  $S_L<0<S_R,S_D<0<S_U$, the updated solution
$\overline{\bm{U}}_{ij}^{n+1}$ can be reformulated as a convex combination of nine terms:
$\overline{\bm{U}}_{ij}^n,\bm{U}_{i-\frac{1}{2},j-\frac{1}{2}}^{\ast},\bm{U}_{i+\frac{1}{2},j-\frac{1}{2}}^{\ast}$,
$\bm{U}_{i-\frac{1}{2},j+\frac{1}{2}}^{\ast},\bm{U}_{i-\frac{1}{2},j+\frac{1}{2}}^{\ast}$ and
$\bm{U}_{i-\frac{1}{2},j}^{\ast\ast},\bm{U}_{i+\frac{1}{2},j}^{\ast\ast},
\bm{U}_{i,j-\frac{1}{2}}^{\ast\ast},\bm{U}_{i,j+\frac{1}{2}}^{\ast\ast}$, see
the right schematic of Figure \ref{inter}.
On the other hand, it is obvious to know that
each term in the convex combination is admissible, see Section \ref{multi-hlle},
so that the numerical solution
$\overline{\bm{U}}_{ij}^{n+1}$ to the first-order scheme \eqref{eq20} with \eqref{add21}-\eqref{add24} is also admissible.
We conclude such result in the following theorem.

\begin{theorem}\label{thm2}
	If $\overline{\bm{U}}_{ij}^{n}\in \mathcal G$, for all $i,j\in \mathbb Z$, and the wave speeds $S_D,S_U,S_L,S_R$ are estimated
	by \eqref{eq2} {with $\alpha=2$},
	then $\overline{\bm{U}}_{ij}^{n+1}$ obtained by
	the first-order scheme \eqref{eq20} with \eqref{add21}--\eqref{add24} and the multidimensional Riemann solver
	belongs to the admissible state set $\mathcal G$
	under the time step size restriction \eqref{eq39}
	with   $\sigma\le\frac{1}{2}$.
\end{theorem}

\begin{remark}
The numerical fluxes in \eqref{add23}-\eqref{add24} may be further extended as
\begin{align}
	\widehat{\bm{F}}_{i+\frac{1}{2},j}&
	=\alpha \bm{F}^{\ast}_{i+\frac{1}{2},j-\frac{1}{2}}
	+\beta \bm{F}^{\ast}_{i+\frac{1}{2},j+\frac{1}{2}}
	+\left(1-(\alpha+\beta)\right)
	\bm{F}_{i+\frac{1}{2},j}^{\ast\ast},\label{add23-2}
	\\
	\widehat{\bm{G}}_{i,j+\frac{1}{2}}&
	=\tilde{\alpha}\bm{G}^{\ast}_{i-\frac{1}{2},j+\frac{1}{2}}
	+\tilde{\beta}\bm{G}^{\ast}_{i+\frac{1}{2},j+\frac{1}{2}}
	+\left(1-(\tilde\alpha+\tilde\beta)\right)
	\bm{G}_{i,j+\frac{1}{2}}^{\ast\ast},\label{add24-2}
\end{align}
which are similarly obtained  by approximating 
the flux integrals $I^y_{i+\frac12,j}$ and $I^x_{i,j+\frac12}$,
for example, 
in the case of  $S_D<0<S_U$, $S_L<0<S_R$, 
\begin{align*}
&I^y_{i+\frac12,j}
\approx\alpha{\bm{F}}^*_{i+\frac{1}{2},j-\frac{1}{2}}
+\big(1-(\alpha+\beta)\big){\bm{F}}^{**}_{i+\frac{1}{2},j}
+\beta {\bm{F}}^*_{i+\frac{1}{2},j+\frac{1}{2}},\\
& I^x_{i,j+\frac12}
\approx\tilde\alpha{\bm{G}}^*_{i-\frac{1}{2},j+\frac{1}{2}}
+\big(1-(\tilde\alpha+\tilde\beta)\big){\bm{G}}^{**}_{i,j+\frac{1}{2}}
+\tilde\beta {\bm{G}}^*_{i+\frac{1}{2},j+\frac{1}{2}},
\end{align*}
where $\alpha$, $\beta$, $\tilde{\alpha}$, and $\tilde{\beta}$ satisfy
$\alpha,\beta\geq 0$, $\alpha+\beta\leq 1$,
$\tilde\alpha,\tilde\beta\geq 0$, and $\tilde\alpha+\tilde\beta\leq 1$.
Obviously, those may be related to the weights of the Simpson rule or the three-point Gauss-Lobatto quadrature.
It is worth noting that the PCP property as in Theorem \ref{thm2} may be preserved by  the scheme \eqref{eq20} with \eqref{add23-2}--\eqref{add24-2} under some suitable CFL-type conditions.
\end{remark}

\subsection{High-order PCP scheme}

This subsection develops the high-order accurate PCP scheme for \eqref{eq27} with the previous 2D HLL Riemann solver, the high-order initial reconstruction, the high-order  approximation of the flux integrals
  \eqref{eq20bbb}, and the PCP flux limiter as well as the explicit SSP Runge-Kutta time discretization. Here the Gauss-Lobatto quadrature  with $K$ points and  weights $\{\omega_\alpha: \sum_{\alpha=1}^K \omega_\alpha=1\}$ 
  is used to calculate the flux integrals \eqref{eq20bbb}  in order to involve the 2D HLL Riemann solver, where
  $2K-3\ge r$ for a $(r+1)$th-order accurate scheme for \eqref{eq27}. 

Denote the Gauss-Lobatto quadrature points  on the intervals $[x_{i-\frac{1}{2}},x_{i+\frac{1}{2}}]$ and $[y_{j-\frac{1}{2}},y_{j+\frac{1}{2}}]$ respectively as follows
\begin{equation}\label{eq50}
	S_i^x=\left\{x_i^\alpha,~ \alpha=1,2,\cdots,K\right\},
\end{equation}
 and
\begin{equation}\label{eq51}
	S_j^y=\left\{y_j^\beta,~\beta=1,2,\cdots,K\right\},
\end{equation}
and define the Gauss-Lobatto quadrature points on the cell $I_{ij}$
by
$$S_{ij}=
\left\{(x_i^\alpha, y_j^\beta),~\alpha,\beta=1,2,\cdots,K\right\}.
$$

By using the (approximate)
cell average values $\overline{\bm{U}}_{ij}^{n}$  of $\bm{U}$ at $t_n$  over the cell $I_{ij}$, the dimension by dimension WENO reconstruction \cite{shu} with the local characteristic decomposition
is performed in the $x$- and $y$-directions respectively to get the high-order WENO approximations of $\bm{U}$  at those quadrature points
$S_i^x$, $S_j^y$, $S_{ij}$, denoted respectively by
$$
\bm{U}^n_{ij}(x_{i+\frac12}, y_j^\beta), \quad
\bm{U}^n_{ij}(x_i^\alpha, y_{j+\frac12}),\quad
\bm{U}^n_{ij}(x_i^\alpha, y_{j}^\beta).
$$
 The readers are also referred to \cite{zhao} for details.
The numerical solutions
 at $t_n$ can further be reconstructed as a piecewise polynomial
with $\bm{U}_{ij}^{n}(x,y)$ for $(x,y)\in I_{ij}$ by using the Lagrangian interpolation with the point values $\bm{U}^n_{ij}(x_i^\alpha, y_{j}^\beta)$.

Based on the above reconstruction and the Gauss-Lobatto quadrature with $K$ points, 
the flux integrals in \eqref{eq20bbb} can be approximately calculated as follows
%
\begin{equation}\label{add}
\begin{aligned}
&
I^y_{i+\frac12,j}
\approx\sum\limits_{\beta=1}^K\omega_\beta \bm{F}(\bm{U}^n_{ij}(x_{i+\frac{1}{2}},y_j^\beta))
\approx\sum_{\beta=1}^K\omega_\beta\widehat{\bm{F}}_{i+\frac{1}{2},
j_\beta}=:\widehat{\bm{F}}^{\text{high}}_{i+\frac{1}{2},j},\\
&
I^x_{i,j+\frac12}
\approx\sum\limits_{\alpha=1}^K\tilde\omega_\alpha
\bm{G}(\bm{U}^n_{ij}(x_i^\alpha,y_{j+\frac{1}{2}}))
\approx\sum_{\alpha=1}^K\tilde\omega_\alpha\widehat{\bm{G}}_{i_\alpha,j
+\frac{1}{2}}=:\widehat{\bm{G}}^{\text{high}}_{i,j+\frac{1}{2}},
\end{aligned}
\end{equation}
where $\{\omega_\beta\}$ and $\{\tilde\omega_\alpha\}$ denote the quadrature weights
for $I^y_{i+\frac12,j}$ and $I^x_{i,j+\frac12}$,
$\widehat{\bm{F}}_{i+\frac{1}{2},j_\beta}$ and $\widehat{\bm{G}}_{i_\alpha,j+\frac{1}{2}}$
denote the numerical fluxes, evaluated at quadrature points $(x_{i+\frac{1}{2}},y_j^\beta)$
and $(x_i^\alpha,y_{j+\frac{1}{2}})$  with the 1D   or 2D HLL Riemann solver,
respectively. Specially, we have
\begin{equation}\label{eq52}
	\begin{aligned}
		\widehat{\bm{F}}^{\text{high}}_{i\pm\frac{1}{2},j}&=\sum\limits_{\beta=1}^K\omega_\beta\widehat{\bm{F}}_{i\pm\frac{1}{2},j_\beta}
		=\omega_1(\bm{F}_{i\pm\frac{1}{2},j-\frac{1}{2}}^\ast+\bm{F}_{i\pm\frac{1}{2},j+\frac{1}{2}}^\ast)
		+\sum\limits_{\beta=2}^{K-1}\omega_\beta\bm{F}_{i\pm\frac{1}{2},j_\beta}^{\ast\ast},\\
		\widehat{\bm{G}}^{\text{high}}_{i,j\pm\frac{1}{2}}&=
		\sum\limits_{\alpha=1}^K\tilde\omega_\alpha\widehat{\bm{G}}_{i_\alpha,j\pm\frac{1}{2}}
		=\tilde\omega_1(\bm{G}_{i-\frac{1}{2},j\pm\frac{1}{2}}^\ast+\bm{G}_{i+\frac{1}{2},j\pm\frac{1}{2}}^\ast)
		+\sum\limits_{\alpha=2}^{K-1}\tilde\omega_\alpha\bm{G}_{i_\alpha,j+\frac{1}{2}}^{\ast\ast},
	\end{aligned}
\end{equation}
Then the scheme \eqref{eq20} becomes
	\begin{align}
		\overline{\bm{U}}_{ij}^{n+1}&=\overline{\bm{U}}_{ij}^{n}
		-\frac{\Delta t^n}{\Delta x_i}
		\big(\widehat{\bm{F}}^{\text{high}}_{i+\frac{1}{2},j}-\widehat{\bm{F}}^{\text{high}}_{i-\frac{1}{2},j}\big)-\frac{\Delta t^n}{\Delta y_j}
		\big(\widehat{\bm{G}}^{\text{high}}_{i,j+\frac{1}{2}}-\widehat{\bm{G}}^{\text{high}}_{i,j-\frac{1}{2}}\big).\label{eq53}
	\end{align}

In general, the high-order accurate scheme \eqref{eq53} with the numerical fluxes \eqref{eq52} does not satisfy the PCP property, namely, we can not
guarantee that $(\overline{\bm{U}}_{ij}^{n+1})^{\text{high}}:=\overline{\bm{U}}_{ij}^{n+1}$ obtained from \eqref{eq53} belongs to the admissible set $\mathcal{G}$. Here, we utilize the
PCP flux limiter  in \cite{wu2015} 
to get the following high-order PCP scheme 
\begin{equation}\label{PCPscheme}
\overline{\bm{U}}_{ij}^{n+1}=\overline{\bm{U}}_{ij}^n-\frac{\Delta t^n}{\Delta x_i}\big(\widehat{\bm{F}}_{i+\frac12,j}^{\text{PCP}}
-\widehat{\bm{F}}_{i-\frac12,j}^{\text{PCP}}\big)-\frac{\Delta t^n}{\Delta y_j}\big(\widehat{\bm{G}}_{i,j+\frac12}^{\text{PCP}}-
\widehat{\bm{G}}_{i,j-\frac12}^{\text{PCP}}\big),
\end{equation}
where
\begin{align} &\widehat{\bm{F}}_{i\pm\frac12,j}^{\rm PCP}=
\big(1-\theta^x_{i\pm\frac12,j}\big)\widehat{\bm{F}}_{i\pm\frac12,j}^{\rm low}
+\theta^x_{i\pm\frac12,j}\widehat{\bm{F}}_{i\pm\frac12,j}^{\rm high},\\
&\widehat{\bm{G}}_{i,j\pm\frac12}^{\rm PCP}=
\big(1-\theta^y_{i,j\pm\frac12}\big)\widehat{\bm{G}}_{i,j\pm\frac12}^{\rm low}
+\theta^y_{i,j\pm\frac12}\widehat{\bm{G}}_{i,j\pm\frac12}^{\rm high},\end{align}
with
$$\begin{aligned}
&\widehat{\bm{F}}_{i+\frac12,j}^{\rm low}=\frac{1}{2}\Big(\bm{F}(\overline{\bm{U}}_{ij}^n)
+\bm{F}(\overline{\bm{U}}_{i+1,j}^n)-\alpha^n_{i+\frac12,j}(\overline{\bm{U}}_{i+1,j}^n-\overline{\bm{U}}_{ij}^n)\Big),\\
&
\widehat{\bm{G}}_{i,j+\frac12}^{\rm low}=\frac{1}{2}\Big(\bm{G}(\overline{\bm{U}}_{ij}^n)
+\bm{G}(\overline{\bm{U}}_{i,j+1}^n)-\beta^n_{i,j+\frac12}(\overline{\bm{U}}_{i,j+1}^n-\overline{\bm{U}}_{ij}^n)\Big),
\end{aligned}$$
and $\alpha^n_{i+\frac12,j}=\max\{\varrho^x(\overline{\bm{U}}_{i+1,j}^n),
\varrho^x(\overline{\bm{U}}_{i,j}^n)\}$, 
$\beta^n_{i,j+\frac12}=\max\{\varrho^y(\overline{\bm{U}}_{i,j+1}^n),
\varrho^y(\overline{\bm{U}}_{i,j}^n)\}$.
Here
$ \varrho^x$ and $\varrho^y$ are the spectral radii of the Jacobian matrices
$\partial \bm{F}/\partial \bm{U}$ and $\partial \bm{G}/\partial \bm{U}$, respectively,
$\theta^x_{i\pm\frac12,j}$ and $\theta^y_{i,j\pm\frac12}$ are 
the PCP flux limiters defined below for $\widehat{\bm{F}}$ and $\widehat{\bm{G}}$, respectively.

If  assuming $\overline{\bm{U}}_{ij}^n\in\mathcal{G}$ for any $i,j$,
and
the CFL-type condition
\begin{equation}\label{addtime}
\Delta t^n\le\frac14\min\limits_{i,j}\Big(\frac{\Delta x_i}{\alpha^n_{i+\frac12,j}},\frac{\Delta y_j}{\beta^n_{i,j+\frac12}}\Big),
\end{equation}
then  using the  properties  of 
$\mathcal{G}$ in Lemma \ref{lem2} 
 yields
$$\bm{U}_{ij}^{\pm,\rm low}:=\overline{\bm{U}}_{ij}^n\mp\frac{4\Delta t^n}{\Delta x_i}\widehat{\bm{F}}_{i\pm\frac12,j}^{\rm low}\in\mathcal{G},
~~~~\widetilde{\bm{U}}_{ij}^{\pm,\rm low}:=\overline{\bm{U}}_{ij}^n\mp\frac{4\Delta t^n}{\Delta y_j}\widehat{\bm{G}}_{i,j\pm\frac12}^{\rm low}\in\mathcal{G}.$$

Define
$$\bm{U}_{ij}^{\pm,\rm high}:=\overline{\bm{U}}_{ij}^n\mp\frac{4\Delta t^n}{\Delta x_i}\widehat{\bm{F}}_{i\pm\frac12,j}^{\rm high},
~~~~\widetilde{\bm{U}}_{ij}^{\pm,\rm high}:=\overline{\bm{U}}_{ij}^n\mp\frac{4\Delta t^n}{\Delta y_j}\widehat{\bm{G}}_{i,j\pm\frac12}^{\rm high},$$
and introduce two small positive numbers $\varepsilon_D$ and $\varepsilon_q$ such that
$D\big(\bm{U}_{ij}^{\pm,\rm low}\big)\ge\varepsilon_D>0$, $D\big(\widetilde{\bm{U}}_{ij}^{\pm,\rm low}\big)\ge\varepsilon_D>0$,
$q\big(\bm{U}_{ij}^{\pm,\rm low}\big)\ge\varepsilon_q>0$, $q\big(\widetilde{\bm{U}}_{ij}^{\pm,\rm low}\big)\ge\varepsilon_q>0$. 
In our coming computations, $\varepsilon_D=\varepsilon_q=10^{-14}$.

The PCP flux limiters $\theta^x_{i\pm\frac12,j}$ and $\theta^y_{i,j\pm\frac12}$ are defined as follows.

{\tt (i)} Enforce the positivity of the mass density $D(\bm{U})$. 
For each $i$ and $j$, define
$$\begin{aligned}
&\theta_{i+\frac12,j}^{D,x,\pm}=\left\{\begin{array}{ll}
(D_{i+\frac12\mp\frac12,j}^{\pm,\rm low}-\varepsilon_D)/(D_{i+\frac12\mp\frac12,j}^{\pm,\rm low}-D_{i+\frac12\mp\frac12,j}^{\pm,\rm high}),
&~\text{if}~D_{i+\frac12\mp\frac12,j}^{\pm,\rm high}<\varepsilon_D,\\
1,&~\text{otherwise},\end{array}\right.\\
&\theta_{i,j+\frac12}^{D,y,\pm}=\left\{\begin{array}{ll}
(\widetilde{D}_{i,j+\frac12\mp\frac12}^{\pm,\rm low}-\varepsilon_D)/(\widetilde{D}_{i,j+\frac12\mp\frac12}^{\pm,\rm low}-\widetilde{D}_{i,j+\frac12\mp\frac12}^{\pm,\rm high}),
&~\text{if}~\widetilde{D}_{i,j+\frac12\mp\frac12}^{\pm,\rm high}<\varepsilon_D,\\
1,&~\text{otherwise},\end{array}\right.
\end{aligned}$$
and limit
\begin{align}
&\left\{\widehat{\bm{F}}_{i+\frac12,j}^{D}\right\}_k=\left\{\begin{array}{ll}
(1-\theta_{i+\frac12,j}^{D,x})\left\{\widehat{\bm{F}}_{i+\frac12,j}^{\text{low}}\right\}_k
+\theta_{i+\frac12,j}^{D,x}\left\{\widehat{\bm{F}}_{i+\frac12,j}^{\text{high}}\right\}_k, &k=1,\\
\left\{\widehat{\bm{F}}_{i+\frac12,j}^{\text{high}}\right\}_k, & k>1,
\end{array}\right.\label{limiter-f1}\\
&\left\{\widehat{\bm{G}}_{i,j+\frac12}^{\rm D}\right\}_k=\left\{\begin{array}{ll}
	(1-\theta_{i,j+\frac12}^{D,y})\left\{\widehat{\bm{G}}_{i,j+\frac12}^{\text{low}}\right\}_k
	+\theta_{i,j+\frac12}^{D,y}\left\{\widehat{\bm{G}}_{i,j+\frac12}^{\text{high}}\right\}_k, &k=1,\\
	\left\{\widehat{\bm{G}}_{i,j+\frac12}^{\text{high}}\right\}_k, & k>1,
\end{array}\right.\label{limiter-g1}
\end{align}
where $\theta_{i+\frac12,j}^{D,x}=\min\{\theta_{i+\frac12,j}^{D,x,+},\theta_{i+\frac12,j}^{D,x,-}\},
\theta_{i,j+\frac12}^{D,y}=\min\{\theta_{i,j+\frac12}^{D,y,+},\theta_{i,j+\frac12}^{D,y,-}\}$, and $\left\{\widehat{\bm{F}}_{i+\frac12,j}\right\}_k,\left\{\widehat{\bm{G}}_{i,j+\frac12}\right\}_k$
are the $k$-th components of $\widehat{\bm{F}}_{i+\frac12,j}$ and $\widehat{\bm{G}}_{i,j+\frac12}$ respectively.

{\tt (ii)} Enforce the positivity of the term $q(\bm{U})=E-\sqrt{D^2+|\bm{m}|^2}$.
For each $i,j$, compute
$$\begin{aligned}
	&\theta_{i+\frac12,j}^{q,x,\pm}=\left\{\begin{array}{ll}
		\big(q(\bm{U}_{i+\frac12\mp\frac12,j}^{\pm,\rm low})-\varepsilon_q\big)/\big(q(\bm{U}_{i+\frac12\mp\frac12,j}^{\pm,\rm low})-q(\bm{U}_{i+\frac12\mp\frac12,j}^{\pm,\rm D})\big),
		&~\text{if}~q(\bm{U}_{i+\frac12\mp\frac12,j}^{\pm,\rm D})<\varepsilon_q,\\
		1,&~\text{otherwise},\end{array}\right.\\
	&\theta_{i,j+\frac12}^{q,y,\pm}=\left\{\begin{array}{ll}
		\big(q(\widetilde{\bm{U}}_{i,j+\frac12\mp\frac12}^{\pm,\rm low})-\varepsilon_q\big)/\big(q(\widetilde{\bm{U}}_{i,j+\frac12\mp\frac12}^{\pm,\rm low})-q(\widetilde{\bm{U}}_{i,j+\frac12\mp\frac12}^{\pm,\rm D})\big),
		&~\text{if}~q(\widetilde{\bm{U}}_{i,j+\frac12\mp\frac12}^{\pm,\rm D})<\varepsilon_q,\\
		1,&~\text{otherwise},\end{array}\right.
\end{aligned}$$
and then limit the numerical fluxes as
\begin{align}
&\widehat{\bm{F}}_{i+\frac12,j}^{\text{PCP}}=
(1-{\theta_{i+\frac12,j}^{q,x}})\widehat{\bm{F}}_{i+\frac12,j}^{\text{low}}
+{\theta_{i+\frac12,j}^{q,x}}\widehat{\bm{F}}_{i+\frac12,j}^{\text{D}},\label{limiter-f2}\\
&\widehat{\bm{G}}_{i,j+\frac12}^{\text{PCP}}=
(1-{\theta_{i,j+\frac12}^{q,y}})\widehat{\bm{G}}_{i,j+\frac12}^{\text{low}}
+{\theta_{i,j+\frac12}^{q,y}}\widehat{\bm{G}}_{i,j+\frac12}^{\text{D}},\label{limiter-g2}
\end{align}
where $\theta_{i+\frac12,j}^{q,x}=\min\{\theta_{i+\frac12,j}^{q,x,+},\theta_{i+\frac12,j}^{q,x,-}\}$ and
$\theta_{i,j+\frac12}^{q,y}=\min\{\theta_{i,j+\frac12}^{q,y,+},\theta_{i,j+\frac12}^{q,y,-}\}$.

{\tt (iii)} Define $\theta_{i\pm\frac12,j}^x:=\theta_{i\pm\frac12,j}^{D,x}\theta_{i\pm\frac12,j}^{q,x}$ and
$\theta_{i,j\pm\frac12}^y:=\theta_{i,j\pm\frac12}^{D,y}\theta_{i,j\pm\frac12}^{q,y}$.
 
It is not difficult to prove that the    scheme \eqref{PCPscheme} is consistent with the 2D RHD equations in \eqref{eq27} and also
is PCP,   
when $\overline{\bm{U}}_{ij}^{n}\in \mathcal{G}$ and a suitable time stepsize is given (see Theorem \ref{thm3}).
Furthermore, we also remark that such PCP  limiter does not destroy the original high order accuracy in the smooth region, more details
can be seen in \cite{wu2015}.

\begin{theorem}\label{thm3}
	Under the time step size restriction in \eqref{addtime}, if $\bm{U}_{ij}^{\pm,\rm low},\widetilde{\bm{U}}_{ij}^{\pm,\rm low}\in\mathcal{G}$
	for all $i,j$, and the wave speeds are estimated in \eqref{eq2}, then for the high-order finite volume
	scheme \eqref{eq53} we have 
	$$\overline{\bm{U}}_{ij}^{n+1}=\frac{1}{4}\left(\bm{U}_{ij}^{+,\rm PCP}+\bm{U}_{ij}^{-,\rm PCP}
	+\widetilde{\bm{U}}_{ij}^{+,\rm PCP}+\widetilde{\bm{U}}_{ij}^{-,\rm PCP}\right)\in\mathcal{G}, ~~~~\forall i,j,$$
where
$$\bm{U}_{ij}^{\pm,\rm PCP}=\overline{\bm{U}}_{ij}^{n}\mp\frac{4\Delta t^n}{\Delta x_i}\widehat{\bm{F}}^{\rm PCP}_{i\pm\frac{1}{2},j},
~~~\widetilde{\bm{U}}_{ij}^{\pm,\rm PCP}=\overline{\bm{U}}_{ij}^{n}\mp\frac{4\Delta t^n}{\Delta y_j}\widehat{\bm{G}}^{\rm PCP}_{i,j\pm\frac{1}{2}}.$$
\end{theorem}
\begin{proof}
Since $\overline{\bm{U}}_{ij}^{n+1}$ is a convex combination of four terms $\bm{U}_{ij}^{+,\rm PCP},\bm{U}_{ij}^{-,\rm PCP}
,\widetilde{\bm{U}}_{ij}^{+,\rm PCP},\widetilde{\bm{U}}_{ij}^{-,\rm PCP}$, a sufficient condition for $\overline{\bm{U}}_{ij}^{n+1}\in\mathcal{G}$
is that each term belongs to the admissible set $\mathcal{G}$ due to the convexity of $\mathcal{G}$, see Lemma \ref{lem1}. Without loss of generality, we here just provide
the proof for $\bm{U}_{ij}^{\pm,\rm PCP}\in\mathcal{G}$ and similar analysis can be applied on $\widetilde{\bm{U}}_{ij}^{\pm,\rm PCP}$.

With the assumption and the above PCP flux limiter, we know that $0\le\theta_{i+\frac12,j}^{D,x}\le1,0\le\theta_{i+\frac12,j}^{q,x}\le1$
and there exist two small positive numbers $\varepsilon_D$ and $\varepsilon_q$ such that $D_{ij}^{\pm,\rm low}\ge\varepsilon_D>0,
q\big(\bm{U}_{ij}^{\pm,\rm low}\big)\ge\varepsilon_q>0$.

Combining \eqref{limiter-f1} and \eqref{limiter-f2}  gets
$$\left\{\widehat{\bm{F}}_{i\pm\frac12,j}^{\rm PCP}\right\}_1=
\big(1-\theta^x_{i\pm\frac12,j}\big)\left\{\widehat{\bm{F}}_{i\pm\frac12,j}^{\rm low}\right\}_1
+\theta_{i\pm\frac12,j}^x\left\{\widehat{\bm{F}}_{i\pm\frac12,j}^{\rm high}\right\}_1,$$
and then
$$D_{ij}^{\pm,\rm PCP}=\big(1-\theta^x_{i\pm\frac12,j}\big)D_{ij}^{\pm,\rm low}+\theta^x_{i\pm\frac12,j}D_{ij}^{\pm,\rm high}.$$
 According to the definitions of 
$\theta^x_{i\pm\frac12,j}$ and $\theta_{i\pm\frac12,j}^{D,x,\pm}$, one has
$$0<\theta^x_{i\pm\frac12,j}\le\theta_{i\pm\frac12,j}^{D,x}\le\theta_{i\pm\frac12,j}^{D,x,\pm}\le1,$$
and
$$\big(1-\theta_{i\pm\frac12,j}^{D,x,\pm}\big)D_{ij}^{\pm,\rm low}+\theta_{i\pm\frac12,j}^{D,x,\pm}D_{ij}^{\pm,\rm high}\ge\varepsilon_D>0,$$
which implies
$$\begin{aligned}
D_{ij}^{\pm,\rm PCP}&=\big(1-\theta_{i\pm\frac12,j}^x\big)D_{ij}^{\pm,\rm low}+\theta_{i\pm\frac12,j}^xD_{ij}^{\pm,\rm high}\\
&=\frac{\theta^x_{i\pm\frac12,j}}{\theta_{i\pm\frac12,j}^{D,x,\pm}}\left((1-\theta_{i\pm\frac12,j}^{D,x,\pm})D_{ij}^{\pm,\rm low}+\theta_{i\pm\frac12,j}^{D,x,\pm}D_{ij}^{\pm,\rm high}\right)
+\left(1-\frac{\theta^x_{i\pm\frac12,j}}{\theta_{i\pm\frac12,j}^{D,x,\pm}}\right)D_{ij}^{\pm,\rm low}\\
&\ge\frac{\theta^x_{i\pm\frac12,j}}{\theta_{i\pm\frac12,j}^{D,x,\pm}}\varepsilon_D
+\left(1-\frac{\theta^x_{i\pm\frac12,j}}{\theta_{i\pm\frac12,j}^{D,x,\pm}}\right)\varepsilon_D=\varepsilon_D>0.
\end{aligned}$$
On the other hand, \eqref{limiter-f2} gives
$$\bm{U}_{ij}^{\pm,\rm PCP}=(1-\theta_{i\pm\frac12,j}^{q,x})\bm{U}_{ij}^{\pm,\rm low}+\theta_{i\pm\frac12,j}^{q,x}\bm{U}_{ij}^{\pm,\rm D},$$
and then
$$q\left(\bm{U}_{ij}^{\pm,\rm PCP}\right)=q\left((1-\theta_{i\pm\frac12,j}^{q,x})\bm{U}_{ij}^{\pm,\rm low}
+\theta_{i\pm\frac12,j}^{q,x}\bm{U}_{ij}^{\pm,\rm D}\right)
\ge (1-\theta_{i\pm\frac12,j}^{q,x})q\left(\bm{U}_{ij}^{\pm,\rm low}\right)+\theta_{i\pm\frac12,j}^{q,x}q\left(\bm{U}_{ij}^{\pm,\rm D}\right),$$
since the function $q(\bm{U})$ is concave. Hence we obtain
$$\begin{aligned}
q\left(\bm{U}_{ij}^{\pm,\rm PCP}\right)&\ge (1-\theta_{i\pm\frac12,j}^{q,x})q\left(\bm{U}_{ij}^{\pm,\rm low}\right)+\theta_{i\pm\frac12,j}^{q,x}q\left(\bm{U}_{ij}^{\pm,\rm D}\right)\\
&=\frac{\theta_{i\pm\frac12,j}^{q,x}}{\theta_{i\pm\frac12,j}^{q,x,\pm}}\left((1-\theta_{i\pm\frac12,j}^{q,x,\pm})q\left(\bm{U}_{ij}^{\pm,\rm low}\right)+\theta_{i\pm\frac12,j}^{q,x,\pm}q\left(\bm{U}_{ij}^{\pm,\rm D}\right)\right)
+\left(1-\frac{\theta_{i\pm\frac12,j}^x}{\theta_{i\pm\frac12,j}^{q,x,\pm}}\right)q\left(\bm{U}_{ij}^{\pm,\rm low}\right)\\
&\ge\frac{\theta_{i\pm\frac12,j}^{q,x}}{\theta_{i\pm\frac12,j}^{q,x,\pm}}\varepsilon_q
+\left(1-\frac{\theta_{i\pm\frac12,j}^{q,x}}{\theta_{i\pm\frac12,j}^{q,x,\pm}}\right)\varepsilon_q=\varepsilon_q>0.
\end{aligned}$$
Thus, $\bm{U}_{ij}^{\pm,\rm PCP}\in\mathcal{G}$. The proof is completed.
\end{proof}

\begin{remark}
In practice, the computation of the numerical fluxes $\widehat{\bm{F}}_{i+\frac12,j}^{\rm high}$ and $\widehat{\bm{G}}_{i,j+\frac12}^{\rm high}$
needs   that the reconstructed values $\bm{U}_{ij}^n(x_i^\alpha,y_{j+\frac{1}{2}})$
and $\bm{U}_{ij}^n(x_{i+\frac{1}{2}},y_j^\beta)$ are in the admissible set $\mathcal{G}$. 
It may be enforced by using the   PCP limiter 
similar to the above  by replacing $\bm{U}_{ij}^{\pm,\rm high}$ 
and $\widetilde{\bm{U}}_{ij}^{\pm,\rm high}$, see \cite{wu2017,ling}.
\end{remark}
\begin{remark}
In order to get a scheme of high order accuracy both in space and time,
we replace the forward Euler time discretization in the PCP scheme \eqref{PCPscheme} with an explicit third-order accurate SSP Runge-Kutta time discretization,
which is still PCP under a suitable CFL-type condition and is implemented as follows:
$$\begin{aligned}
	&\overline{\bm{U}}_{ij}^{(1)}
	=\overline{\bm{U}}_{ij}^{n}-\Delta t^n\mathcal{L}(\overline{\bm{U}}^n_{ij}),\\
	&\overline{\bm{U}}_{ij}^{(2)}
	=\frac{3}{4}\overline{\bm{U}}_{ij}^{n}+\frac{1}{4}\left(\overline{\bm{U}}_{ij}^{(1)}
	-\Delta t^n\mathcal{L}(\overline{\bm{U}}^{(1)}_{ij})\right),\\
	&\overline{\bm{U}}_{ij}^{n+1}
	=\frac{1}{3}\overline{\bm{U}}_{ij}^{n}+\frac{2}{3}\left(\overline{\bm{U}}_{ij}^{(2)}
	-\Delta t^n\mathcal{L}(\overline{\bm{U}}^{(2)}_{ij})\right),
\end{aligned}$$
where
$$\mathcal{L}(\overline{\bm{U}}_{ij})=\frac{1}{\Delta x_i}\big(\widehat{\bm{F}}^{\text{\rm PCP}}_{i+\frac{1}{2},j}-\widehat{\bm{F}}^{\text{\rm PCP}}_{i-\frac{1}{2},j}\big)+
\frac{1}{\Delta y_j}\big(\widehat{\bm{G}}^{\text{\rm PCP}}_{i,j+\frac{1}{2}}-\widehat{\bm{G}}^{\text{\rm PCP}}_{i,j-\frac{1}{2}}\big).$$
\end{remark}

\section{Numerical tests}\label{num}
This section conducts several numerical experiments on the 2D ultra-relativistic RHD problems with large Lorentz factor, or strong discontinuities, or low rest-mass density or pressure, to verify the
accuracy, robustness, and effectiveness of the present PCP schemes.
It is worth remarking that those ultra-relativistic RHD problems seriously challenge the numerical schemes. Unless otherwise stated, all the computations are restricted to the EOS
\eqref{eq15} with the adiabatic index $\Gamma=5/3$, and
the time step size $\Delta t^n$ determined by \eqref{eq39} with the CFL number $\sigma=0.45$. Moreover, in our computations, we apply the fifth-order WENO reconstruction \cite{shu} and replace $\Delta t^n$ with $(\Delta t^n)^{5/3}$ to match the spacial accuracy in Examples \ref{exam1} and \ref{exam2}.

\begin{example}\label{exam3} \rm
	We first construct a explosion problem to test the multi-dimensionality of our scheme by referring to that in \cite{barsara2}.
	Initially, the rest fluid with a unit rest-mass density is  in the domain  $\Omega=[-0.5,0.5]^2$.
	The pressure is set as 20 inside a circle of radius 1/10, while
	a smaller pressure of 0.1 is given all over outside the circle.
	Figure \ref{fig8}  plots the contours and cross sections along $y$-axis and $y=x$
	of the rest-mass density at $t=0.1$ obtained by using
	our first-order PCP scheme with the multidimensional Riemann solver, i.e.
	\eqref {eq20} with \eqref{add23}-\eqref{add24},
	on the mesh of $64\times64$ uniform cells.
	For a comparison, -eps-converted-to.pdf \ref{fig8b} gives the numerical solutions
	obtained by using
	corresponding scheme with the 1D Riemann solver, i.e. \eqref {eq20} with 
	$\widehat{\bm{F}}_{i+\frac{1}{2},j}=\bm{F}_{i+\frac{1}{2},j}^{\ast\ast}$
	and   $\widehat{\bm{G}}_{i,j+\frac{1}{2}}=\bm{G}_{i,j+\frac{1}{2}}^{\ast\ast}$.
	It can be clearly seen from them that the results obtained by
	our scheme with the 2D Riemann solver preserve the spherical symmetry better.
\end{example}

\begin{figure}[H]
\begin{minipage}{0.5\textwidth}
\centering
\includegraphics[width=3.0in]{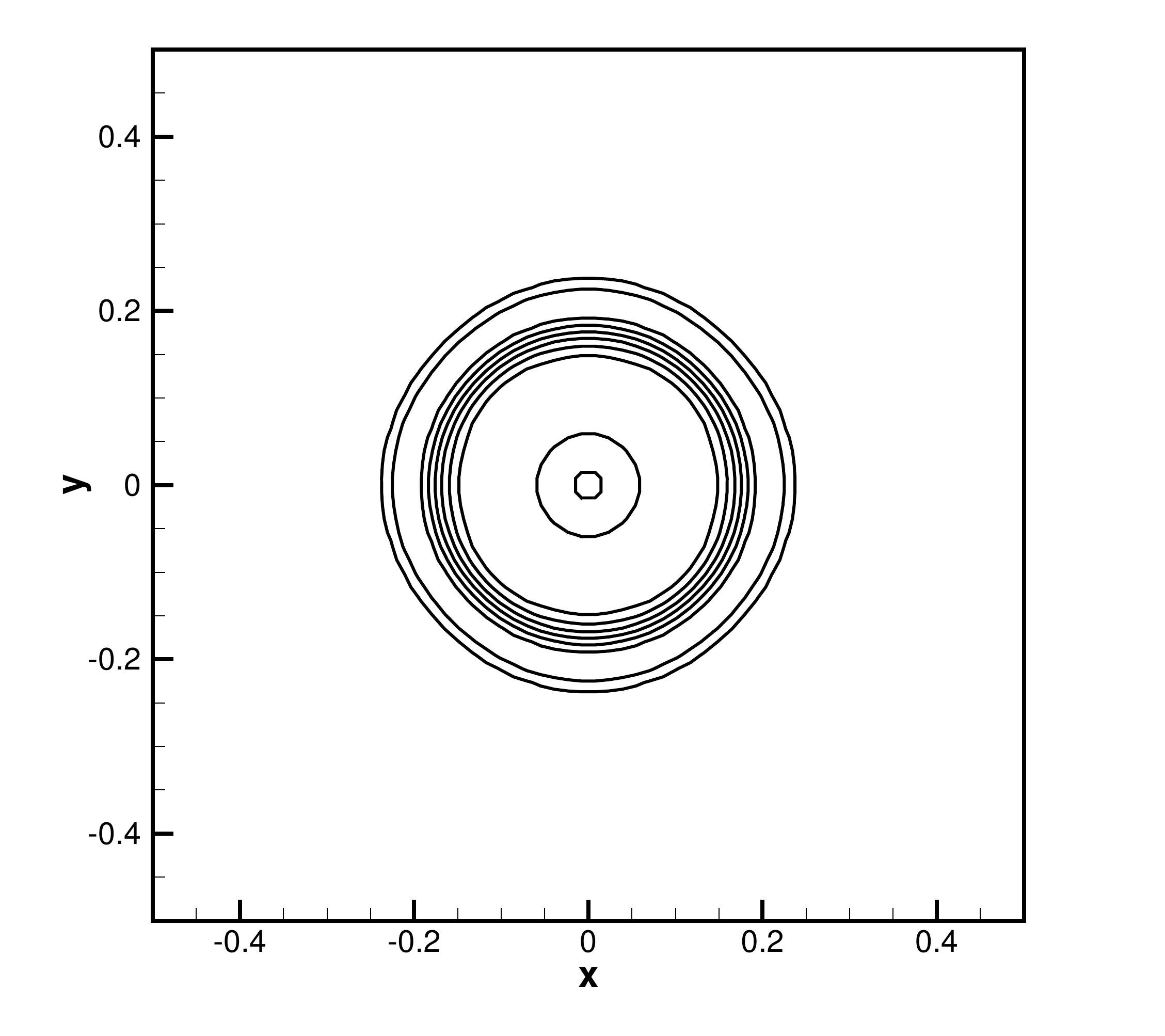}
\end{minipage}
\begin{minipage}{0.5\textwidth}
\centering
\includegraphics[width=3.2in]{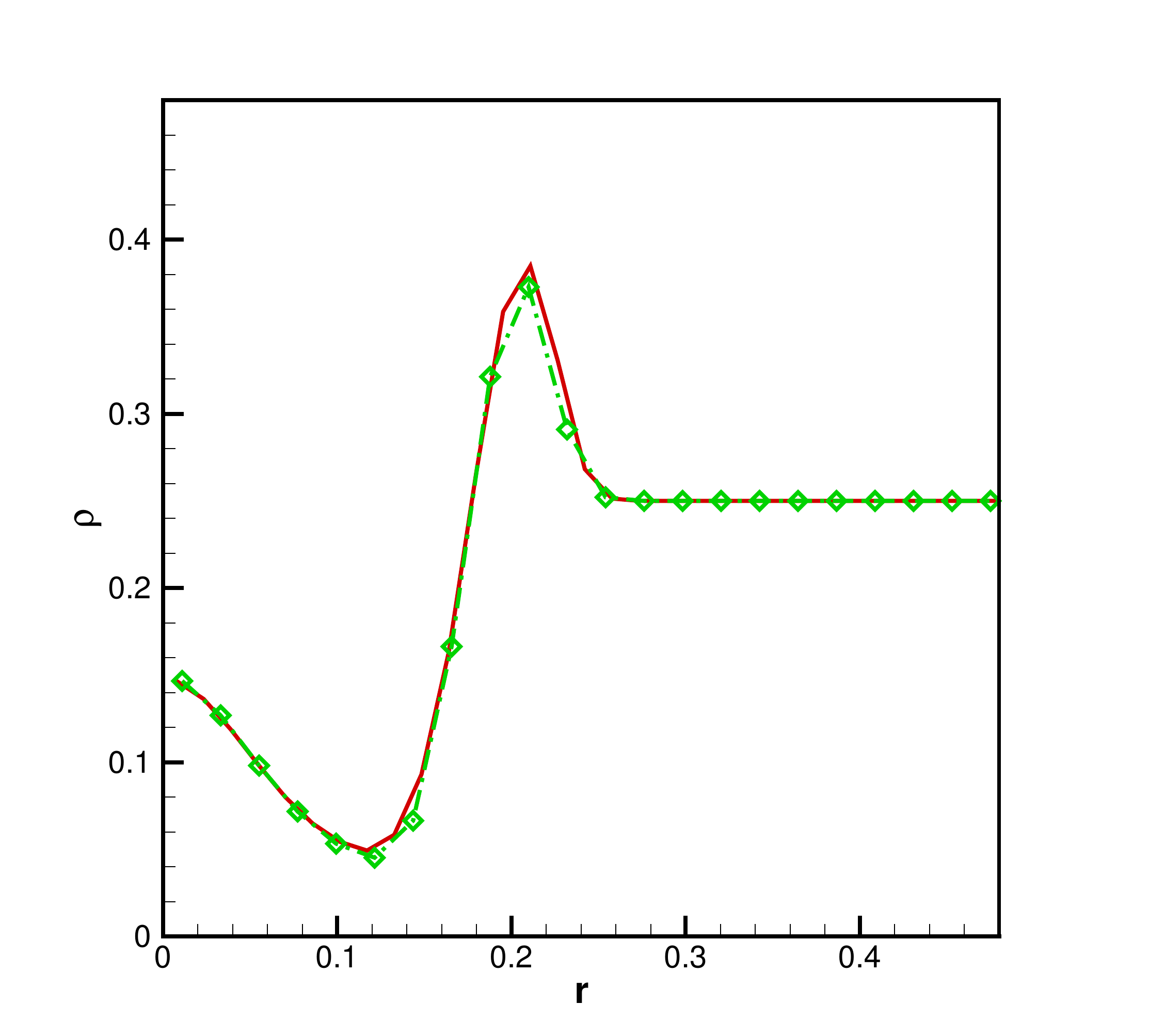}
\end{minipage}
\caption{\small Example \ref{exam3}: The rest-mass density at $t = 0.1$
		obtained from our first-order PCP scheme with the 2D
		HLL Riemann solver.
		Left: The contours with eight equally spaced contour lines;
		right: the cross sections of $\rho$ at $y$-axis (solid line) and $y=x$ (dashed line with the symbol ``$\diamond$'').}
\label{fig8}
\end{figure}

\begin{figure}[H]
\begin{minipage}{0.5\textwidth}
\centering
\includegraphics[width=3.0in]{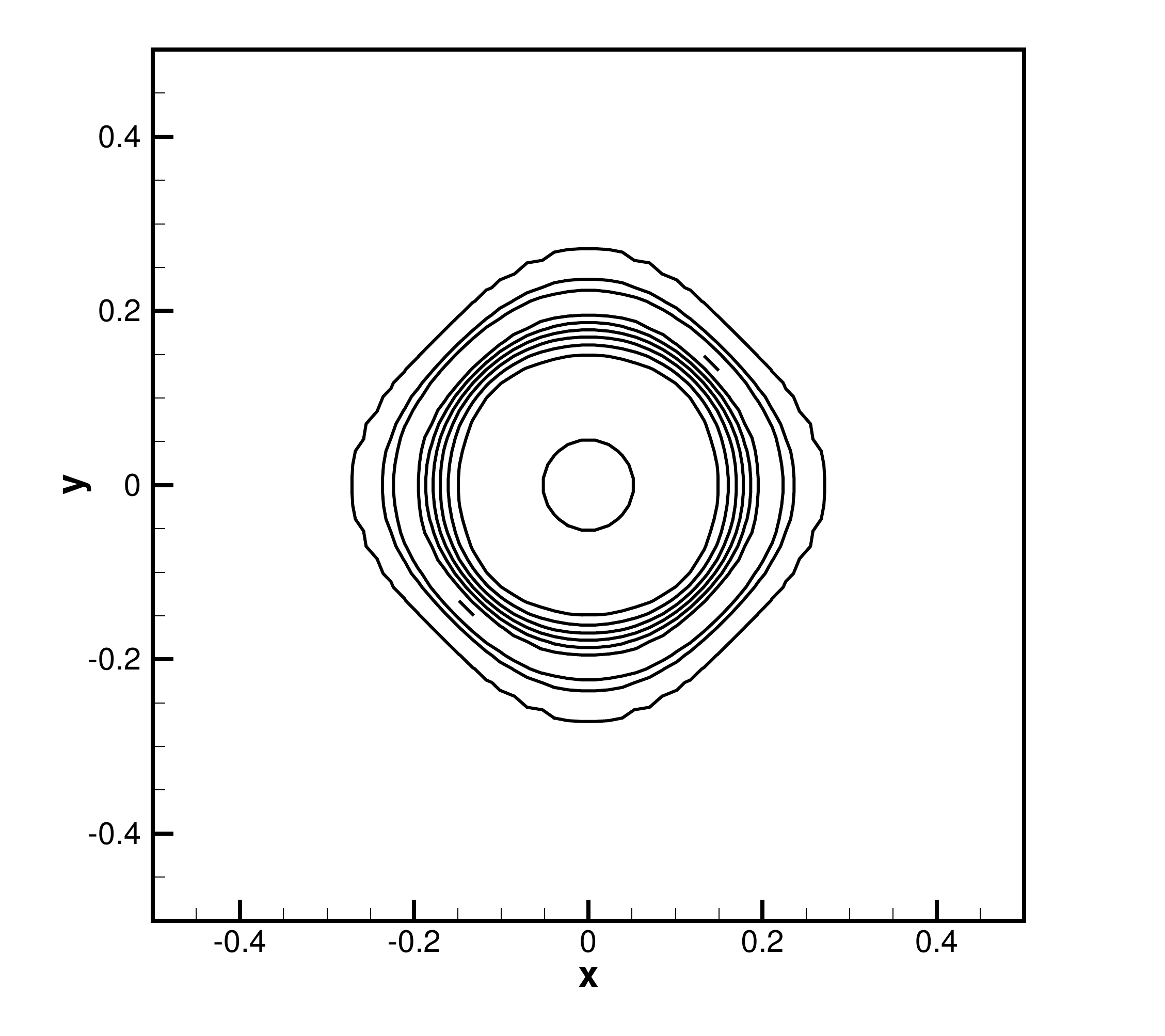}
\end{minipage}
\begin{minipage}{0.5\textwidth}
\centering
\includegraphics[width=3.2in]{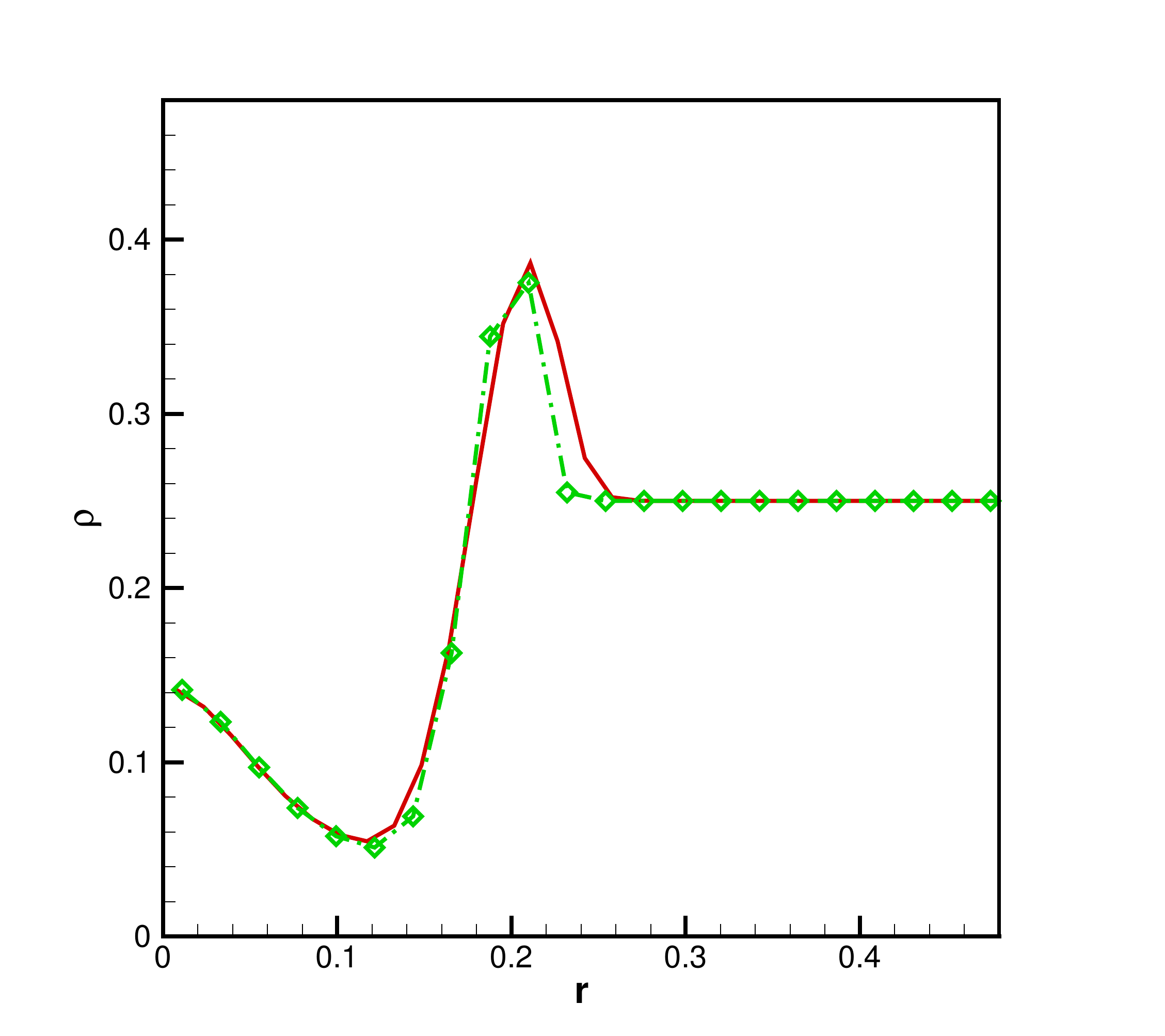}
\end{minipage}
	\caption{\small Same as Figure \ref{fig8}, except for the 1D
		HLL Riemann solver.}
	\label{fig8b}
\end{figure}

\begin{example}[Sine wave propagation \cite{wu2017}]\label{exam1}\rm
	This problem is used to test the accuracy of our PCP finite volume schemes. Its
	exact solution is
	$$(\rho,u,v,p)(x,y,t)=(1+0.99999\sin(2\pi(x+y-0.99\sqrt{2}t)),
	0.99/\sqrt{2},0.99/\sqrt{2},0.01),\ t\geq 0$$
	which describes an RHD sine wave propagating periodically in the domain $\Omega=[0,1]^2$ at
	an angle $45^\circ$ with the $x$-axis.
	The computational domain is divided into $N\times N$ uniform cells, and the periodic boundary
	conditions are specified on the boundary of $\Omega$.  Tables \ref{table1} and   \ref{table1-add} list the $\ell^1,\ell^2$ and
	$\ell^\infty$ errors at $t=0.1$ and orders of convergence obtained from our first-order and fifth-order multidimensional
	PCP schemes respectively. The results show the expected PCP performance.
	Table \ref{table1-add} also lists the proportions of the PCP limited cells at all time levels, denoted by $\Theta_N^1$ (scaling PCP limiter) and
	$\Theta_N^2$ (PCP flux limiter). It can be observed
	that the PCP limiter has been performed in the fifth-order accurate scheme because of the low density, and the usage of the limiter
	does not destroy the higher-order accuracy.
\end{example}
\begin{table}[H]
	\centering
\caption{\small Example \ref{exam1}: Errors and orders of convergence for the mass density at $t=0.1$
		obtained by using the first-order PCP scheme with the mesh of $N\times N$ uniform cells.}
	\label{table1}
	\begin{tabular}{c||c|c||c|c||c|c}
		\hline
		$N$ & $\ell^1$ error & $\ell^1$ order & $\ell^2$ error & $\ell^2$ order & $\ell^{\infty}$ error & $\ell^{\infty}$ order   \\
		\hline
 20 &  3.91E-01 &  --- &  4.36E-01 &  --- &  6.16E-01 &  ---\\
40 &  1.92E-01 &  1.03 &  2.13E-01 &  1.03 &  3.01E-01 &  1.03\\
80 &  9.49E-02 &  1.02 &  1.05E-01 &  1.02 &  1.49E-01 &  1.01\\
160 &  4.76E-02 &  1.00 &  5.28E-02 &  1.00 &  7.47E-02 &  1.00\\
320 &  2.38E-02 &  1.00 &  2.65E-02 &  1.00 &  3.74E-02 &  1.00\\
		\hline
	\end{tabular}
\end{table}

\begin{table}[H]
	\centering
\caption{\small Example \ref{exam1}: Same as Table \ref{table1}, except for   the fifth-order PCP scheme.}
	\label{table1-add}
	\begin{tabular}{c||c|c||c|c||c|c|c|c}
		\hline
		$N$ & $\ell^1$ error & $\ell^1$ order & $\ell^2$ error & $\ell^2$ order & $\ell^{\infty}$ error & $\ell^{\infty}$ order 
		& $\Theta_N^1$ (\%) & $\Theta_N^2$ (\%)  \\
		\hline
10   &    3.70E-02   &  --   &    4.08E-02   &  --   &    6.18E-02   &  -- & 36.18 &  0.00\\
20   &    1.37E-03   &  4.75   &    1.59E-03   &  4.68   &    3.07E-03   &  4.33 & 12.36 &  0.00\\
40   &    3.96E-05   &  5.12   &    4.64E-05   &  5.10   &    9.12E-05   &  5.07  & 1.85 &  0.00\\
80   &    1.19E-06   &  5.05   &    1.38E-06   &  5.07   &    2.87E-06   &  4.99 & 0.00  &  0.00\\
160  &    3.64E-08   &  5.03   &    4.15E-08   &  5.06   &    8.56E-08   &  5.07 & 0.00 &  0.00\\ 
\hline
	\end{tabular}
\end{table}

\begin{example}[Relativistic isentropic vortex]\label{exam2} \rm
	It is a 2D relativistic isentropic vortex problem constructed first in \cite{ling}, where
	the vortex in the space-time coordinate system $(x,y,t)$ moves with a constant speed of
	magnitude $w$ in $(-1,-1)$ direction.
	The time-dependent solution $(\rho,u,v,p)$ at time $t\geq 0$
	is given as follows
	\begin{align*}
		&\rho=(1-\alpha e^{1-r^2})^{\frac{1}{\Gamma-1}},\quad p=\rho^\Gamma,\\
		&u=\frac{1}{1-\frac{w(u_0+v_0)}{\sqrt{2}}}\left[\frac{u_0}{\gamma}-\frac{w}{\sqrt{2}}+\frac{\gamma w^2}{2(\gamma+1)}(u_0+v_0)\right],\\
		&v=\frac{1}{1-\frac{w(u_0+u_0)}{\sqrt{2}}}\left[\frac{v_0}{\gamma}-\frac{w}{\sqrt{2}}+\frac{\gamma w^2}{2(\gamma+1)}(u_0+v_0)\right],
	\end{align*}
	where
	\begin{align*}
		&\gamma=\frac{1}{\sqrt{1-w^2}},\quad r=\sqrt{x_0^2+y_0^2},\quad (u_0,v_0)=(-y_0,x_0)f,\\
		&\alpha=\frac{(\Gamma-1)}{8\Gamma\pi^2}\epsilon^2,\quad\beta=\dfrac{2\Gamma\alpha e^{1-r^2}}{2\Gamma-1-\Gamma\alpha e^{1-r^2}},
		\quad f=\sqrt{\frac{\beta}{1+\beta r^2}},\\
		&x_0=x+\frac{\gamma-1}{2}(x+y)+\frac{\gamma tw}{\sqrt{2}},~~y_0=y+\frac{\gamma-1}{2}(x+y)+\frac{\gamma tw}{\sqrt{2}}.
	\end{align*}
	Our computations are performed in the domain
	$\Omega=[-6,6]^2$ with the adiabatic index $\Gamma=1.4$, $w=0.5\sqrt{2}$,
	the vortex strength $\epsilon=10.0828$, and the periodic boundary conditions. In this case, the
	lowest density and lowest pressure are $7.83\times 10^{-15}$ and
	$1.78\times 10^{-20}$, respectively.
	
	Tables \ref{table2} and \ref{addtable2} give the errors of the rest-mass density at $t=1$ and the orders of convergence obtained
	from our first- and fifth-order PCP schemes respectively. It is clear to see that our multidimensional PCP schemes achieves the expected accuracy and preserves the positivity of the density and pressure simultaneously. Also the proportions of the PCP limited cells at all time levels, denoted by {$\Theta_N^1$ (scaling PCP limiter) and
		$\Theta_N^2$ (PCP flux limiter)}, are listed to show that the PCP limiter is indeed used to preserve the admissibility of numerical solutions.
\end{example}
\begin{table}[H]
	\centering
	\caption{\small Example \ref{exam2}: Errors and orders of convergence for mass
		density at $t=1$ obtained by using the first-order PCP scheme with the mesh of $N\times N$ uniform cells.}
	\label{table2}
	\begin{tabular}{c||c|c||c|c||c|c}
		\hline
		$N$ & $\ell^1$ error & $\ell^1$ order & $\ell^2$ error & $\ell^2$ order & $\ell^{\infty}$ error &
		$\ell^{\infty}$ order   \\
		\hline
20   &    2.48E+00   &  ---   &    7.41E-01   &  ---   &    5.54E-01   &  ---\\
40   &    1.63E+00   &  0.60   &    4.90E-01   &  0.60   &    3.61E-01   &  0.62\\
80   &    9.42E-01   &  0.80   &    2.91E-01   &  0.75   &    2.19E-01   &  0.72\\
160  &    5.12E-01   &  0.88   &    1.63E-01   &  0.84   &    1.30E-01   &  0.76\\
320  &    2.68E-01   &  0.93   &    8.66E-02   &  0.91   &    7.14E-02   &  0.86\\
		\hline
	\end{tabular}
\end{table}

\begin{table}[H]
	\centering
	\caption{\small Example \ref{exam2}: Same as Table \ref{table2}, except for  the fifth-order PCP scheme.}
	\label{addtable2}
	\begin{tabular}{c||c|c||c|c||c|c|c|c}
		\hline
		$N$ & $\ell^1$ error & $\ell^1$ order & $\ell^2$ error & $\ell^2$ order & $\ell^{\infty}$ error &
		$\ell^{\infty}$ order  & $\Theta_N^1$ (\%) & $\Theta_N^2$ (\%)\\
		\hline
20   &    9.12E-01   &  ---   &    2.88E-01   &  ---   &    2.37E-01   &  --- & 3.21  & 3.82\\
40   &    1.63E-01   &  2.49   &    7.60E-02   &  1.92   &    9.66E-02   &  1.30  & 1.53  & 1.85\\
80   &    8.66E-03   &  4.23   &    4.56E-03   &  4.06   &    1.19E-02   &  3.03  & 5.13E-01   & 7.13E-02\\
160  &    3.22E-04   &  4.75   &    1.64E-04   &  4.80   &    4.36E-04   &  4.77  & 2.34E-02 & 2.57E-04\\
320  &    1.12E-05   &  4.84   &    6.12E-06   &  4.74   &    1.84E-05   &  4.56  & 7.74E-04 & 2.77E-05\\
640  &    3.58E-07   &  4.97   &    1.95E-07   &  4.97   &    7.86E-07   &  4.55  & 9.40E-05 & 2.04E-06\\
		\hline
	\end{tabular}
\end{table}

\begin{example}[Riemann problem I]\label{exam4} \rm
	This example solves the 2D Riemann problem \cite{wu2015}.
	The initial data are given  by
	\begin{equation*}
		(\rho,u,v,p)(x,y,0)=\begin{cases}
			(0.1,0,0,0.01),& x>0.5,~y>0.5,\\
			(0.1,0.99,0,1),& x<0.5,~y>0.5,\\
			(0.5,0,0,1),& x<0.5,~y<0.5,\\
			(0.1,0,0.99,1),& x>0.5,~y<0.5,
		\end{cases}
	\end{equation*}
	where both the left and bottom discontinuities are contact discontinuities with a jump in the transverse
	velocity, while both the right and top discontinuities are not simple waves.
	
	The computational domain $\Omega$ is taken as $[0,1]^2$ and is divided into a uniform mesh with $400\times400$ cells.
	-eps-converted-to.pdf \ref{fig2} and \ref{fig2add} display the contours of the rest-mass density logarithm $\ln\rho$ and the pressure
	logarithm $\ln p$ at $t=0.4$ obtained by using the first- and the fifth-order PCP schemes respectively. We can see that
	the four initial discontinuities interact each other and form
	two reflected curved shock waves, an elongated jet-like spike.
	It is worth mentioning that a non-PCP scheme fails when simulating this problem.
	 Figure \ref{fig2add1}   also presents the cross sections of the numerical approximations along the line $y=x$ with $400\times400$ uniform mesh for the fifth-order PCP
	scheme and the same mesh, finer meshes of $800\times800$ and $1200\times1200$ for the first-order PCP scheme. It is obvious that the fifth-order
	scheme can capture the discontinuities better than the first-order scheme. Moreover, we count the PCP limited cells at each time level and the proportions are plotted in the Figure \ref{limiter-RP1}, from which one can clearly conclude that for the method without PCP property
	the simulation of this problem may fail.
\end{example}
\vspace{-2ex}
\begin{figure}[H]
	\centering
	\includegraphics[width=2.6in]{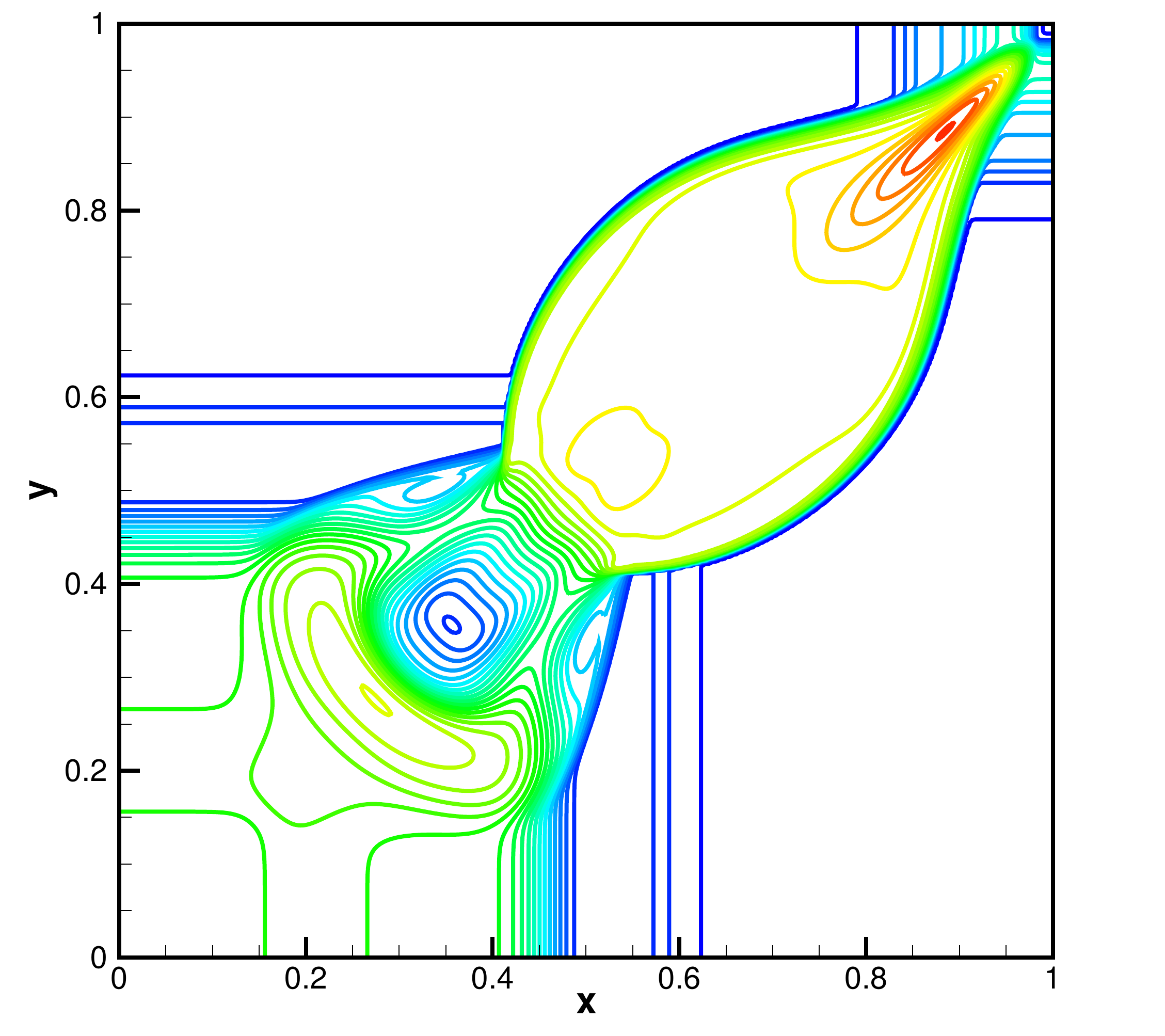}
	\includegraphics[width=2.6in]{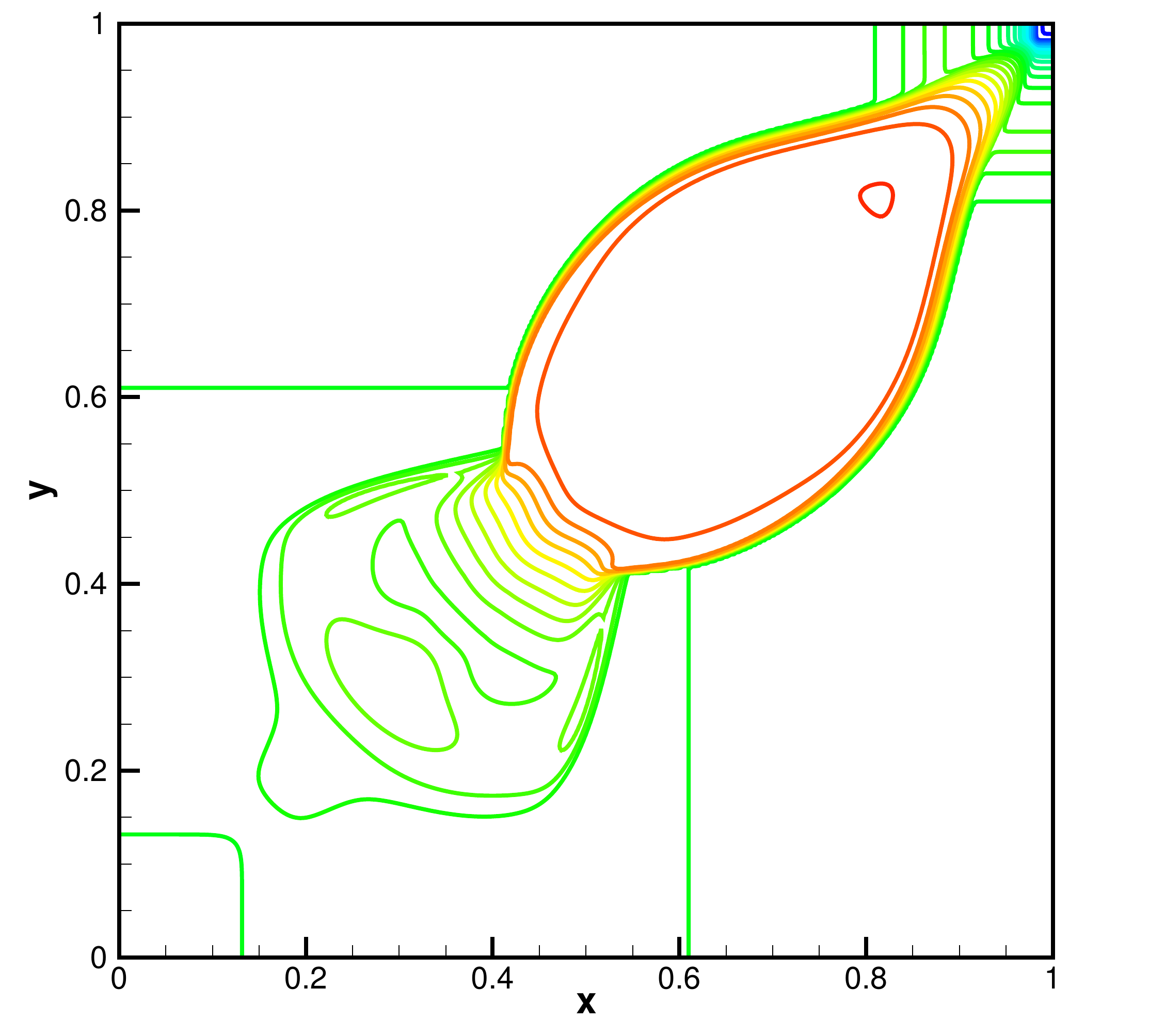}
	\caption{\small Example \ref{exam4}: The contours of the density logarithm $\ln\rho$ (left)
		and the pressure logarithm $\ln p$ (right) at $t=0.4$ obtained from the first-order PCP scheme. 25 equally spaced contour lines are used.}
	\label{fig2}
\end{figure}
\begin{figure}[H]
	\centering
	\includegraphics[width=2.6in]{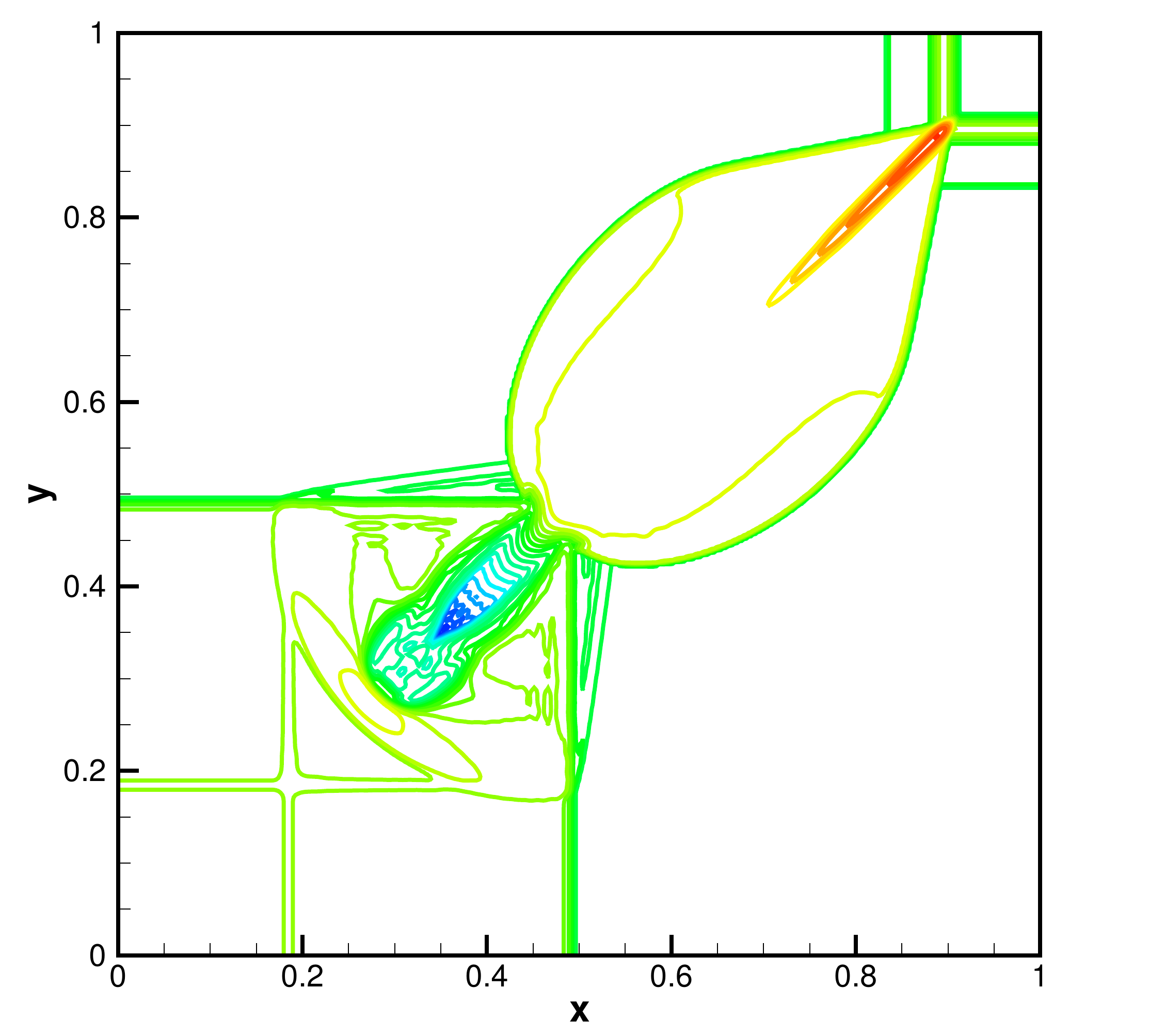}
	\includegraphics[width=2.6in]{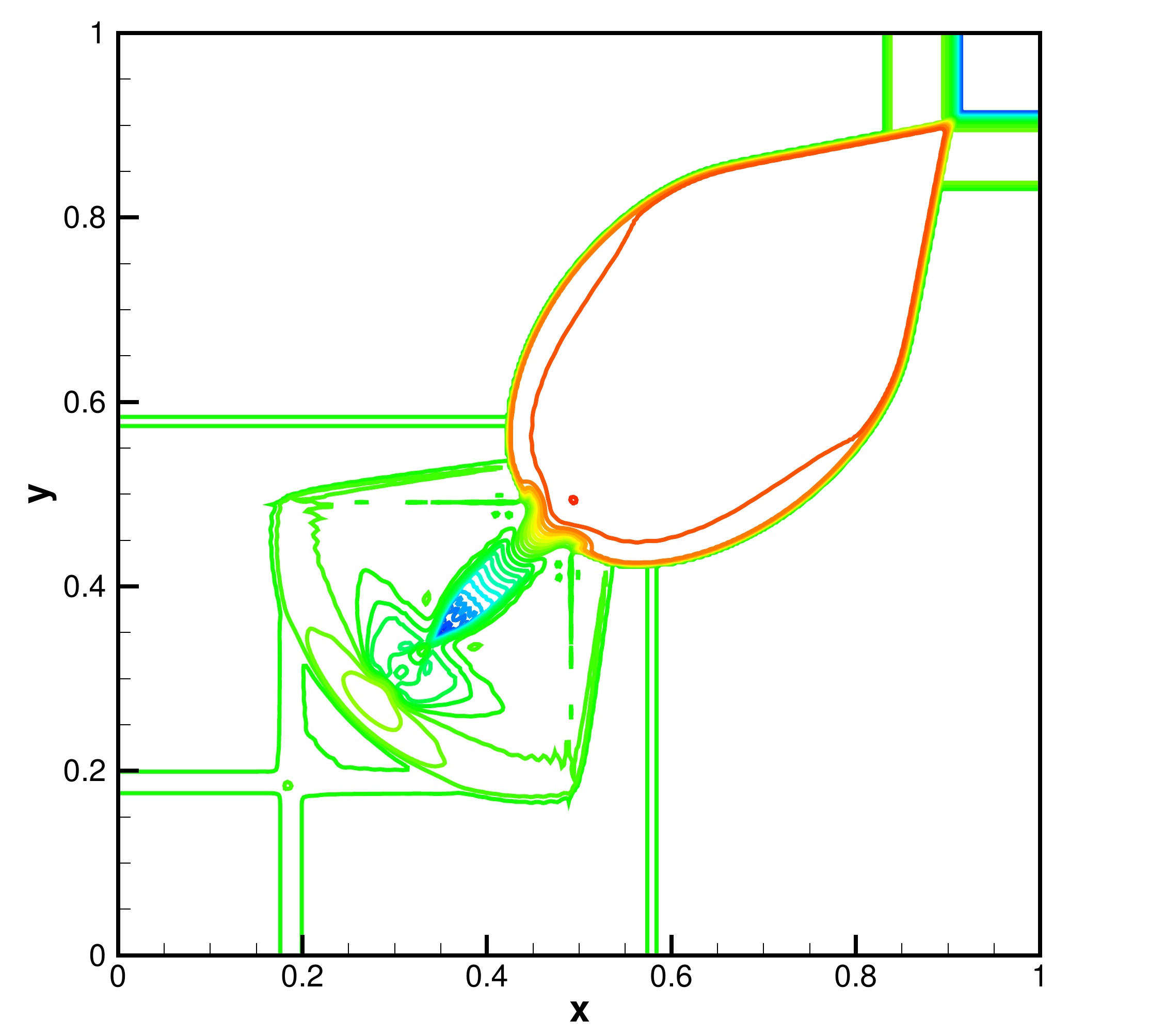}
	\caption{\small Example \ref{exam4}: Same as Figure \ref{fig2} except for the fifth-order PCP scheme.}
	\label{fig2add}
\end{figure}
\begin{figure}[H]
	\centering
	\includegraphics[width=2.6in]{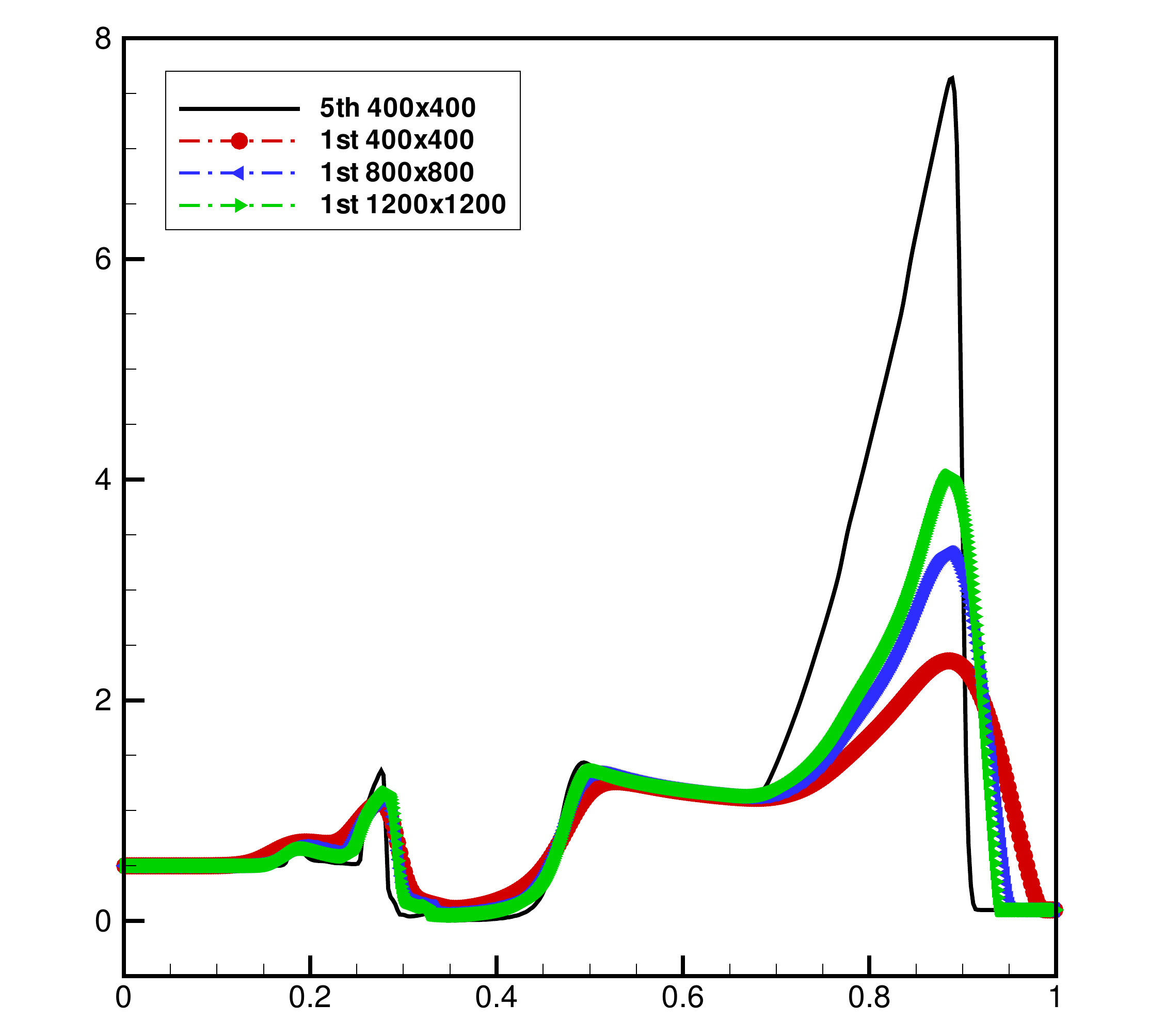}
	\includegraphics[width=2.6in]{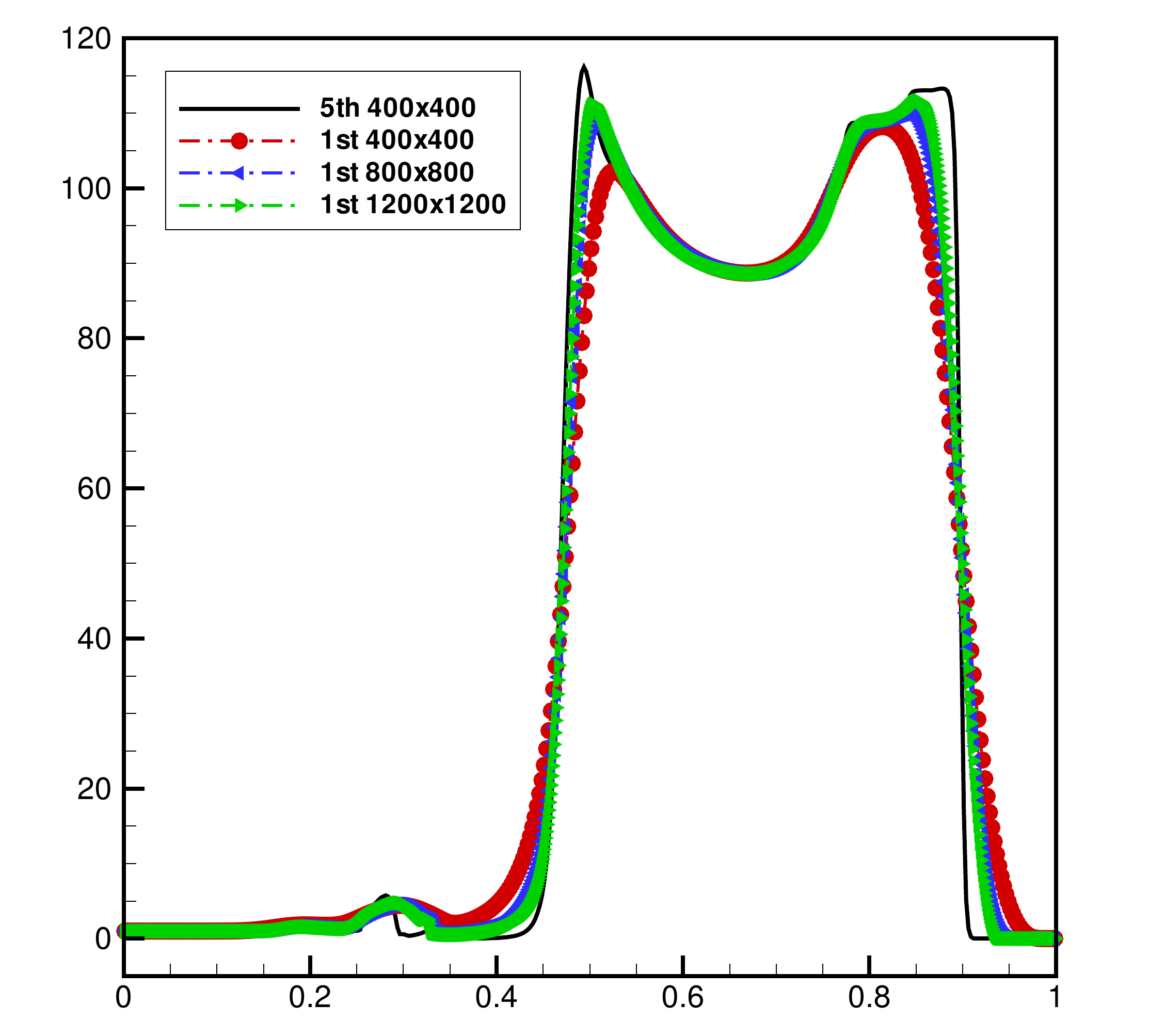}
	\caption{\small Example \ref{exam4}: Comparison of the cross sections of the numerical solutions at the line $y=x$ and $t=0.4$ in the closed interval $x\in[0,1]$ for different meshes and schemes. Left: the rest-mass density $\rho$; Right: the pressure $p$.}
	\label{fig2add1}
\end{figure}
\begin{figure}[H]
	\centering
	\includegraphics[width=2.6in]{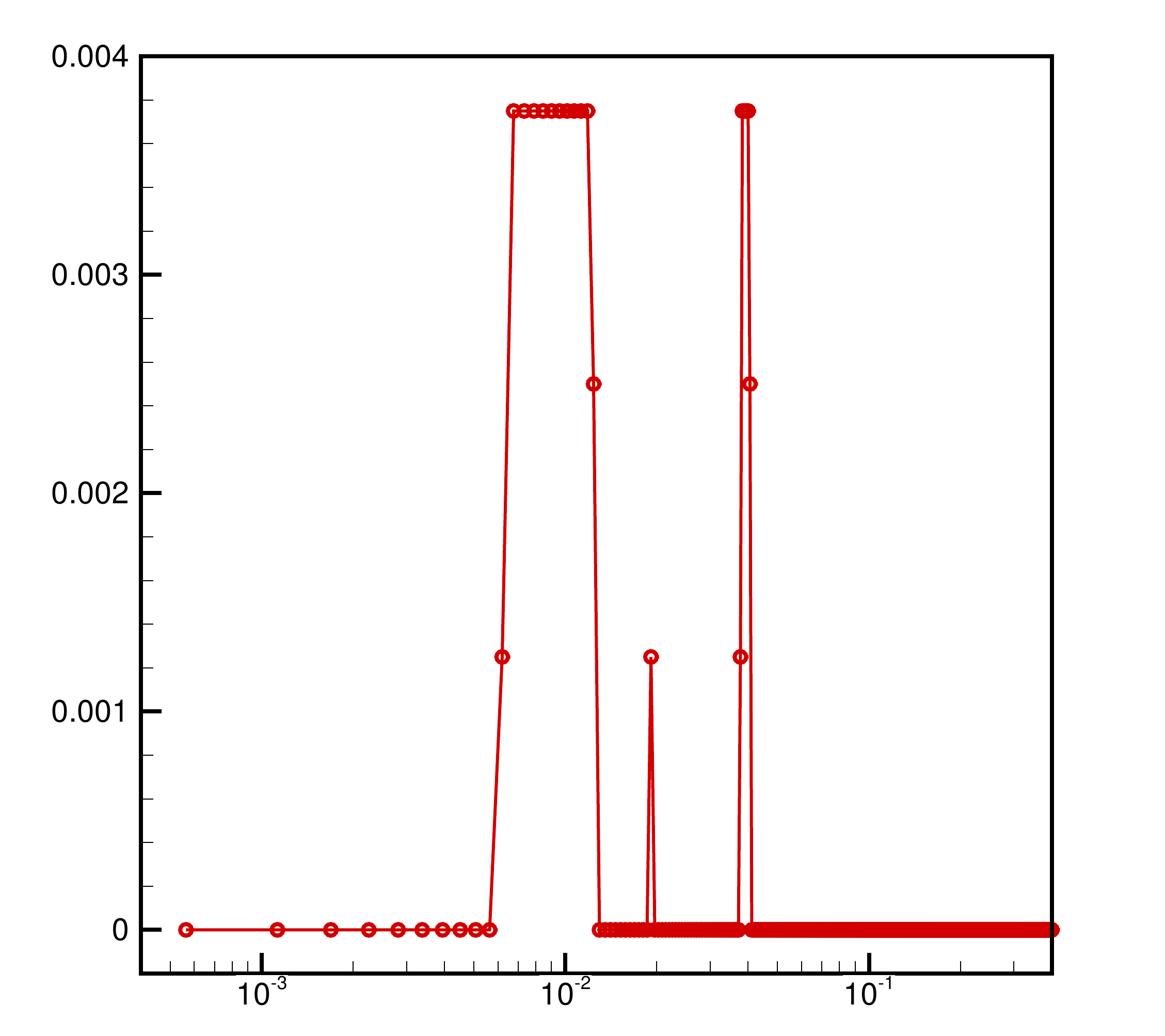}
	\includegraphics[width=2.6in]{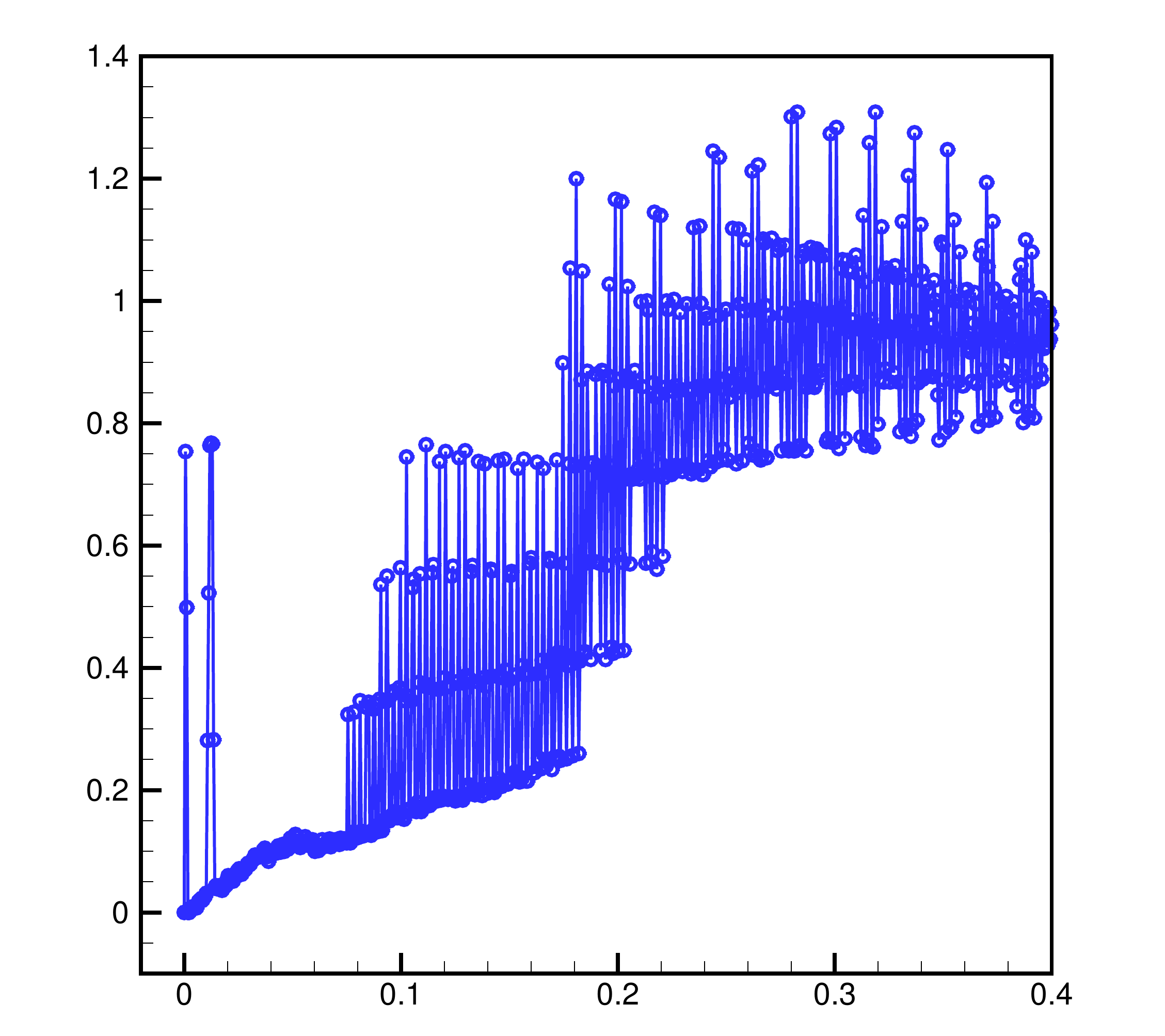}
	\caption{\small Example \ref{exam4}: Proportions of the PCP limited cells at each time level. Left: scaling PCP limiter; Right: PCP flux limiter.}
	\label{limiter-RP1}
\end{figure}

\begin{example}[Riemann problem II \cite{wu2015}]\label{exam5} \rm
	The initial data of the second Riemann problem are
	\begin{equation*}
		(\rho,u,v,p)(x,y,0)=\begin{cases}
			(0.1,0,0,20),& x>0.5,~y>0.5,\\
			(\widetilde{\rho},\widetilde{u},0,0.05),& x<0.5,~y>0.5,\\
			(0.01,0,0,0.05),& x<0.5,~y<0.5,\\
			(\widetilde{\rho},0,\widetilde{u},0.05),& x>0.5,~y<0.5,
		\end{cases}
	\end{equation*}
	with $\widetilde{\rho}=0.00414329639576$, $\widetilde{u}=0.9946418833556542$ and
	 the computational domain $\Omega=[0,1]^2$.
	In this problem, the left and lower initial discontinuities are contact discontinuities, while
	the upper and right are shock waves with a speed of $-0.66525606186639$.
	As the time increases, the maximal value of the fluid velocity becomes very large and close to the
	speed of light, which leads to the numerical simulation more challenging.
	Figures \ref{fig3} and \ref{fig3add} show the contours of the rest-mass density logarithm $\ln\rho$ and
	the pressure logarithm $\ln p$ at $t=0.4$ obtained by using the first- and fifth-order PCP schemes   on the uniform mesh of $400\times400$ cells, respectively.
	The interaction of four initial discontinuities results in the distortion of the initial
	shock waves and the formation of a ``mushroom cloud'' starting from the point (0.5,0.5)
	and expanding to the left bottom region. We also compare the numerical solutions obtained from the first- and high-order schemes
	in Figure \ref{fig3add1}, which displays the plots of the rest-mass logarithm $\ln\rho$ and the pressure logarithm $\ln p$ along the line
	$y=x$. We can see that the fifth-order PCP scheme gets better resolution for discontinuities than the first-order scheme even on a finer mesh.
	Furthermore, we also want to remark that a PCP scheme is necessary to simulate this problem since the PCP flux limiter and scaling PCP limiter are indeed used to preserve the physical-constraints property, see Figure \ref{limiter-RP2}.
\end{example}
\vspace{-2ex}
\begin{figure}[H]
\centering
\includegraphics[width=2.6in]{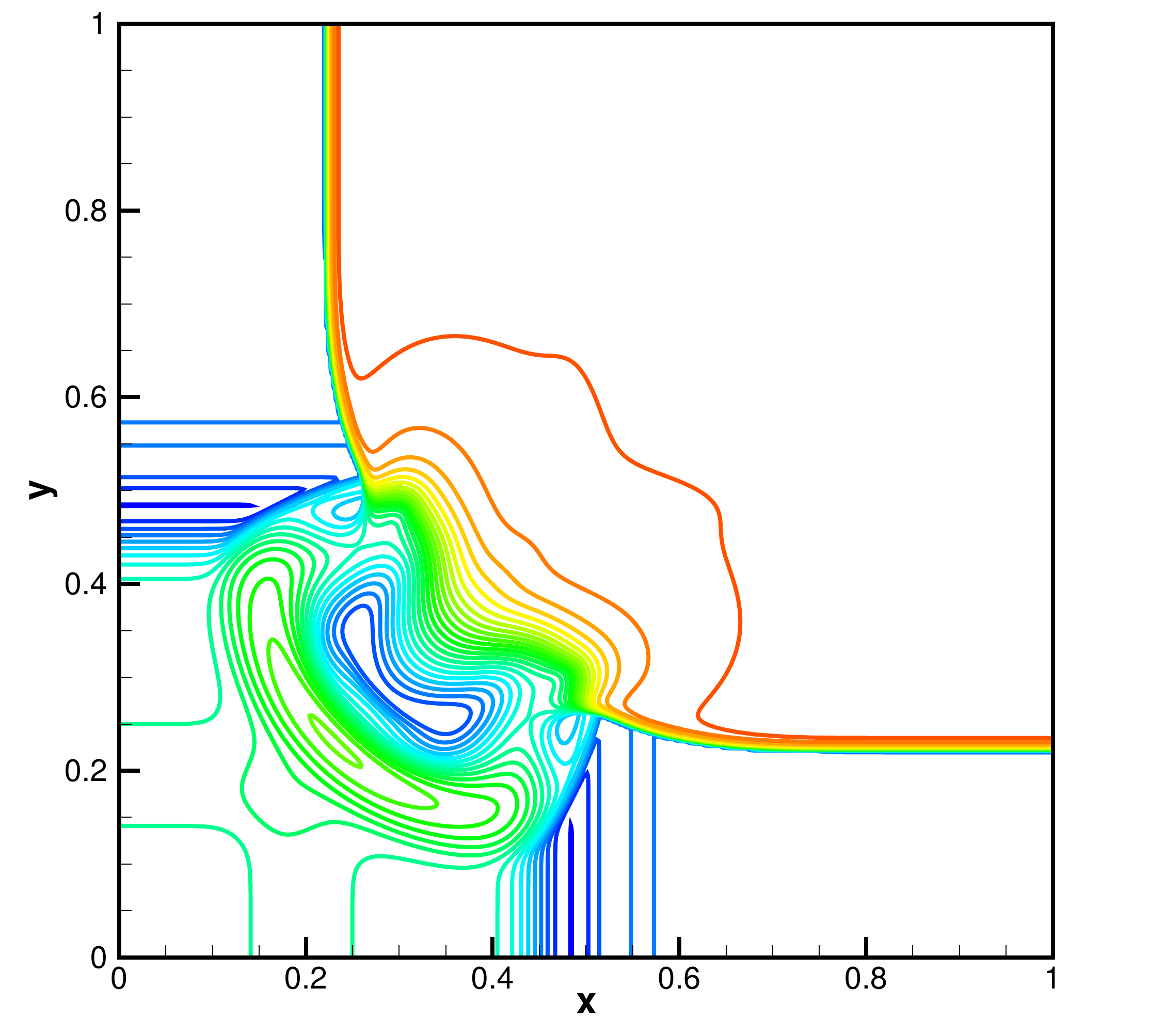}
\includegraphics[width=2.6in]{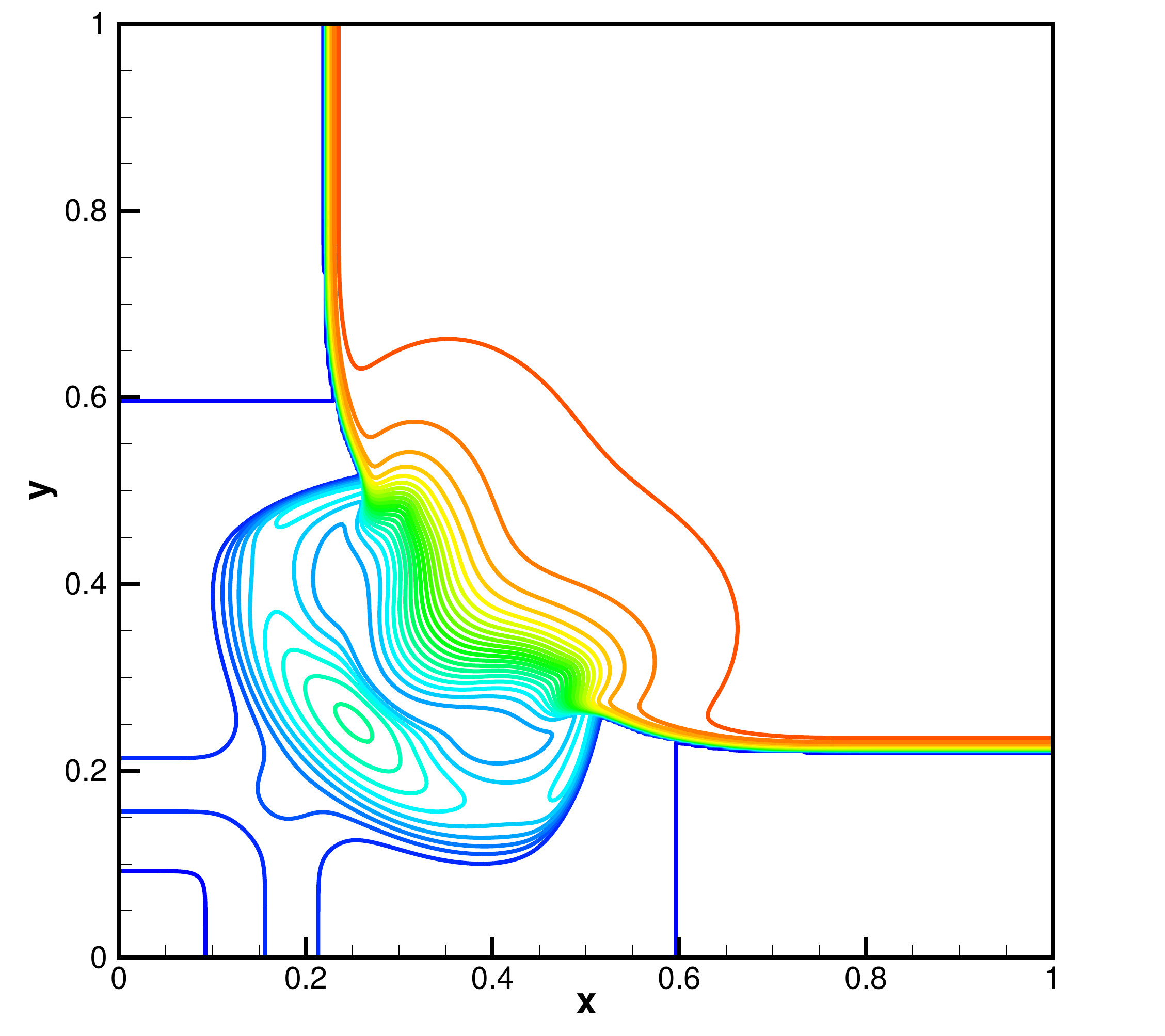}
\caption{\small Example \ref{exam5}: The contours of the density logarithm $\ln\rho$ (left)
		and the pressure
		logarithm $\ln p$ (right) at $t=0.4$ obtained from the first-order PCP scheme. 25 equally spaced contour lines are shown.}
\label{fig3}
\end{figure}
\begin{figure}[H]
	\centering
	\includegraphics[width=2.6in]{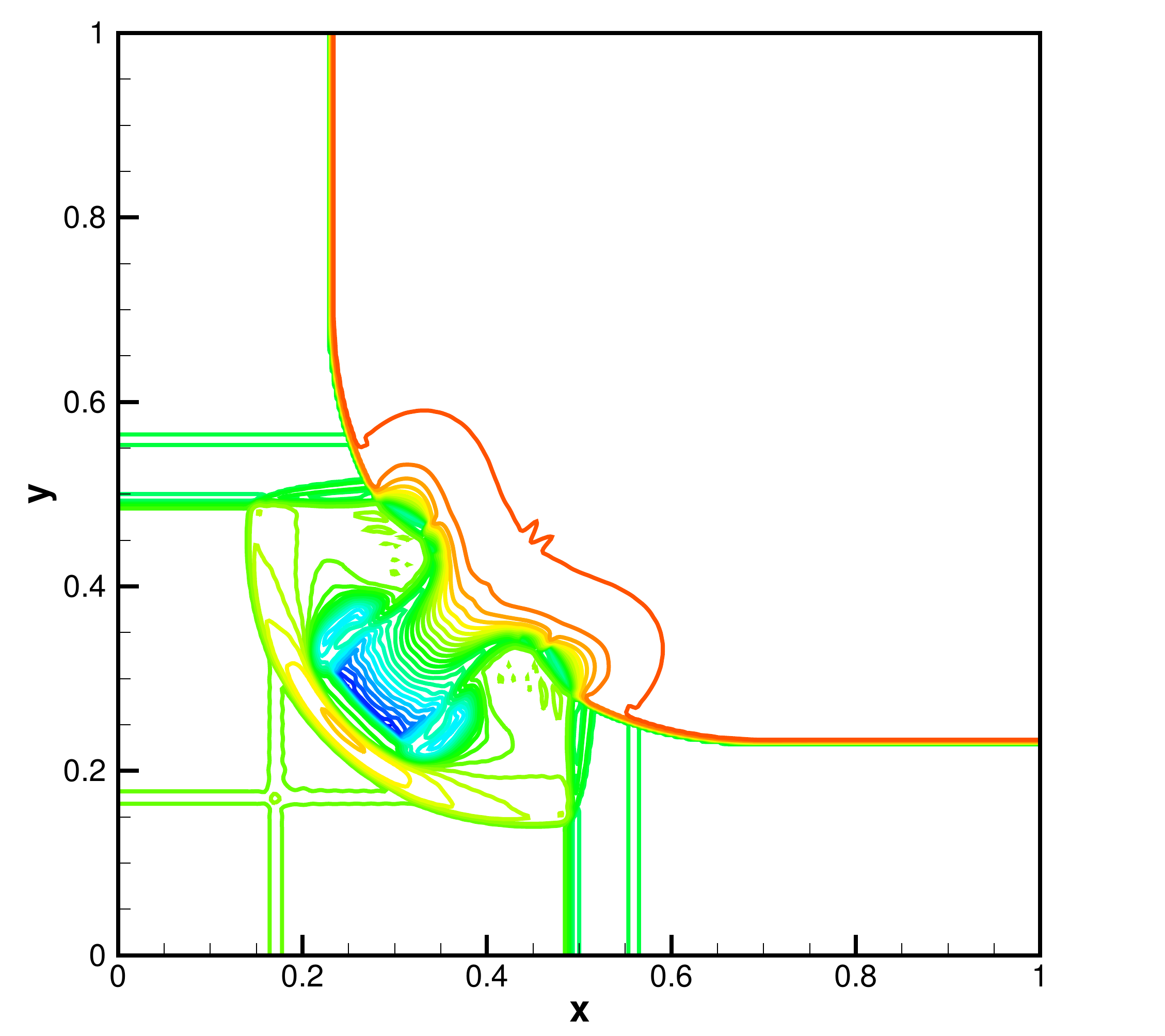}
	\includegraphics[width=2.6in]{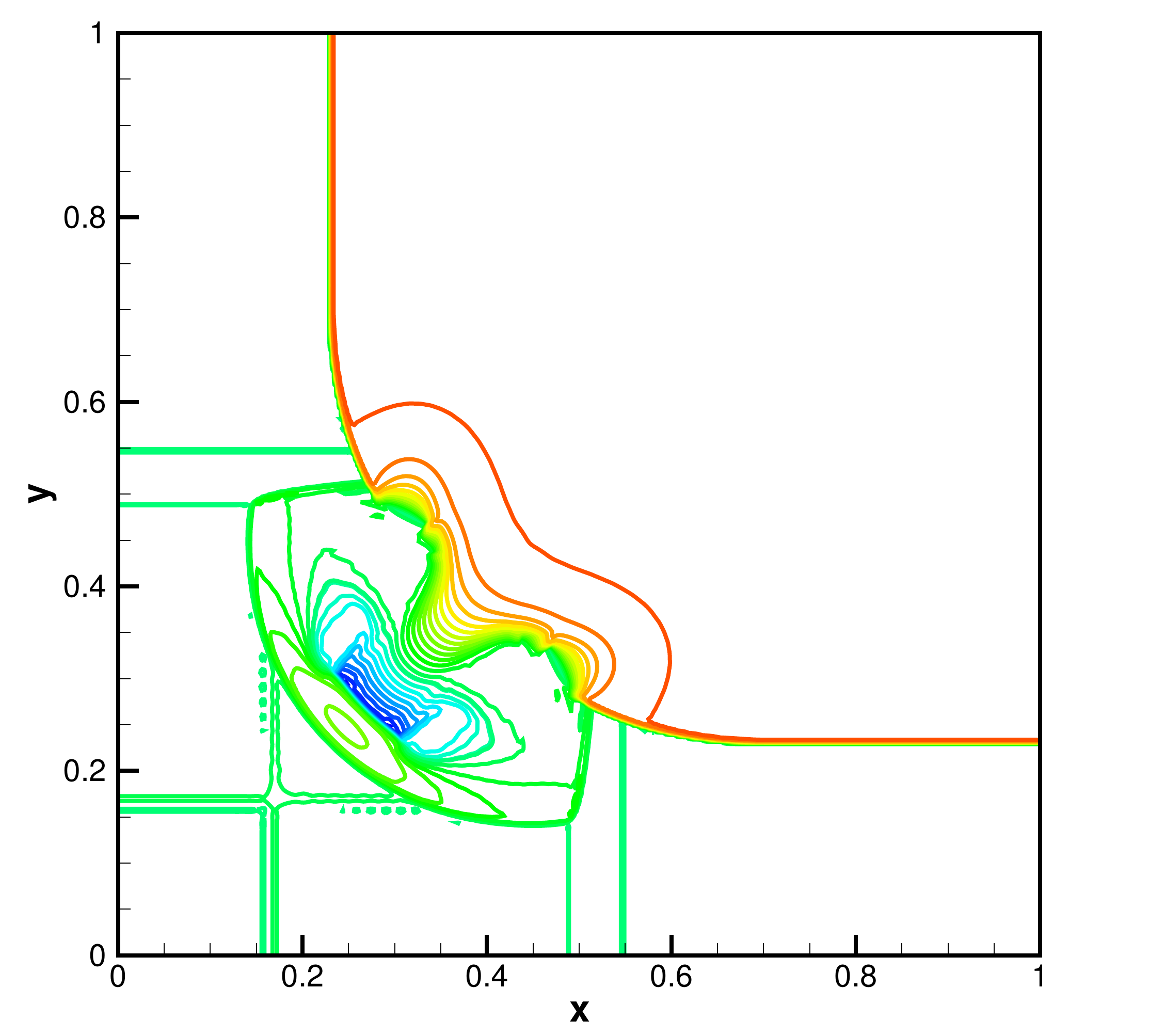}
	\caption{\small Example \ref{exam5}: Same as Figure \ref{fig3} except for the fifth-order PCP scheme.}
	\label{fig3add}
\end{figure}
\begin{figure}[H]
	\centering
	\includegraphics[width=2.6in]{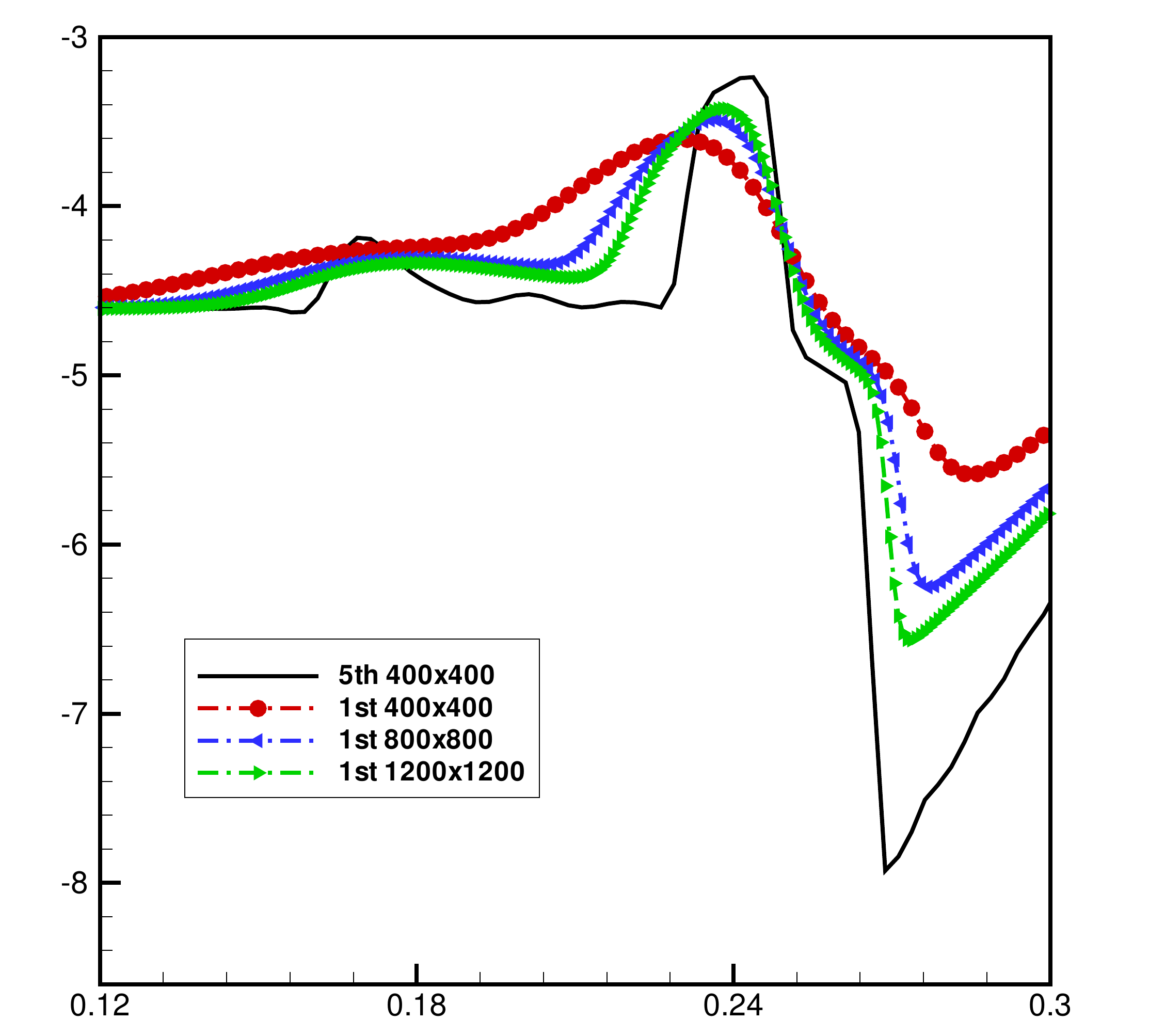}
	\includegraphics[width=2.6in]{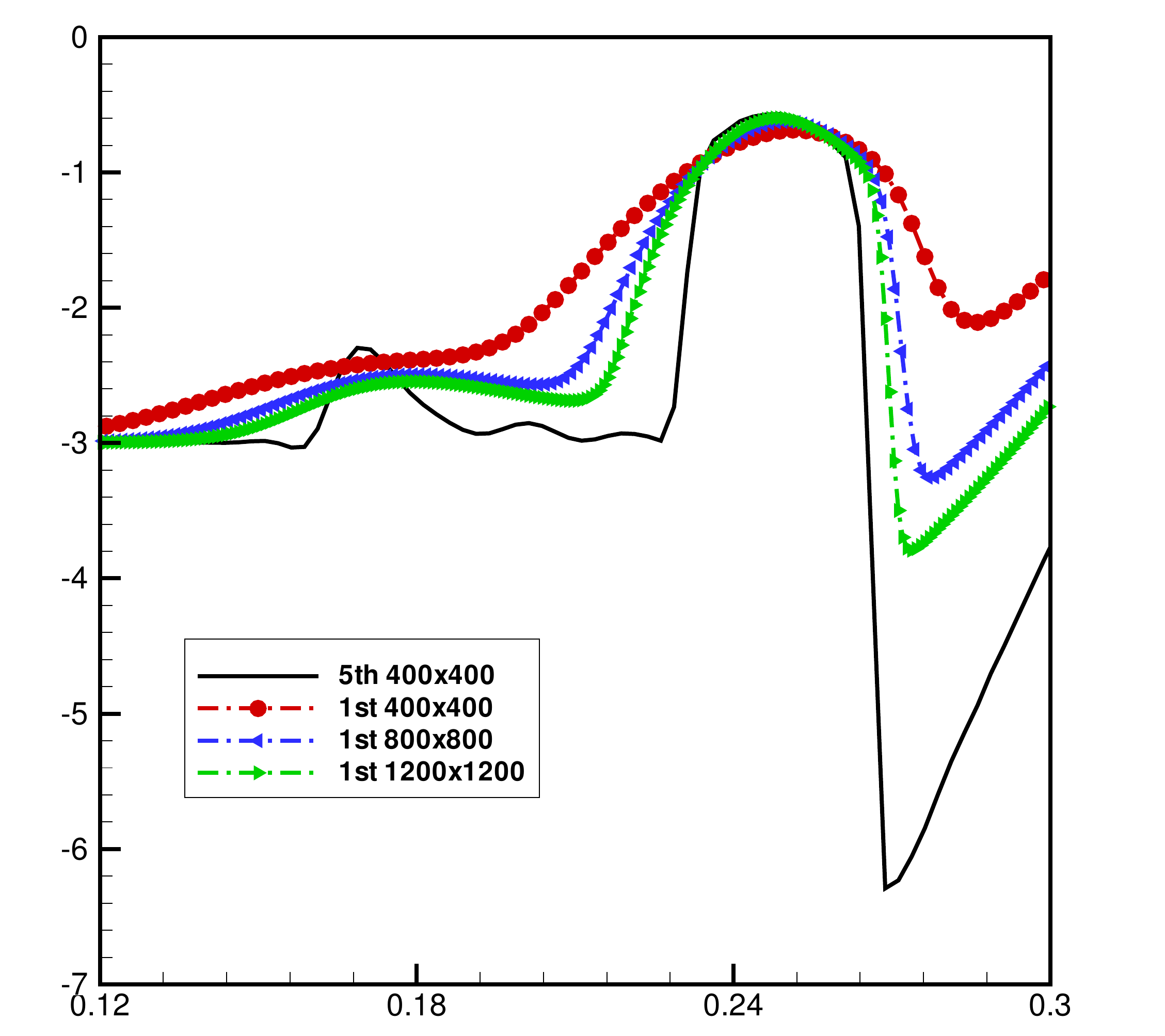}
	\caption{\small Example \ref{exam5}: Comparison of the cross sections of the numerical solutions at the line $y=x$ and $t=0.4$ for different meshes and schemes in the closed interval $x\in[0.12,0.3]$. Left: the rest-mass density logarithm $\ln\rho$; Right: the pressure logarithm $\ln p$.}
	\label{fig3add1}
\end{figure}
\begin{figure}[H]
	\centering
	\includegraphics[width=2.6in]{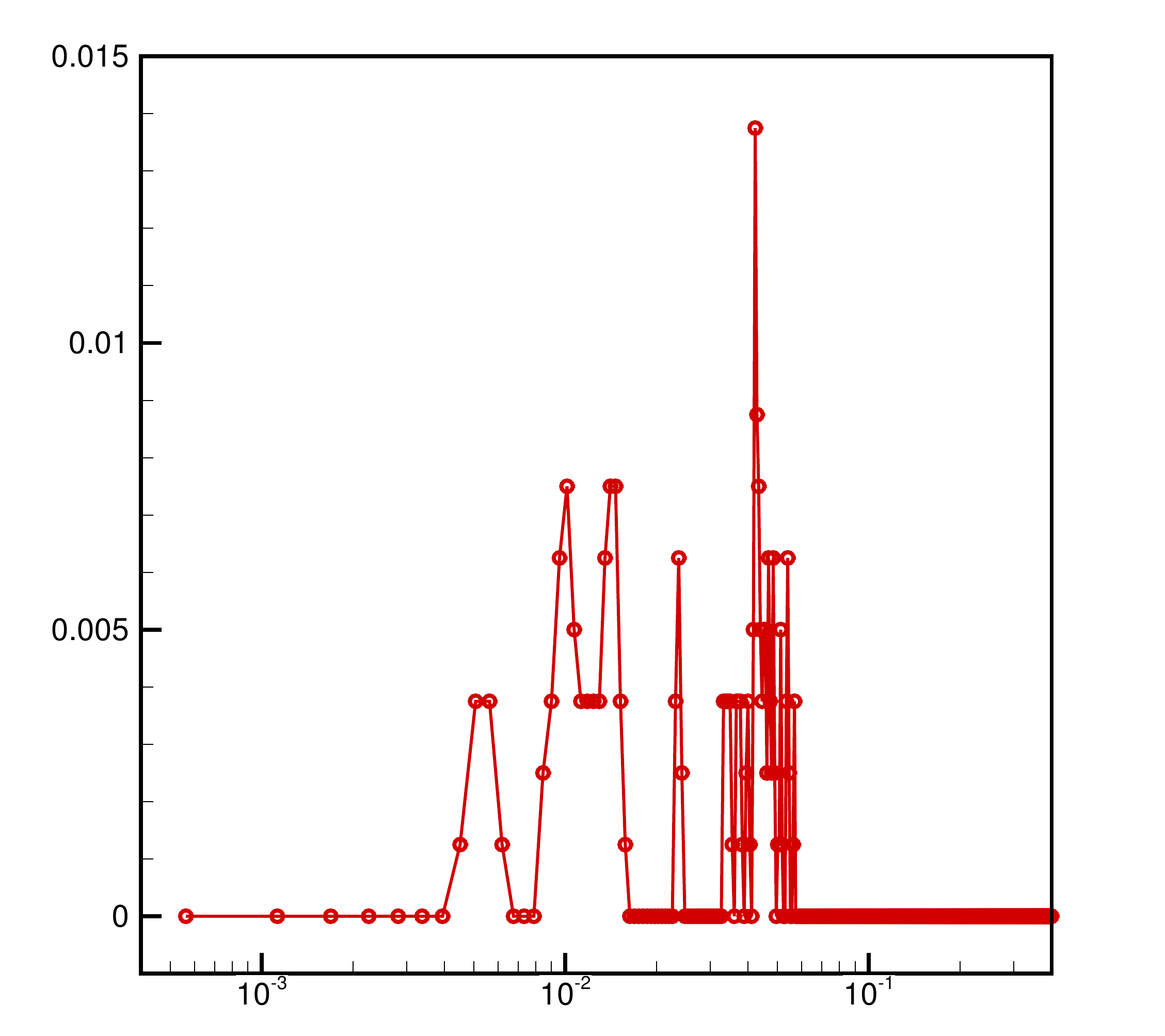}
	\includegraphics[width=2.6in]{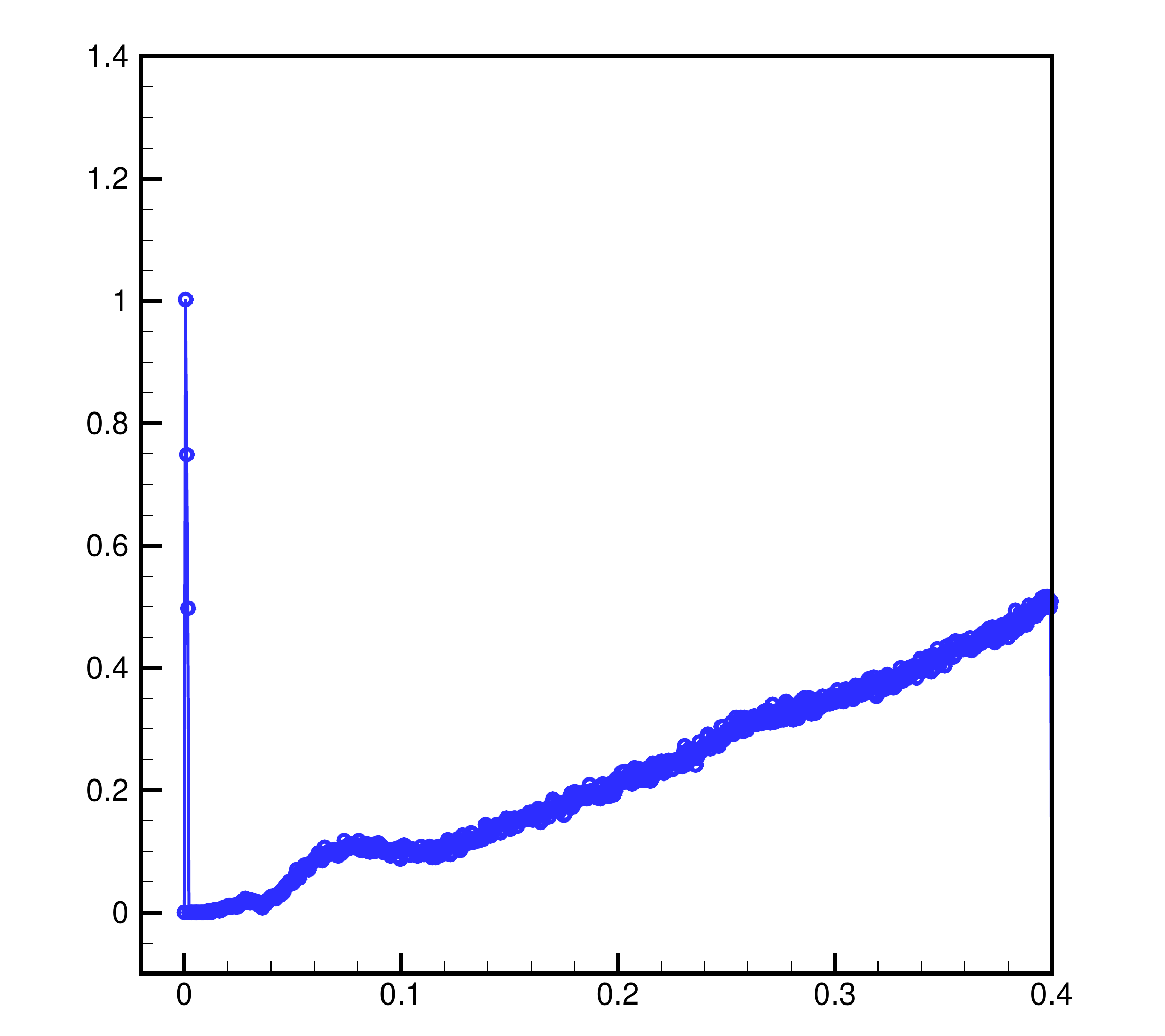}
	\caption{\small Example \ref{exam5}: Proportions of the PCP limited cells at each time level. Left: scaling PCP limiter; Right: PCP flux limiter.}
	\label{limiter-RP2}
\end{figure}

\begin{example}[Relativistic jets]\label{exam6} \rm
	The last example is to simulate the high-speed relativistic jet flows, which are ubiquitous in the extragalactic radio sources associated with the active galactic nuclei, and the most compelling case for a special relativistic phenomenon \cite{wu2017}.
	The simulation of such jet flows is full of challenge since there may appear the strong relativistic shock waves, shear waves, interface instabilities, and ultra-relativistic regions, as well as high speed
	jets etc.
	
	Here we consider a pressure-matched hot jet model, in which
	the relativistic effects from the large beam internal energies are
	important and comparable to the effects from the fluid velocity
	near the speed of light because the classical beam Mach number $M_b=1.72$ is near the minimum Mach number for given beam speed
	$v_b$. We remark that the data setting is the same as that in \cite{wu2017} but the EOS is different.
	Initially, the computational domain
	$[0,12]\times[0, 30]$ is filled with a static uniform medium with an
	unit rest-mass density, and a light relativistic jet is injected in the
	$y$-direction through the inlet part $|x|\le0.5$ on the bottom
	boundary $(y=0)$ with a high speed $v_b$, a rest-mass density of 0.01, and
	a pressure equal to the ambient pressure.
	The  fixed inflow beam condition
	is specified on the nozzle $\{y=0,|x|\le0.5\}$, the reflecting boundary
	condition is specified at $x=0$, whereas the
	outflow boundary conditions are on other boundaries.
	The following three different cases are considered:
	\begin{enumerate}[(\romannumeral1)]
		\item $v_b=0.99$, corresponding to the case of $\gamma\approx7.089$ and
		$M_r\approx9.971$.
		\item $v_b=0.999$, corresponding to the case of $\gamma\approx22.366$ and
		$M_r\approx31.316$.
		\item $v_b=0.9999$, corresponding to the case of $\gamma\approx70.712$ and
		$M_r\approx98.962$.
	\end{enumerate} 
	Here  $M_r:=M_b\gamma/\gamma_s$ denotes the relativistic Mach number
	with $\gamma_s=1/\sqrt{1-c_s^2}$ being the Lorentz factor associated with the
	local sound speed.
	
	As $v_b$ becomes much closer to the speed of light, the simulation of the jet becomes more challenging.
	Figures \ref{fig4}-\ref{fig7} display the schlieren images of the rest-mass density logarithm
	$\ln\rho$ and the
	pressure logarithm $\ln p$ within the domain $[-12,12]\times[0,30]$ at $t=30$ obtained by using
	the first- and the fifth-order schemes  on $240\times600$ uniform meshes
	for the computational domain $[0,12]\times[0,30]$. It is clear to observe that the high-order scheme can capture the
	beam interfaces much better than the first-order scheme.
\end{example}
	\begin{figure}[H]
		\centering
		\includegraphics[width=1.9in]{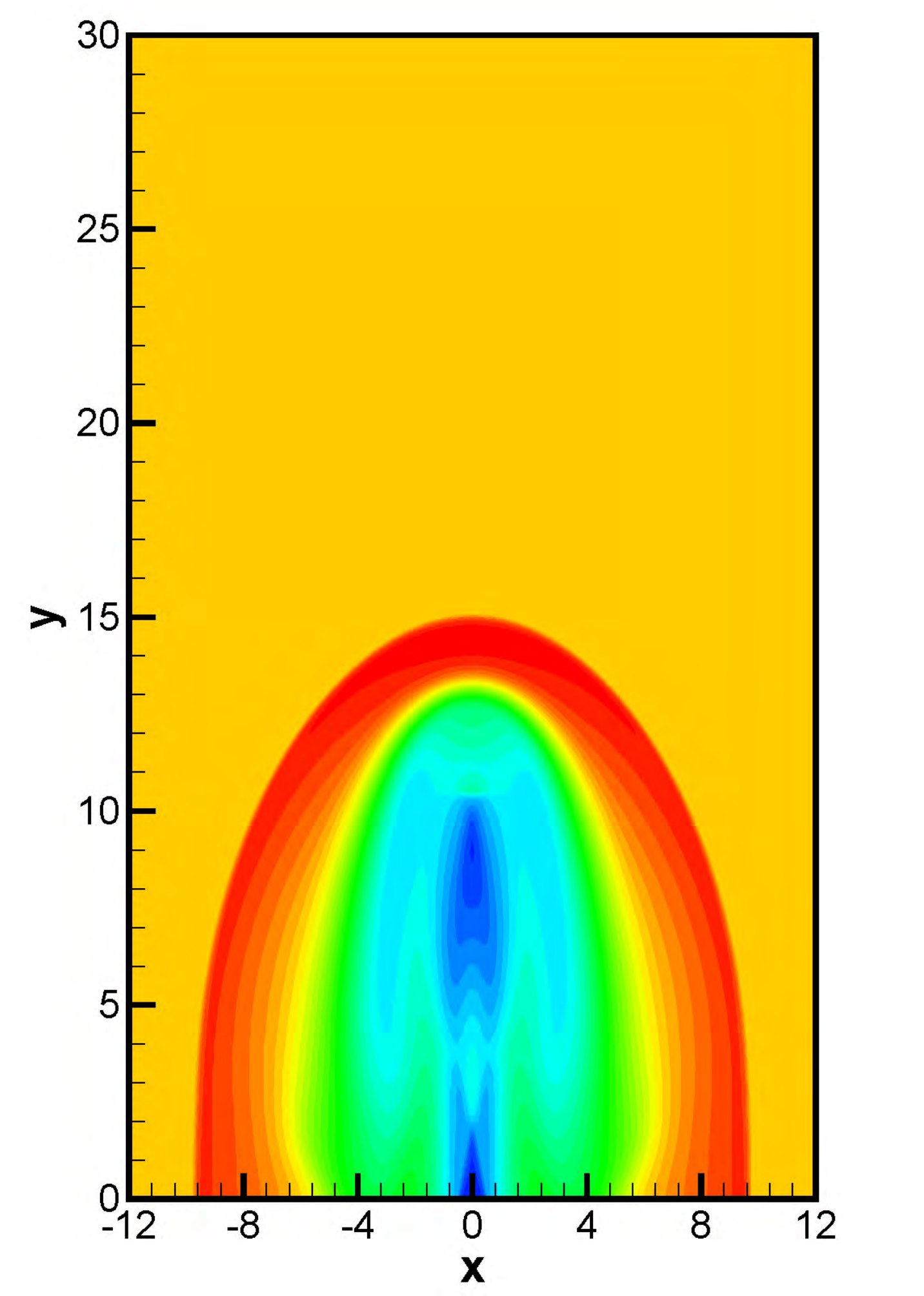}
		\includegraphics[width=1.9in]{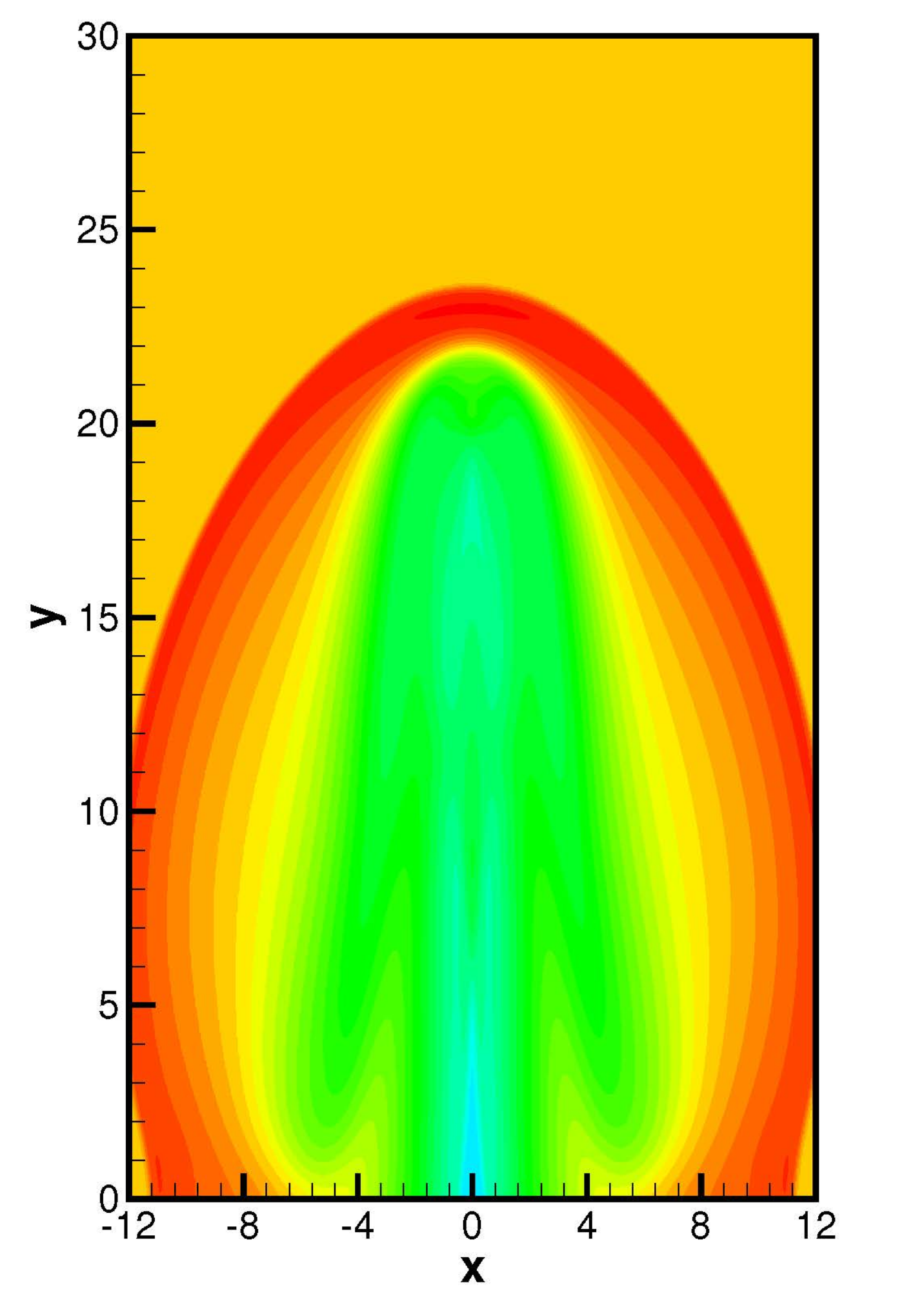}
		\includegraphics[width=1.9in]{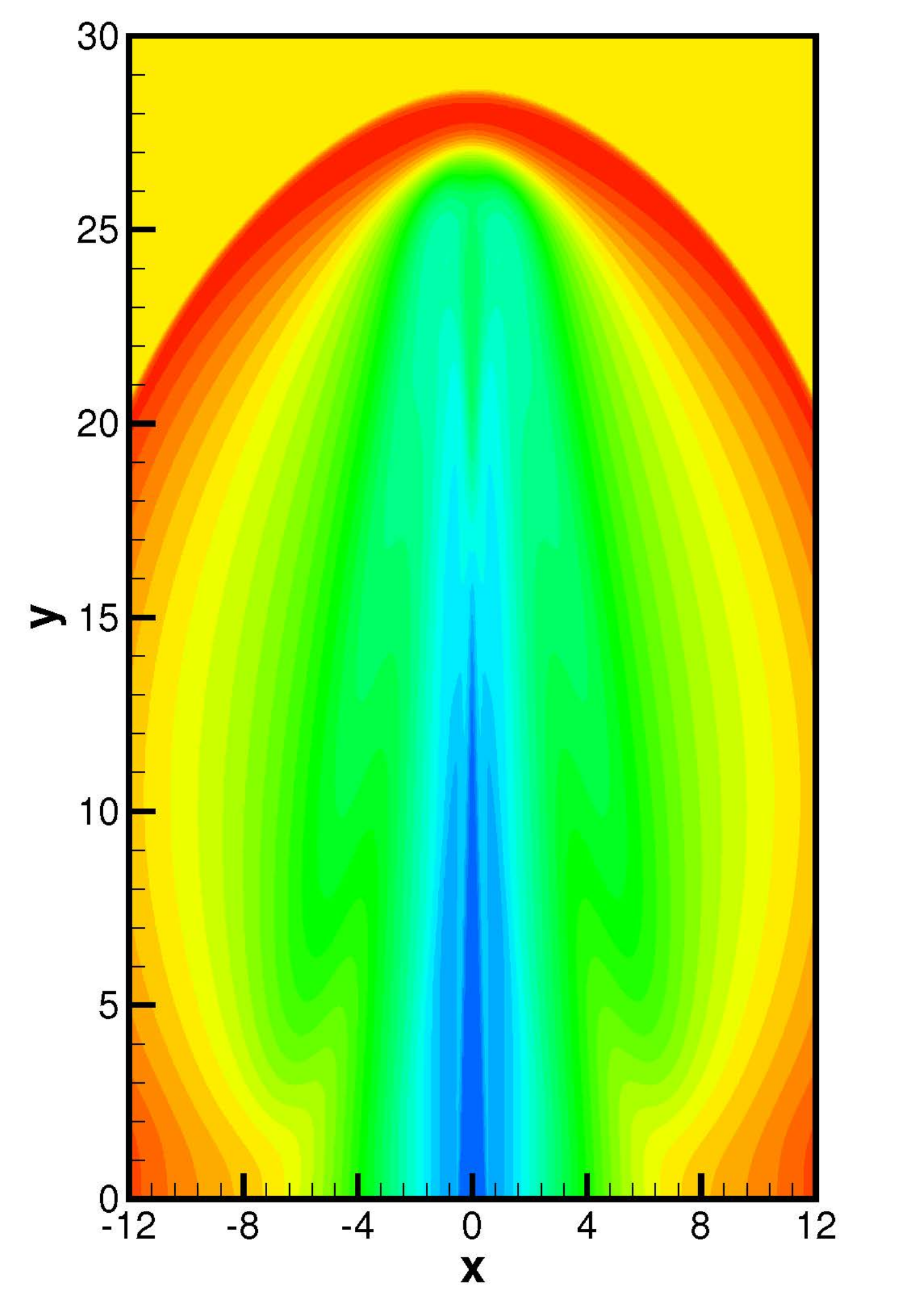}
		\caption{\small Example \ref{exam6}: Schlieren images of rest-mass density logarithm $\ln\rho$
			at $t=30$ for the hot jet model obtained by the first-order PCP scheme with $240\times600$ uniform cells. From left to right:
			configurations (\romannumeral1), (\romannumeral2) and (\romannumeral3).}
		\label{fig4}
	\end{figure}

	\begin{figure}[H]
	\centering
	\includegraphics[width=1.9in]{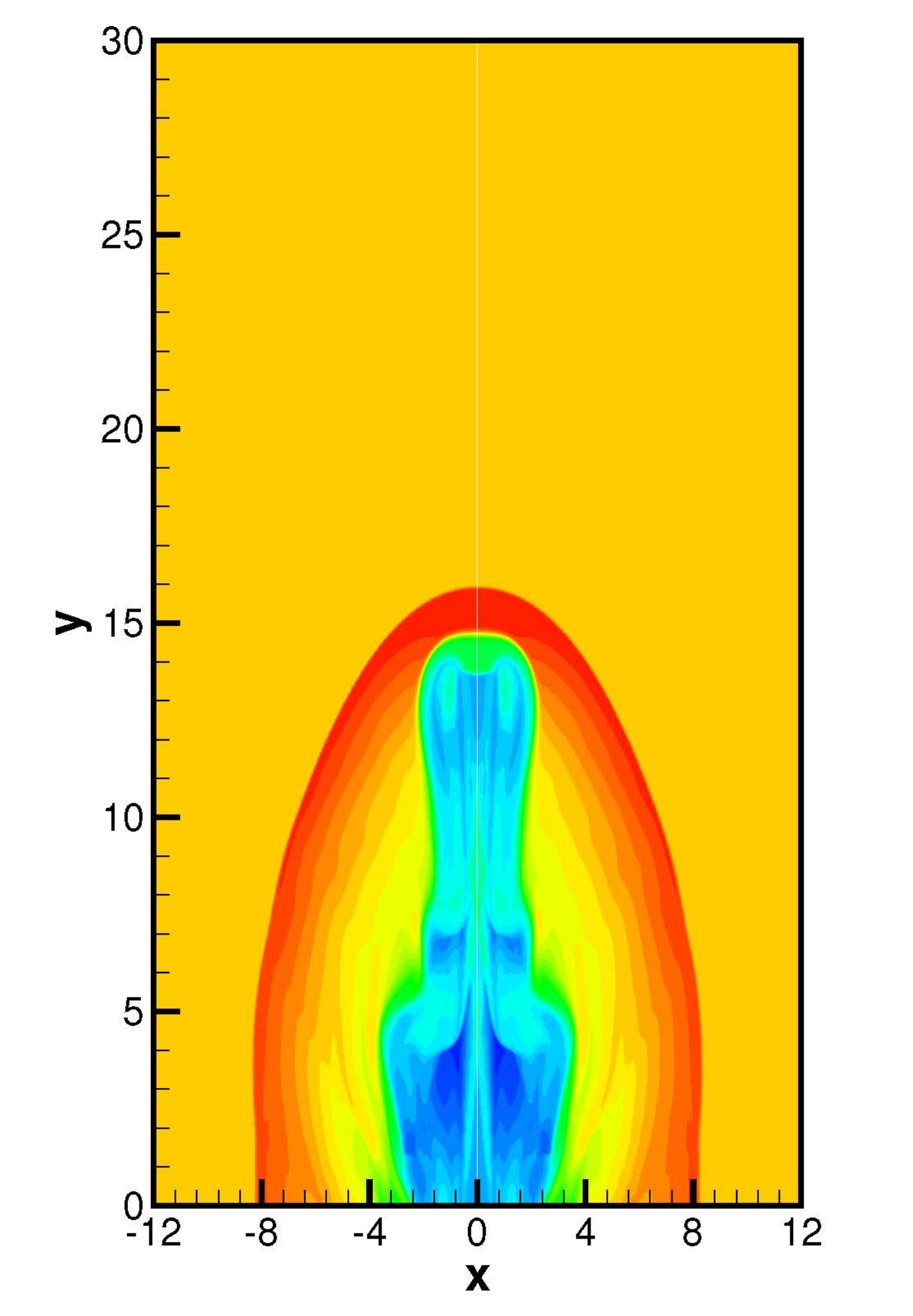}
	\includegraphics[width=1.9in]{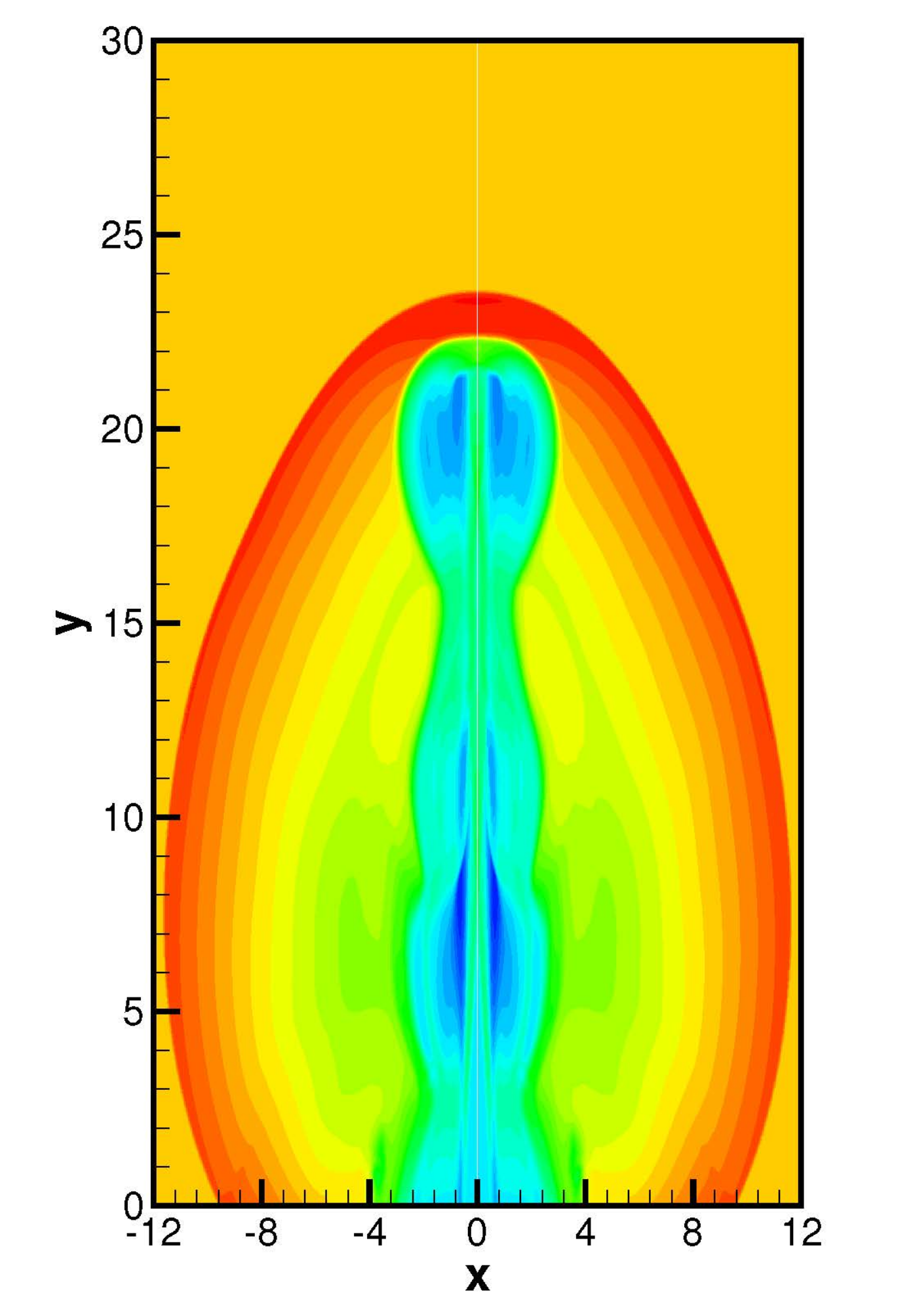}
	\includegraphics[width=1.9in]{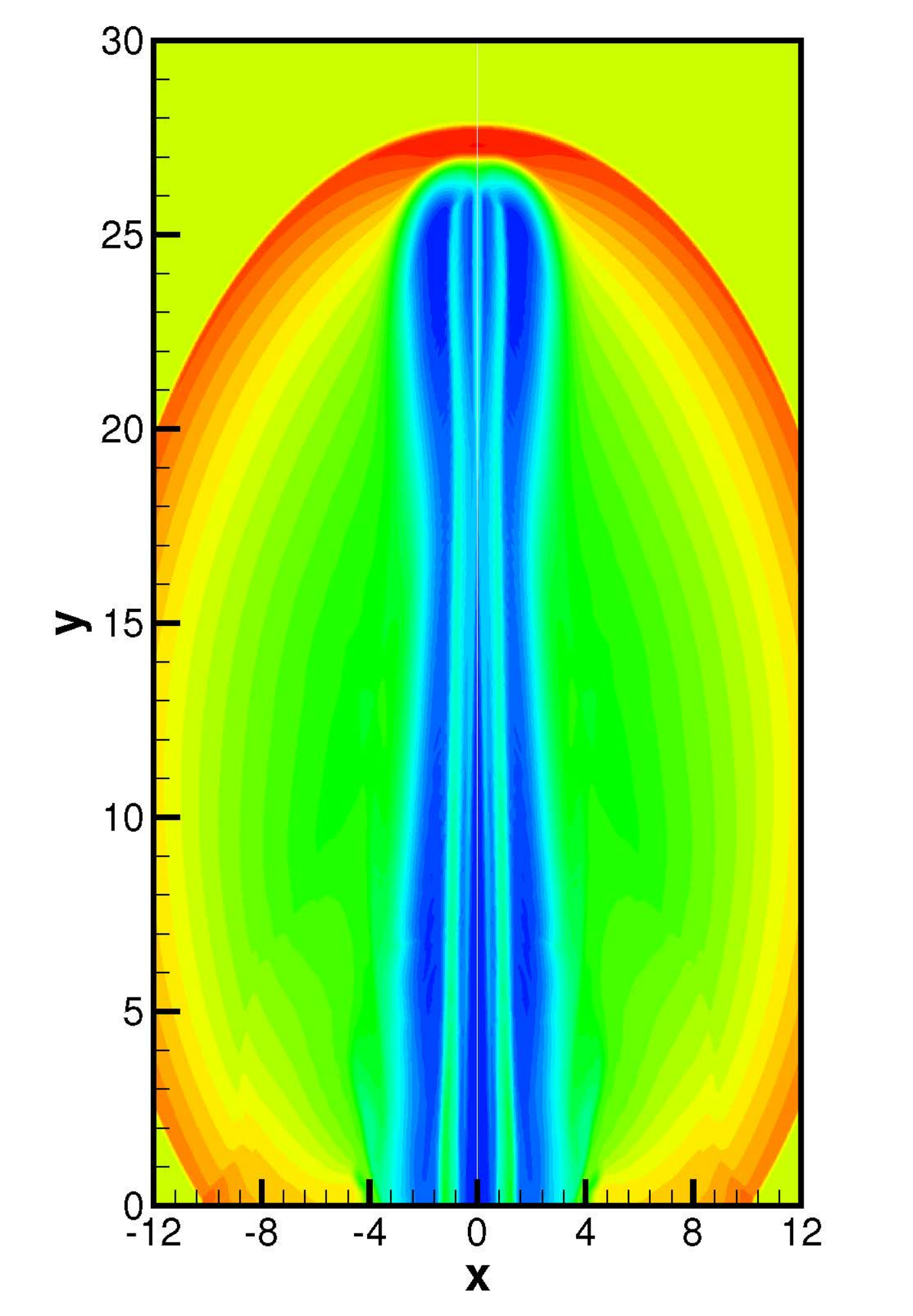}
	\caption{\small Example \ref{exam6}: Same as Figure \ref{fig4} except for the fifth-order PCP scheme.}
	\label{fig5}
	\end{figure}

	\begin{figure}[H]
	\centering
	\includegraphics[width=1.9in]{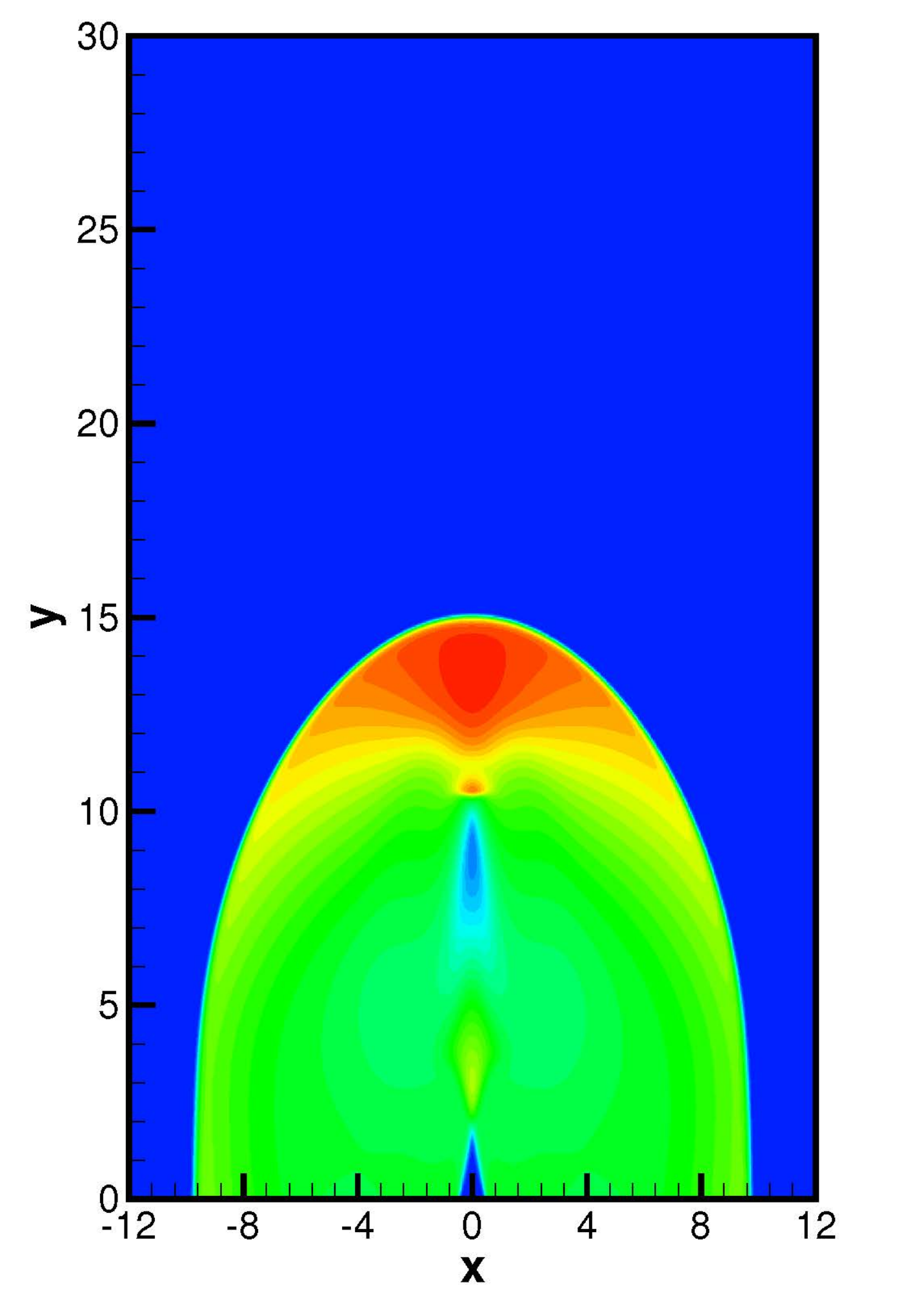}
	\includegraphics[width=1.9in]{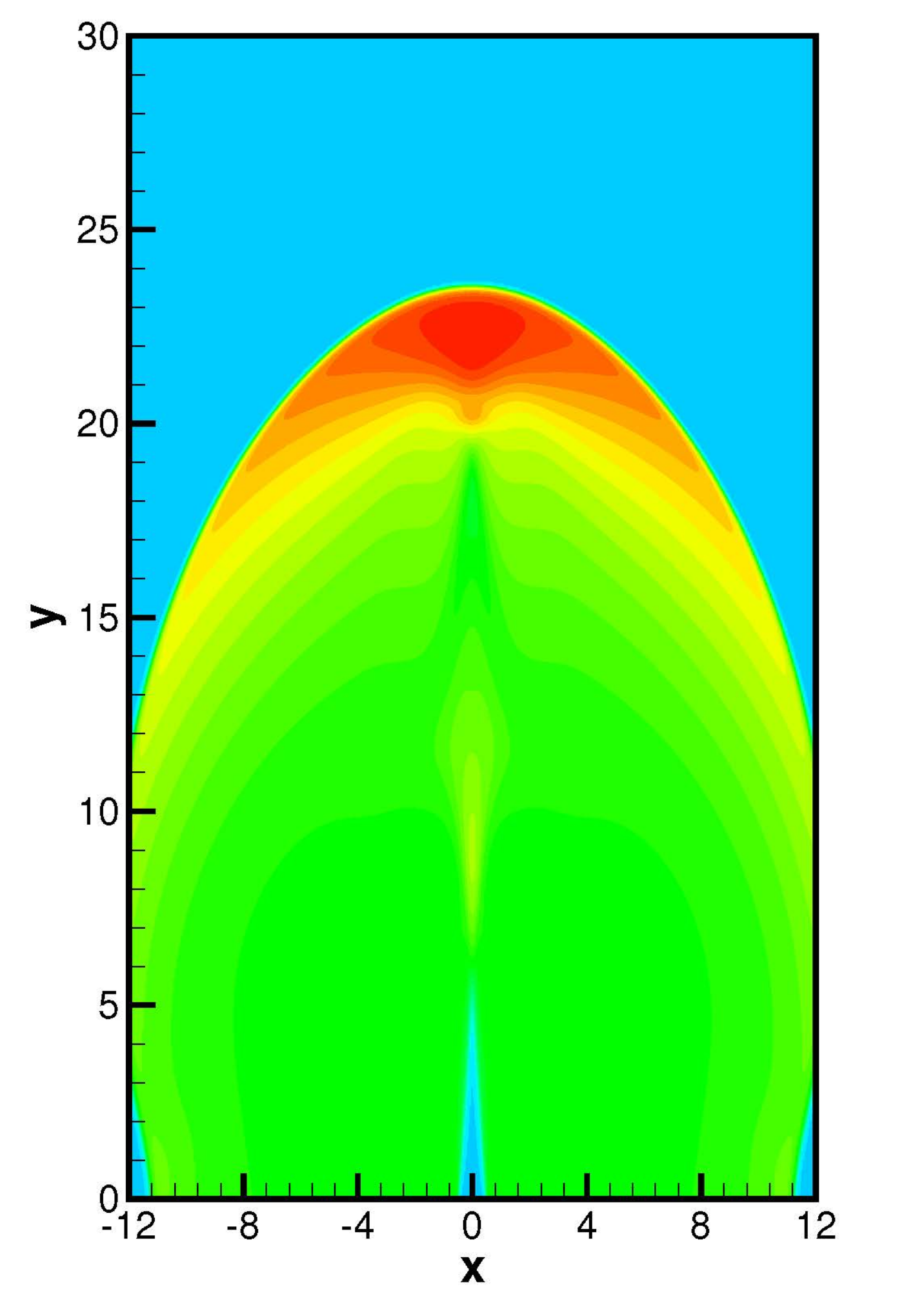}
	\includegraphics[width=1.9in]{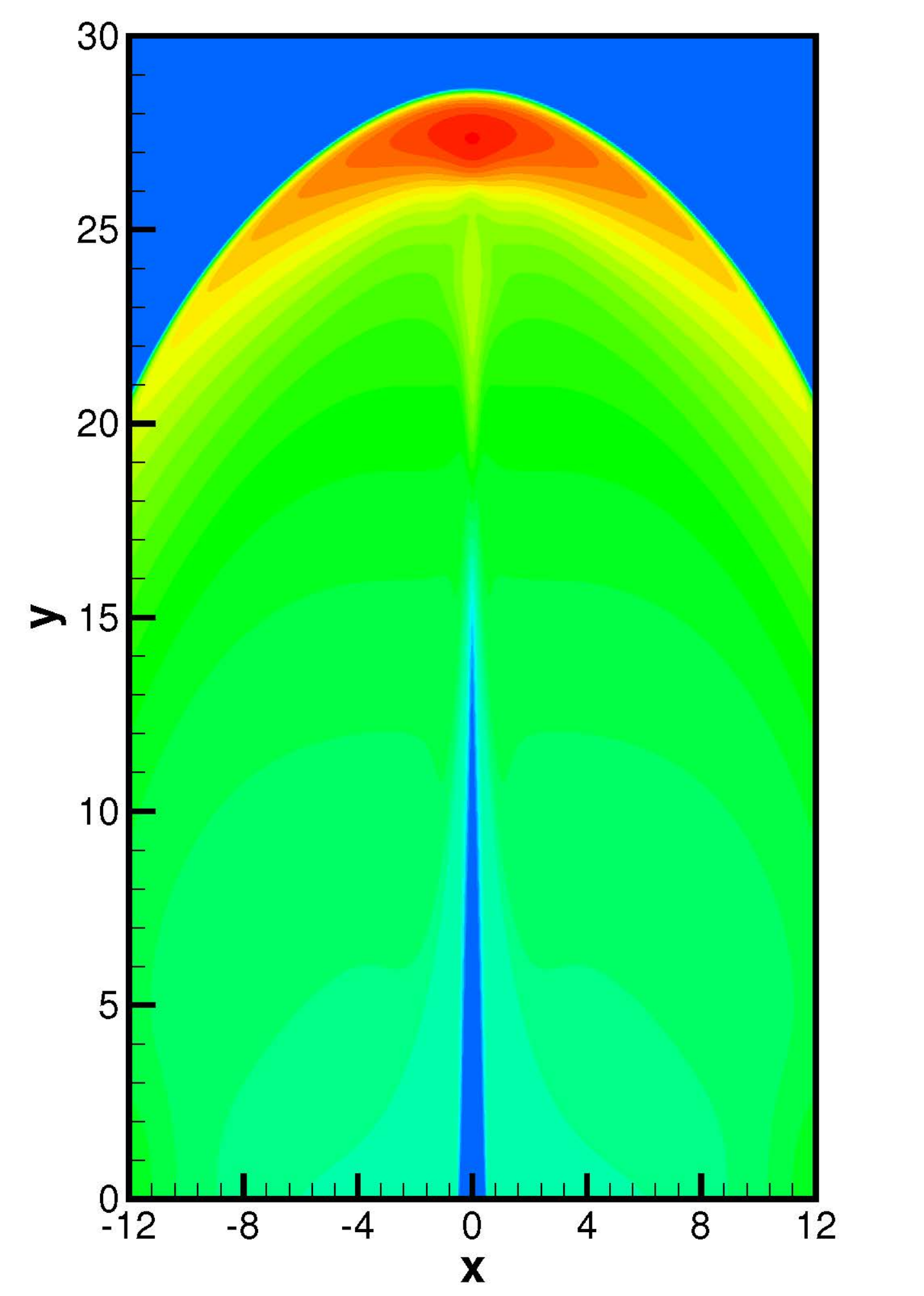}
	\caption{\small Example \ref{exam6}: Schlieren images of pressure logarithm $\ln p$
		at $t=30$ for the hot jet model obtained by the first-order PCP scheme with $240\times600$ uniform cells. From left to right:
		configurations (\romannumeral1), (\romannumeral2) and (\romannumeral3).}
	\label{fig6}
	\end{figure}

\begin{figure}[H]
	\centering
	\includegraphics[width=1.9in]{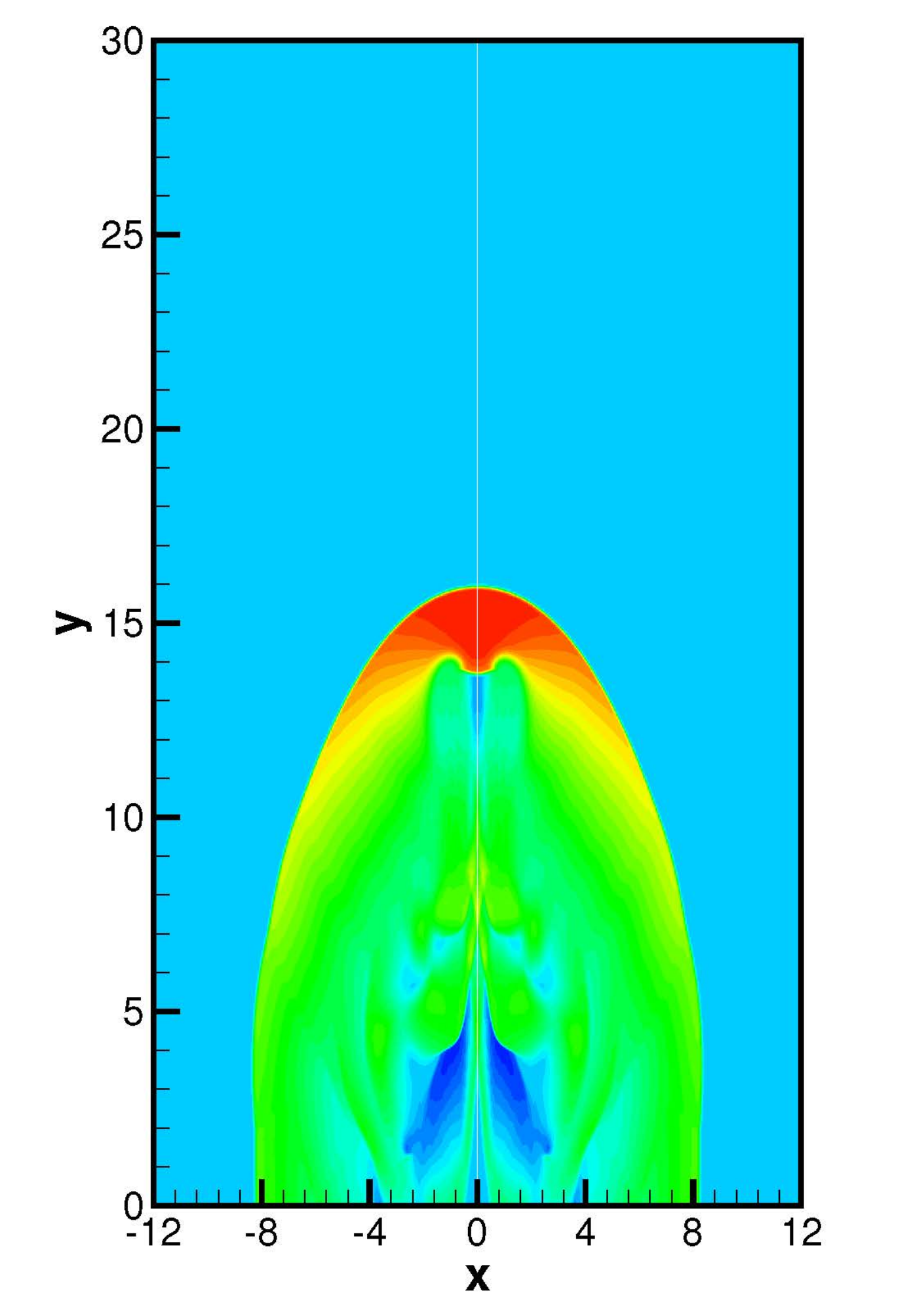}
	\includegraphics[width=1.9in]{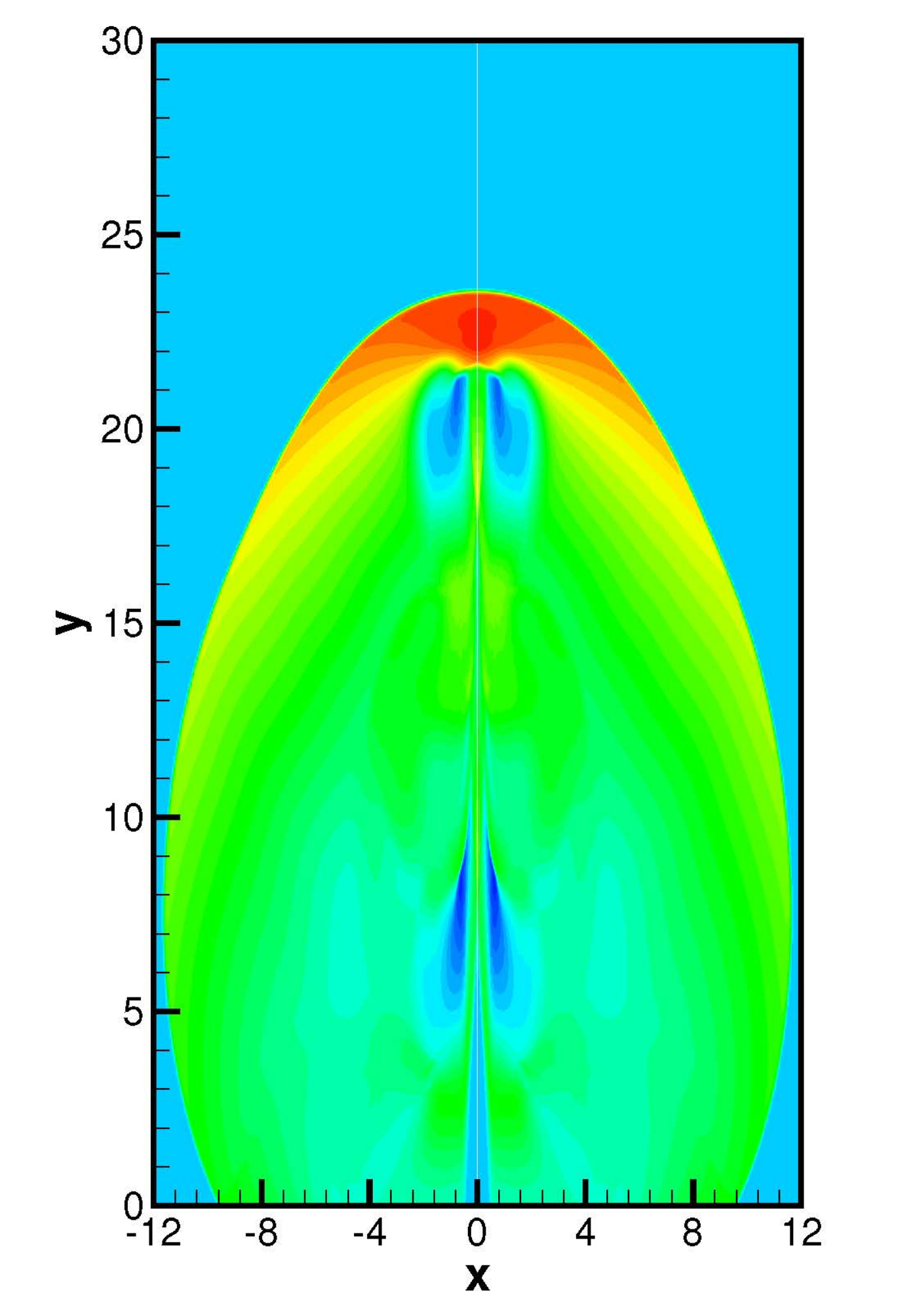}
	\includegraphics[width=1.9in]{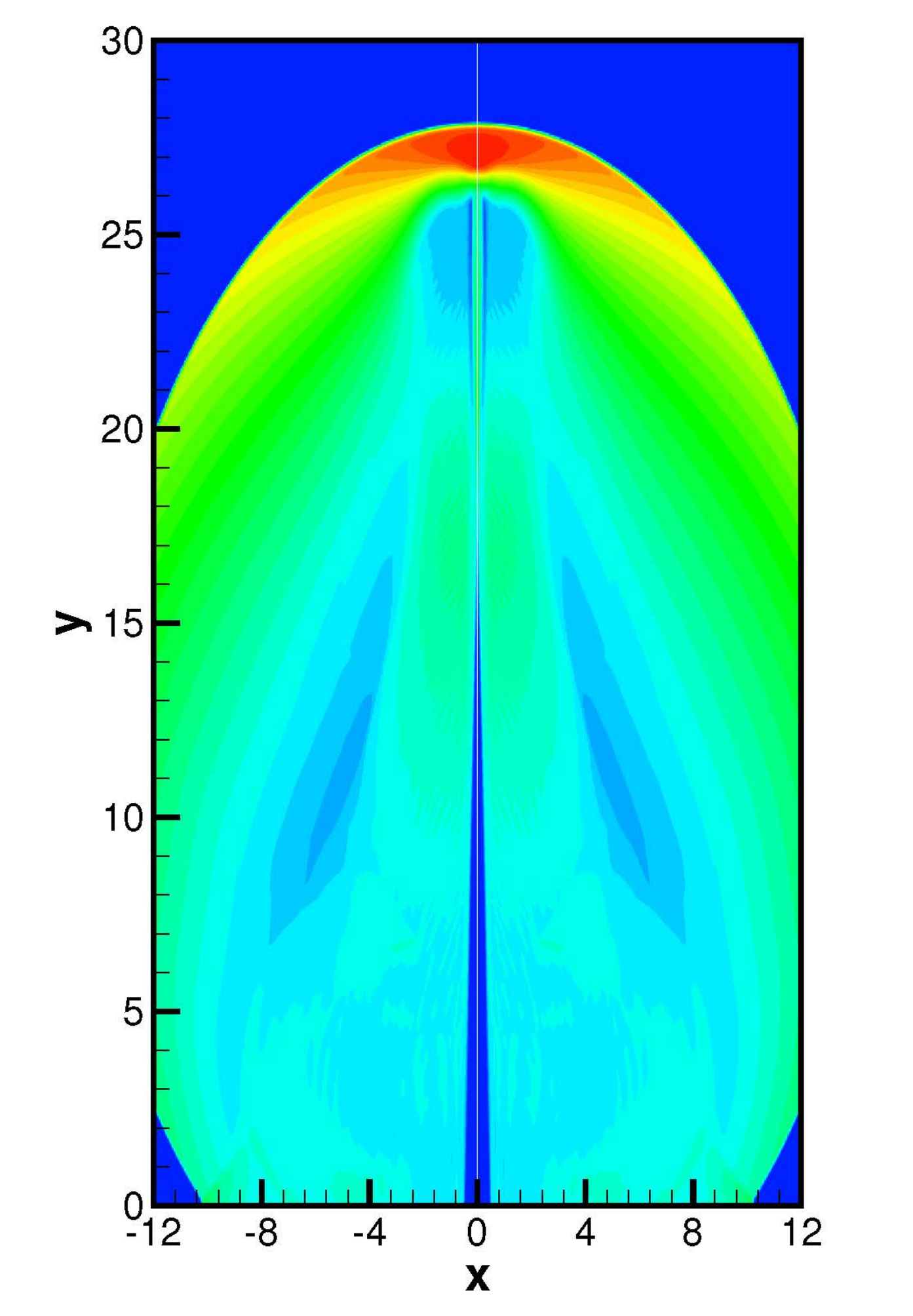}
	\caption{\small Example \ref{exam6}: Same as Figure \ref{fig6} except for the fifth-order PCP scheme.}
	\label{fig7}
\end{figure}

\section{Conclusion}\label{con}

This paper proposed a finite volume scheme based on the multidimensional HLL Riemann solver for the 2D special relativistic hydrodynamics and then 
studied its PCP property (i.e., preserving the positivity of the rest-mass density and the pressure and the boundness of the fluid velocity).
We first proved that the intermediate states in the multidimensional HLL Riemann solver were PCP when the HLL wave speeds were estimated suitably, and then showed the first-order accurate finite volume
scheme with the multidimensional HLL Riemann solver and forward Euler time discretization
was PCP.
Based on the resulting multidimensional HLL solver,  we developed the higher-order accurate PCP scheme by using the high-order accurate strong stability preserving (SSP) time 
discretization, the WENO reconstruction procedure, and the PCP flux limiter. Finally, several 2D numerical experiments were conducted to
demonstrate the accuracy and the effectiveness of the proposed PCP scheme in solving
the special RHD problems involving large Lorentz factor, or low
rest-mass density or low pressure or strong discontinuities, etc.

\section*{Acknowledgments}
The work was partially supported by the National Key R\&D Program of China (Project Number 2020YFA0712000). Moreover,
D. Ling would  like to acknowledge support by the National Natural Science Foundation
of China (Grant No. 12101486), the China Postdoctoral Science Foundation (Grant No. 2020M683446),
and the High-performance Computing Platform at Xi’an Jiaotong University;
H.Z. Tang would  like to acknowledge support    by  the National Natural
Science Foundation of China (Grant  No. 12171227 \& 12288101).

\begin{appendix}
	\renewcommand{\thesection}{Appendix \Alph{section}}
	\section{Proof of Lemma \ref{lem2}}\label{appendix}
	\renewcommand{\thesection}{\Alph{section}}
	This appendix provides a  proof of Lemma \ref{lem2}, 
	which is slightly different from that of  Lemma 2.3 in \cite{wu2015}. 
	Noting that the second and third properties in Lemma \ref{lem2} are  formally different from those in Lemma 2.3 of \cite{wu2015}.

	(\romannumeral1) For any positive number $\kappa$, let $(D^\kappa,\bm{m}^\kappa,E^\kappa )^T=\bm{U}^{\kappa}:=\kappa\bm{U}$. Since $\bm{U}\in\mathcal{G}$, it is easy to verify
	\begin{equation*}
		\begin{aligned}
			&D^{\kappa}=\kappa D>0,\ \
			&E^\kappa-\sqrt{(D^\kappa)^2+|\bm{m}^\kappa|^2}=\kappa\big(E-\sqrt{D^2+|\bm{m}|^2}\big)>0,
		\end{aligned}
	\end{equation*}
	which leads to admissibility of $\kappa\bm{U}$.
	
	(\romannumeral2) The convexity of $\mathcal{G}$ shows
	$$\frac{a_1}{a_1+a_2}\bm{U}_1+\frac{a_2}{a_1+a_2}\bm{U}_2\in\mathcal{G},$$
	for any $a_1,a_2>0$ and $\bm{U}_1,\bm{U}_2\in\mathcal{G}$.
	Combining it with the conclusion in
	(\romannumeral1) yields
	$$a_1\bm{U}_1+a_2\bm{U}_2\in\mathcal{G}.$$
	
	(\romannumeral3) For simplicity, denote
	$$\begin{aligned}
		(D^\alpha,\bm{m}_i^\alpha,E^\alpha)^T=\bm{U}^\alpha:
		&=\alpha\bm{U}-\bm{F}_i(\bm{U}),\\
		(D^\beta,\bm{m}_i^\beta,E^\beta)^T=\bm{U}^\beta:
		&=-\beta\bm{U}+\bm{F}_i(\bm{U}).
	\end{aligned}$$
	For the state $\bm{U}^\alpha$ with $\alpha\ge\lambda_i^{(4)}(\bm{U})$,
	we can get
	\begin{align*}
		D^\alpha&=D(\alpha-u_i)\ge D\big(\lambda_i^{(4)}(\bm{U})-u_i\big)>0,\\
		E^\alpha&=E(\alpha-u_i)-pu_i\ge E\big(\lambda_i^{(4)}(\bm{U})-u_i\big)-pu_i\\
		&=\frac{p\gamma^{2}}{c_s^2}\bigg(\big(\Gamma-c_s^2\gamma^{-2}\big)
		\frac{u_i(1-c_s^2)+c_s\gamma^{-1}\sqrt{1-u_i^2-c_s^2(|\bm{u}|^2-u_i^2)}}
		{1-c_s^2|\bm{u}|^2}-\Gamma u_i\bigg)\\
		&\ge\frac{p\gamma^{2}}{c_s^2}\bigg(\big(\Gamma-c_s^2\gamma^{-2}\big)
		\frac{u_i(1-c_s^{2})+c_s\gamma^{-2}}{1-c_s^{2}|\bm{u}|^2}-\Gamma u_i\bigg)\\
		&=\frac{p}{c_s(1-c_s^{2}|\bm{u}|^2)}\bigg(-c_su_i(\Gamma-c_s^2+1)
		+\Gamma-c_s^2\gamma^{-2}\bigg)\\
		&\ge\frac{p}{c_s(1-c_s^{2}|\bm{u}|^2)}\bigg(-c_s|
		\bm{u}|(\Gamma-c_s^2+1)+\Gamma-c_s^2\gamma^{-2}\bigg)\\
		&=\frac{p}{c_s(1+c_s|\bm{u}|)}\bigg(\Gamma-c_s^2-c_s|\bm{u}|\bigg)>0,
	\end{align*}
	and
	\begin{align*}
		(E^\alpha)^2-|\bm{m}^\alpha|^2-(D^\alpha)^2&=(E^2-|\bm{m}|^2-D^2-p^2)
		(\alpha-u_i)^2+p^2(\alpha^2-1)\\
		&=\frac{\Gamma p^2}{c_s^2}\gamma^{2}\bigg(\frac{2}{\Gamma-1}-\frac{\Gamma c_s^2}{(\Gamma-1)^2}\bigg)
		(\alpha-u_i)^2+p^2(\alpha^2-1)\\
		&=p^2\cdot f(\alpha),
	\end{align*}
	where $f(s)$ is a quadratic function of $s\in[\lambda_i^{(4)}(\bm{U}),1)$ with the form of
	\begin{equation*}
		f(s)=\frac{\Gamma\gamma^{2}}{c_s^2}\bigg(\frac{2}{\Gamma-1}-\frac{\Gamma c_s^2}{(\Gamma-1)^2}\bigg)
		(s-u_i)^2+s^2-1.
	\end{equation*}
	It is easy to prove that $f(s)$ is monotonically increasing with $s\in[\lambda_i^{(4)}(\bm{U}),1)$,
	so that
	$f(s)\ge f(\lambda_i^{(4)}(\bm{U}))$ for any $s\in[\lambda_i^{(4)}(\bm{U}),1)$ and then
	$f(\alpha)\ge f(\lambda_i^{(4)}(\bm{U}))$. Moreover, we   have
	\begin{equation*}
		\begin{aligned}
			f(\lambda_i^{(4)}(\bm{U}))&=\frac{2\Gamma(\Gamma-1)-\Gamma^2c_s^2}{(\Gamma-1)^2(1-c_s^2|\bm{u}|^2)^2}
			\bigg(-\frac{c_su_i}{\gamma}+\sqrt{1-u_i^{2}-c_s^{2}(|\bm{u}|^2-u_i^2)}\bigg)^2\\
			&~~~+\frac{\bigg(u_i(1-c_s^2)+c_s\gamma^{-1}\sqrt{1-u_i^{2}-c_s^{2}(|\bm{u}|^2-u_i^2)}\bigg)^2}
			{(1-c_s^2|\bm{u}|^2)^2}-1\\
			&=C_1\bigg(c_su_i\gamma^{-1}-\sqrt{1-u_i^{2}-c_s^{2}(|\bm{u}|^2-u_i^2)}\bigg)^2\ge0,
		\end{aligned}
	\end{equation*}
	with
	$$C_1=\frac{1}{(1-c_s^2|\bm{u}|^2)^2}\bigg(\Gamma^2-1+c_s^2(1-2\Gamma)\bigg)>0.$$
	Therefore,   $f(\alpha)>0$ and then $(E^\alpha)^2-|\bm{m}^\alpha|^2-(D^\alpha)^2>0$. So far,
	we have proved the conclusion $\alpha\bm{U}-\bm{F}(\bm{U})\in \mathcal{G}$ for $\alpha\ge\lambda_i^{(4)}(\bm{U})$.
	
	For the state  $\bm{U}^\beta$ with $\beta\le\lambda_i^{(1)}(\bm{U})$,
	one can similarly have
	\begin{align*}
		D^\beta&=D(u_i-\beta)\ge D\big(u_i-\lambda_i^{(1)}(\bm{U})\big)>0,\\
		E^\beta&=E(u_i-\beta)+pu_i\ge E\big(u_i-\lambda_i^{(1)}(\bm{U})\big)+pu_i\\
		&=\frac{p\gamma^{2}}{c_s^2}\bigg(-\big(\Gamma-c_s^2\gamma^{-2}\big)
		\frac{u_i(1-c_s^{2})-c_s\gamma^{-1}\sqrt{1-u_i^2-c_s^2(|\bm{u}|^2-u_i^2)}}
		{1-c_s^{2}|\bm{u}|^2}+\Gamma u_i\bigg)\\
		&\ge\frac{p\gamma^{2}}{c_s^2}\bigg(-\big(\Gamma-c_s^2\gamma^{-2}\big)
		\frac{u_i(1-c_s^{2})-c_s\gamma^{-2}}{1-c_s^{2}|\bm{u}|^2}+\Gamma u_i\bigg)\\
		&=\frac{p}{c_s(1-c_s^{2}|\bm{u}|^2)}\bigg(c_su_i(\Gamma-c_s^2+1)+\Gamma-c_s^2\gamma^{-2}\bigg)\\
		&\ge\frac{p}{c_s(1-c_s^{2}|\bm{u}|^2)}\bigg(-c_s|\bm{u}|(\Gamma-c_s^2+1)+\Gamma-c_s^2\gamma^{-2}\bigg)\\
		&=\frac{p}{c_s(1+c_s|\bm{u}|)}\bigg(\Gamma-c_s^2-c_s|\bm{u}|\bigg)>0,
	\end{align*}
	and
	\begin{align*}
		(E^\beta)^2-|\bm{m}^\beta|^2-(D^\beta)^2&=(E^2-|\bm{m}|^2-D^2-p^2)(\beta-u_i)^2+p^2(\beta^2-1)\\
		&=\frac{\Gamma p^2}{c_s^2}\gamma^{2}\bigg(\frac{2}{\Gamma-1}-\frac{\Gamma c_s^2}{(\Gamma-1)^2}\bigg)
		(\beta-u_i)^2+p^2(\beta^2-1)\\
		&=p^2\cdot g(\beta),
	\end{align*}
	where $g(s)$ is a quadratic function of $s\in(-1,\lambda_i^{(1)}(\bm{U})]$ with the form of
	\begin{equation*}
		g(s)=\frac{\Gamma\gamma^{2}}{c_s^2}\bigg(\frac{2}{\Gamma-1}-\frac{\Gamma c_s^2}{(\Gamma-1)^2}\bigg)
		(s-u_i)^2+s^2-1.
	\end{equation*}
	It is easy to  prove that $g(s)$ is monotonically decreasing with $s\in(-1,\lambda_i^{(1)}(\bm{U})]$,
	so that $g(s)\ge g(\lambda_i^{(1)}(\bm{U}))$ for any $s\in(-1,\lambda_i^{(1)}(\bm{U})]$ and then
	$g(\beta)\ge g(\lambda_i^{(1)}(\bm{U}))$. Moreover, we can show
	\begin{align*}
		g(\lambda_i^{(1)}(\bm{U}))&=\frac{2\Gamma(\Gamma-1)-\Gamma^2c_s^2}{(\Gamma-1)^2(1-c_s^2|\bm{u}|^2)^2}
		\bigg(c_su_i\gamma^{-1}+\sqrt{1-u^{2}-c_s^{2}(|\bm{u}|^2-u_i^2)}\bigg)^2\\
		&~~~+\frac{\bigg(u_i(1-c_s^2)-c_s\gamma^{-1}\sqrt{1-u_i^{2}-c_s^{2}(|\bm{u}|^2-u_i^2)}\bigg)^2}
		{(1-c_s^2|\bm{u}|^2)^2}-1\\
		&=C_2\bigg(c_su_i\gamma^{-1}+\sqrt{1-u_i^{2}-c_s^{2}(|\bm{u}|^2-u_i^2)}\bigg)^2\ge0,
	\end{align*}
	with
	$$C_2=\frac{1}{(1-c_s^2|\bm{u}|^2)^2}\bigg(\Gamma^2-1+c_s^2(1-2\Gamma)\bigg)>0.$$
	Therefore,  $g(\beta)>0$ and then $(E^\beta)^2-|\bm{m}^\beta|^2-(D^\beta)^2>0$, which leads to
	$-\beta\bm{U}+\bm{F}_i(\bm{U})\in \mathcal{G}$ for $\beta\le\lambda_i^{(1)}(\bm{U})$.
	\qed
\end{appendix}


\begin{thebibliography}{}
	
	\bibitem{abgrall}R. Abgrall, A genuinely multidimensional Riemann solver, \emph{Research Report},
	RR-1859, 1993 (https://hal.inria.fr/inria-00074814).
	
	
	\bibitem{barsara}D.S. Balsara, Multidimensional HLLE Riemann solver: Application to Euler and
	magnetohydrodynamic flows, \emph{J. Comput. Phys.}, 229 (2010) 1970-1993.
	
	
	\bibitem{barsara2}D.S. Balsara, A two-dimensional HLLC Riemann solver for conservation laws:
	Application to Euler and magnetohydrodynamic flow, \emph{J. Comput. Phys.}, 231 (2012) 7476-7503.
	
	\bibitem{barsara3}D.S. Balsara, M. Dumbser and R. Abgrall, A multidimensional HLLC Riemann solver for unstructured meshes-With application to Euler and MHD flows,
	\emph{J. Comput. Phys.}, 261 (2014) 172-208.
	
	\bibitem{barsara2015}D.S. Balsara and M. Dumbser, Divergence-free MHD on unstructured meshes using high order finite volume schemes based on multidimensional Riemann solvers, \emph{J. Comput. Phys.}, 299 (2015) 687-715.
	
	\bibitem{batten}P. Batten, N. Clarke, C. Lambert and D.M. Causon, On the choice of wavespeeds for the HLLC Riemann solver,
	\emph{SIAM J. Sci. Comput.}, 18 (1997) 1553-1570.
	
	\bibitem{biswasa}B. Biswasa, H. Kumarb and D. Bhoriya, Entropy stable discontinuous Galerkin schemes for the special relativistic hydrodynamics equations, \emph{Comput. Math. Appl.}, 112 (2022) 55-75.

	
	\bibitem{capdeville1}G. Capdeville, A multidimensional HLL-Riemann solver for Euler equations of gas dynamics,
	\emph{Comput. Fluids}, 47 (2011) 122-147.
	
	\bibitem{capdeville2}G. Capdeville, A multidimensional HLL-Riemann solver for non-linear hyperbolic systems,
	\emph{Int. J. Numer. Meth. Fluids}, 67 (2011) 1899-1931.
	
	\bibitem{colella}P. Colella, A direct Eulerian MUSCL scheme for gas dynamics,
	\emph{SIAM J. Sci. Stat. Comput.}, 6 (1985) 104-117.
	
	
	\bibitem{davis}S. F. Davis, Simplified second-order Godunov-type methods,
	\emph{SIAM J. Sci. Stat. Comput.}, 9(3)(1988) 445-473.
	
	\bibitem{dhoriya}D. Bhoriya and H. Kumar, Entropy-stable schemes for relativistic hydrodynamics equations, \emph{Z. Angew. Math. Phys.}, 71 (2020) 1-29.
	
	\bibitem{dolezal}A. Dolezal and  S.S.M. Wong,
	Relativistic hydrodynamics and essentially non-oscillatory shock capturing schemes,
	\emph{J. Comput. Phys.}, 120 (1995) 266-277.
	
	
	\bibitem{Duan-Tang2020RHD}
	J.M. Duan and H.Z. Tang, High-order accurate entropy stable finite difference schemes for one- and two-dimensional special relativistic hydrodynamics, \emph{Adv. Appl. Math. Mech.}, 12 (2020) 1-29.
	
	\bibitem{Duan-Tang2020RMHD}
	J.M. Duan and H.Z. Tang, High-order accurate entropy stable nodal discontinuous Galerkin schemes for the ideal special relativistic magnetohydrodynamics, \emph{J. Comput. Phys.}, 421 (2020) 109731.
	
	\bibitem{Duan-Tang2020RHD2}
	J.M. Duan and H.Z. Tang,
	Entropy stable adaptive moving mesh schemes for 2D and 3D special relativistic hydrodynamics, \emph{J. Comput. Phys.}, 426 (2021) 109949.
	
	\bibitem{duan2022}J.M. Duan and H.Z. Tang, High-order accurate entropy stable adaptive moving mesh finite difference
	schemes for special relativistic (magneto)hydrodynamics, \emph{J. Comput. Phys.}, 456 (2022) 111038.
	
	%
	%
	\bibitem{einfeldt}B. Einfeldt, On Godunov-type methods for gas dynamics, \emph{SIAM J. Numer. Anal.}, 25 (3) (1988) 294-318.
	
	
	%
	%
	
	\bibitem{font}J.A. Font,
	Numerical hydrodynamics and magnetohydrodynamics in general relativity,
	\emph{Living Rev. Relativ.}, 11 (2008) 7.
	
	\bibitem{harten}A. Harten, P.D. Lax and B.van Leer, On upstream differencing and Godunov-type schemes
	for hyperbolic conservation laws, \emph{SIAM Rev.}, 25 (1983) 289-315.
	
	\bibitem{he1}P. He and H.Z. Tang, An adaptive moving mesh method for two-dimensional relativistic
	hydrodynamics, \emph{Commun. Comput. Phys.}, 11 (2012) 114-146.
	
	\bibitem{he2}P. He and H.Z. Tang, An adaptive moving mesh method for two-dimensional
	relativistic magnetohydrodynamics, \emph{Comput. Fluids}, 60 (2012) 1-20.
	
	%
	
	
	
	\bibitem{ling}D. Ling, J.M. Duan and H.Z. Tang, Physical-constraints-preserving Lagrangian
	finite volume schemes for one-and two-dimensional special relativistic hydrodynamics,
	\emph{J. Comput. Phys.}, 396 (2019) 507-543.
	
	\bibitem{lora}F.D. Lora-Clavijo, J.P. Cruz-P\'{e}rez, F.S. Guzm\'{a}n and J.A. Gonz\'{a}lez,
	Exact solution of the 1D Riemann problem in Newtonian and relativistic hydrodynamics,
	\emph{Rev. Mex. F\'{\i}s. E}, 59 (2013) 28-50.
	
	\bibitem{mandal}J.C. Mandal and V. Sharma, A genuinely multidimensional convective pressure flux
	split Riemann solver for Euler equations, \emph{J. Comput. Phys.}, 297 (2015) 669-688.
	
	\bibitem{marti1}J.M. Mart\'{\i} and E. M\"{u}ller,
	The analytical solution of the Riemann problem in relativistic hydrodynamics,
	\emph{J. Fluid Mech.}, 258 (1994) 317-333.
	
	
	\bibitem{marti2}J.M. Mart\'{\i} and E. M\"{u}ller, Numerical hydrodynamics in special relativity,
	\emph{Living Rev. Relativ.}, 6 (2003) 7.
	
	\bibitem{marti2015}J.M. Mart\'{\i} and E. M\"{u}ller, Grid-based methods in relativistic hydrodynamics
	and magnetohydrodynamics, \emph{Living Rev. Comput. Astrophys.}, 1 (2015) 3.
	
	
	\bibitem{may1}M.M. May and R.H. White, Hydrodynamics calculations of general-relativistic collapse,
	\emph{Phys. Rev.}, 141 (1966) 1232-1241.
	
	\bibitem{may2}M.M. May and R.H. White, Stellar dynamics and gravitational collapse,
	\emph{Methods Comput. Phys.}, 7 (1967) 219-258.
	
	
	
	
	\bibitem{pant}
	V. Pant, Global entropy solutions for isentropic relativistic fluid dynamics, \emph{Commun. Part. Diff. Eq.}, 21 (1996) 1609-1641.
	
	
	\bibitem{qin}
	T. Qin, C.-W. Shu and Y. Yang, Bound-preserving discontinuous Galerkin methods for relativistic
	hydrodynamics, \emph{J. Comput. Phys.}, 315 (2016) 323-347.
	
	\bibitem{qu}F. Qu, D. Sun, J. Bai and C. Yan, A genuinely two-dimensional Riemann solver for
	compressible flows in curvilinear coordinates, \emph{J. Comput. Phys.}, 386 (2019) 47-63.
	
	\bibitem{RezzollaDG2011}D. Radice and L. Rezzolla, Discontinuous Galerkin methods for general-relativistic
	hydrodynamics: formulation and application to spherically symmetric spacetimes, \emph{Phys. Rev. D},
	84 (2011) 024010.
	
	\bibitem{roe}P.L. Roe, Approximate Riemann solver, parameter vectors and difference schemes,
	\emph{J. Comput. Phys.}, 43 (1981) 357-372.

	\bibitem{schneider2021}	K.A. Schneider, J.M. Gallardo, D.S. Balsara, B. Nkonga and C. Parés,
	Multidimensional approximate Riemann solvers for hyperbolic nonconservative systems. Applications to shallow water systems,
	\emph{J. Comput. Phys.}, 444 (2021) 110547.

	
	
	
	\bibitem{shu}C.-W. Shu, High order weighted essentially non-oscillatory schemes for convection dominated problems,
	\emph{SIAM Rev.}, 51 (2009) 82-126.
	
\bibitem{tchekhovskoy}A. Tchekhovskoy, J.C. McKinney and R. Narayan, WHAM: a WENO-based general relativistic numerical scheme, I. hydrodynamics, \emph{Mon. Not. R. Astron. Soc.}, 379 (2007) 469-497.
	
	\bibitem{toro}E.F. Toro, \emph{Riemann Solvers and Numerical Methods for Fluid Dynamics: A Practical
		Introdution}, 3rd edition, Springer, 2009.
	
	\bibitem{vanLeer1993}B. van Leer, Progress in multi-dimensional upwind differencing. In: Napolitano M., Sabetta F. (eds) {\em Thirteenth International Conference on Numerical Methods in Fluid Dynamics}, Lecture Notes in Physics, vol 414. Springer, Berlin, Heidelberg, 1993.
	
	
	\bibitem{wendroff}B. Wendroff, A two-dimensional HLLE Riemann solver and associated Godunov-type difference
	scheme for gas dynamics, \emph{Comput. Math. Appl.}, 38 (1999) 175-185.
	
	\bibitem{wilson} J.R. Wilson, Numerical study of fluid flow in a Kerrr space, \emph{Astrophys. J.},
	173 (1972) 431-438.
	
	
	\bibitem{wu2017a} K.L. Wu, Design of provably physical-constraint-preserving methods for general
	relativistic hydrodynamics, \emph{Phys. Rev. D}, 95 (2017) 103001.
	

	\bibitem{wu2014b}
	K.L. Wu and H.Z. Tang, Finite volume local evolution Galerkin method for two-dimensional special relativistic hydrodynamics, \emph{J. Comput. Phys.}, 256 (2014) 277-307.
	
	\bibitem{wu2015}K.L. Wu and H.Z. Tang, High-order accurate physical-constraints-preserving finite
	difference WENO schemes for special relativistic hydrodynamics, \emph{J. Comput. Phys.}, 298 (2015) 539-564.
	
	\bibitem{wu2016} K.L. Wu and H.Z. Tang, A direct Eulerian GRP scheme for spherically symmetric general relativistic hydrodynamics, \emph{SIAM J. Sci. Comput.}, 38 (2016) B458-B489.
	
	\bibitem{wu2017}K.L. Wu and H.Z. Tang, Physical-constraints-preserving central discontinuous Galerkin methods
	for special relativistic hydrodynamics with a general equation of state, \emph{Astrophys. J. Suppl. Ser.},
	228 (2017) 3.
	
	\bibitem{wu2017m3as}
	K.L. Wu and H.Z. Tang, Admissible states and physical-constraints-preserving schemes for
		relativistic magnetohydrodynamic equations, \emph{ Math. Models Methods Appl. Sci.}, 27 (2017) 1871-1928.
	
	
	\bibitem{wu2018zamp}
	K.L. Wu and H.Z. Tang, On physical-constraints-preserving schemes for special relativistic
		magnetohydrodynamics with a general equation of state, \emph{Z. Angew. Math. Phys.}, 69 (2018) 84.
	
	\bibitem{wu2014} K.L. Wu, Z.C. Yang and H.Z. Tang,
	A third-order accurate direct Eulerian GRP scheme for one-dimensional relativistic hydrodynamics,
	\emph{East Asian J. Appl. Math.}, 4 (2014) 95-131.
	
	
	\bibitem{xu}Z.F. Xu, Parameterized maximum principle preserving flux limiters for high order schemes
	solving hyperbolic conservation laws: one-dimensional scalar problem, \emph{Math. Comput.},
	83 (2014) 2213-2238.
	%
	%
	
	\bibitem{yang2011direct}Z.C. Yang, P. He and H.Z. Tang,
	A direct Eulerian GRP scheme for relativistic hydrodynamics: one-dimensional case,
	\emph{J. Comput. Phys.}, 230 (2011) 7964-7987.
	
	\bibitem{yang2012direct}Z.C. Yang and H.Z. Tang, A direct Eulerian GRP scheme for relativistic hydrodynamics:
	two-dimensional case, \emph{J. Comput. Phys.}, 231 (2012) 2116-2139.
	
	\bibitem{Yuan-Tang2020}
	Y.H. Yuan and H.Z. Tang, Two-stage fourth-order accurate time discretizations for 1D and 2D special relativistic hydrodynamics,
	\emph{J. Comput. Math.}, 38 (2020) 746-774.
	
	
	\bibitem{zanna}L.D. Zanna and N. Bucciantini,
	An efficient shock-capturing central-type scheme for multidimensional relativistic flows,
	I: hydrodynamics, \emph{Astron. Astrophys.}, 390 (2002) 1177-1186.
	

\bibitem{zhang}X.X. Zhang and C.-W. Shu, On positivity-preserving high order discontinuous Galerkin schemes for compressible Euler equations on rectangular meshes, \emph{J. Comput. Phys.}, 229 (2010) 8918-8934.
	
	\bibitem{zhao}
	J. Zhao and H.Z. Tang,
	Runge-Kutta discontinuous Galerkin methods with WENO limiter for the special
	relativistic hydrodynamics, \emph{J. Comput. Phys.}, 242 (2013) 138-168.
	
	\bibitem{ZhaoTang-JCP2017}
	J. Zhao and H.Z. Tang, Runge--Kutta discontinuous Galerkin methods for the special
		relativistic magnetohydrodynamics, \emph{J. Comput. Phys.}, 343 (2017) 33-72.
	
	\bibitem{ZhaoTang-CiCP2017}
	J. Zhao and H.Z. Tang, Runge-Kutta central discontinuous Galerkin methods for the special
		relativistic hydrodynamics, \emph{Commun. Comput. Phys.}, 22 (2017) 643-682.
	
\end{thebibliography}
\end{document}